\documentclass{amsart}

\usepackage{amsmath,amssymb,amsthm}
\usepackage{enumerate}
\usepackage{hyperref}
\usepackage{mathtools,pifont}
\usepackage{tikz-cd}
\usepackage{tikz}
\usepackage{xcolor}
\usepackage{mathabx}
\usepackage{bbm}

\usepackage[margin=1in]{geometry}

\newtheorem{theorem}{Theorem}[section]

\newtheorem{lemma}[theorem]{Lemma}
\newtheorem{proposition}[theorem]{Proposition}
\newtheorem{corollary}[theorem]{Corollary}
\newtheorem{fact}[theorem]{Fact}

\theoremstyle{definition}
\newtheorem{definition}[theorem]{Definition}
\newtheorem{notation}[theorem]{Notation}

\theoremstyle{remark}
\newtheorem{remark}[theorem]{Remark}
\newtheorem{example}[theorem]{Example}

\numberwithin{equation}{section}

\newcommand\R{\mathbb{R}}
\newcommand\C{\mathbb{C}}

\newcommand\Z{\mathbb{Z}}
\newcommand\N{\mathbb{N}}

\newcommand{\cA}{\mathcal{A}}
\newcommand{\cB}{\mathcal{B}}

\newcommand{\cT}{\mathcal{T}}

\DeclareMathOperator{\id}{id}
\DeclareMathOperator{\rad}{rad}

\DeclareMathOperator{\Hom}{Hom}

\DeclareMathOperator{\Walk}{Walk}

\DeclareMathOperator{\Digraph}{Digraph}
\DeclareMathOperator{\depth}{depth}

\DeclareMathOperator{\BMO}{BMO}
\DeclareMathOperator{\Sym}{Sym}
\DeclareMathOperator{\vol}{vol}

\DeclareMathOperator{\bool}{Bool}
\DeclareMathOperator{\pred}{pred}

\DeclarePairedDelimiter{\norm}{\lVert}{\rVert}
\DeclarePairedDelimiter{\ip}{\langle}{\rangle}

\newcommand{\assemb}{\bigstar}

\title[General limit theorems for mixtures of free, monotone, and Boolean independence]{General limit theorems for mixtures of free, monotone, and Boolean independence}

\author{David Jekel}
\address{\parbox{\linewidth}{Department of
		Mathematical Sciences, University of Copenhagen \\
		Universitetsparken 5, 2100 Copenhagen \O, Denmark}}
\email{daj@math.ku.dk}
\urladdr{http://davidjekel.com}

\author{Lahcen Oussi}
\address{\parbox{\linewidth}{Department of Mathematics and Cybernetics\\ Wroc{\l}aw University of Economics and Business\\
		ul. Komandorska 118/120, 53-345 Wroc{\l}aw, Poland}}
\email{oussimaths@gmail.com}

\author{Janusz Wysocza\'{n}ski}
\address{\parbox{\linewidth}{Mathematical Institute, University of Wroc{\l}aw \\
		pl. Grunwaldzki 2, 50-384 Wroc{\l}aw, Poland}}
\email{janusz.wysoczanski@math.uni.wroc.pl}
\urladdr{https://www.math.uni.wroc.pl/~jwys/}



\begin{document}

\maketitle

\begin{abstract}We study mixtures of free, monotone, and Boolean independence described by a directed graph $G = (V,E)$ in the context of $\mathcal{T}$-free convolutions from \cite{JekelLiu2020}. We prove general limit theorems for the associated additive convolution operations $\boxplus_G$.  For a sequence of digraphs $G_n = (V_n,E_n)$, we give sufficient conditions for the limit $\widehat{\mu} = \lim_{n \to \infty} \boxplus_{G_n}(\mu_n)$ to exist whenever the Boolean convolution powers $\mu_n^{\uplus |V_n|}$ converge to some $\mu$.  This in particular includes central limit and Poisson limit theorems, as well as limit theorems for each classical domain of attraction.  The hypothesis on the sequence of $G_n$ is that the normalized counts of digraph homomorphisms from rooted trees into $G_n$ converge as $n \to \infty$, and we verify this for several families of examples where the $G_n$'s converge in some sense to a continuum limit, or digraphon.  In particular, we obtain a new limit theorem for multiregular digraphs, as well as recovering several limit theorems in prior work.
\end{abstract}

	\section{Introduction}
	
	\subsection{Motivation}
	
	Non-commutative probability is based on various notions of independence for non-commuting random variables. The non-commuting variables are represented as elements of some unital $*$-algebra $\mathcal{A}$ (often an algebra of operators on a Hilbert space), and the expectation is represented by a state (a positive unital linear functional) $\phi: \mathcal{A} \to \mathbb{C}$.  There are several notions of independence in the non-commutative setting.  The most famous and fruitful is free independence, defined by Voiculescu \cite{Voiculescu1985,Voiculescu1986} (and implicitly by Avitzour \cite{Av1982}).  Boolean independence appeared implicitly in Bo\.{z}ejko's paper \cite{MB1986} and was further developed by  Speicher and Woroudi \cite{SpW1997}.  Muraki \cite{Muraki2000,Muraki2001} invented monotone independence.  One of the main applications of non-commutative independence has been to study the large-$n$ behavior of random matrices, thanks to Voiculescu's work on asymptotic freeness in \cite{Voiculescu1991}.  Matrix models for monotone independence and the related cyclic monotone independence were given in \cite{CHS2018,CSL2024}.  Meanwhile, various permutation-invariant random matrices can produce tensor, free, and Boolean independence \cite{Male2020}.
	
	It was shown by Speicher \cite{Speicher1997}, Ben Ghorbal and Sch{\"u}rmann \cite{BGS2002}, and Muraki \cite{Muraki2003,Muraki2013} that there are only five universal notions of independence coming from an associative binary product operation on non-commutative probability spaces: classical, free, monotone, antimonotone, and Boolean.  However, there are other binary and more generally $n$-ary product operations that do not fall within this framework but nevertheless allow one to develop theories in analogy with classical probability. This includes $n$-ary product operations that provide mixtures of tensor, free, Boolean, monotone independence.  In particular, the third author developed the theory of BM-independence, a mixture of Boolean and monotone independence which appears naturally when replacing the totally ordered set used for monotone independence with a partially ordered set.  A chain in the poset corresponds to a monotone independent family of algebras that are monotone independent, while a set of mutually incomparable elements produces a boolean independent family of algebras.
	BM-independence was proposed originally in \cite[\S 2]{JW20072}, then formulated as a property of BM-extension operators in \cite[\S 2]{JW2008}, and finally established in \cite[\S 2]{JW2010} as a notion of noncommutative independence, i.e.\ a rule for computing joint moments.   A similar mixture of Boolean and free independence, called BF-independence, was introduced by Kula and Wysocza\'{n}ski \cite{AKJW2013}.  Analogously, M{\l}otkowski \cite{Mlo2004} studied a mixture of classical and free independence under the name of $\Lambda$- independence, which was then developed by Speicher and Wysocza\'{n}ski under the name $\varepsilon$-independence \cite{SpWys2016}.  Recently, Arizmendi, Mendoza, and Vazquez-Becerra \cite{OASRMJVB2025}, introduced the notion of BMT independence (through a directed graph) which is mixture of Boolean, monotone and tensor independences, and provided the corresponding Central and Poisson-Type Limit Theorems.
	
	Our work focuses on mixtures of free, Boolean, and monotone independence that are described by directed graphs in
	the context of $\mathcal{T}$-free independence, as in \cite[\S 3.2, 5.5]{JekelLiu2020}.  This setting extends both BM independence from \cite{JW2007,JW2010} (see Proposition \ref{prop: equivalent definitions of BM independence}) and BF-independence from \cite{AKJW2013} (see Remark \ref{rem: BF independence}).  We remark that when a digraph has no bidirectional edges, so that the pairwise relations between the algebras are Boolean and (anti)monotone independence, one might expect that our construction agrees with BMT-independence \cite{OASRMJVB2025}; while this is true for partial orders, this fails for general digraphs---for instance if $G$ is a directed $3$-cycle (see Remark \ref{rem: BMF versus BMT}).  In particular, there is not a unique or canonical mixture of boolean and monotone independence associated to a given digraph specifying the pairwise relations.
	
	Our main goal is to define $G$-independence for a directed graph $G$ and  to give a unified approach to limit theorems for additive convolution in the setting of $G$-independence.  Here the $G$-free additive convolution of a family of probability measures is the distribution of a sum of $G$-independent variables having the specified individual distributions, and our goal is, given a sequence of digraphs $G_n$ and probability measures $\mu_n$, to understand the limiting behavior of the convolution $\boxplus_{G_n}(\mu_n)$ of several copies of $\mu_n$.  We generalize existing limit theorems in two aspects:
	\begin{itemize}
		\item \textbf{More general hypotheses on the digraphs.}  We consider an arbitrary sequence of digraphs $G_n$ with number of vertices tending to infinity, requiring only the convergence of the normalized number of homomorphisms from trees into $G_n$.  Examples of such graphs include both the discretizations of cones from \cite{JW2010,LOJW1,LOJW2,LOJW3} and the iterated compositions of a fixed graph from \cite{JekelLiu2020}.  We also describe new classes of examples such as multi-regular and sparse graphs in \S \ref{sec: examples}.
		\item \textbf{More general hypotheses on the measures.}  Under the above assumptions on a sequence of graphs $G_n$, we show that the $G_n$-free convolutions $\boxplus_{G_n}(\mu_n)$ converge in the weak-$*$ topology for any sequence of measures $\mu_n$ such that the Boolean convolution powers $\mu_n^{\uplus |V_n|}$ converge in the weak-$*$ topology.  Thus, in particular, we obtain a limit theorem for each classical domain of attraction, in the spirit of Bercovici--Pata \cite{BP1999}. We extend the general limit theorem to the non-compactly supported case using tools from \cite{JDW2021}.
	\end{itemize}
	In particular, while papers on new types of non-commutative independence have often unnecessarily restricted themselves to proving only a central limit theorem or a Poisson limit theorem, we attack the general case directly.   On the other hand, the limit theorems for $\mathcal{T}$-independence and $G$-independence in \cite{JekelLiu2020,JDW2021} handled general sequences of measures but with very restrictive assumptions on the graphs, focusing only on iterated compositions of a fixed graph.
	
	\subsection{Results}
	
	To state our results more precisely, first recall that a \emph{digraph} is a pair $G=(V,E)$ where $V$ is the vertex set and the directed edge set $E \subseteq V \times V$ does not intersect the diagonal; in other words, $E$ viewed as a relation on $V$ is irreflexive.  We write $v \rightsquigarrow w$ and $w \leftsquigarrow v$ when $(v,w) \in E$.  Several other types of combinatorial objects important to our paper are (equivalent to) subclasses of digraphs (see \S \ref{subsec: digraphs} for details):
	\begin{enumerate}
		\item A \emph{graph} is a digraph $(V,E)$ such that $(v,w) \in E$ if and only if $(w,v) \in E$, or equivalently, $E$ is a symmetric relation on $V$.
		\item A \emph{partially ordered set} or \emph{poset} can be expressed as a digraph by using the edge set to represent the corresponding \emph{strict} partial order.
		\item A \emph{rooted tree} can be represented as a digraph $G = (V,E)$ by representing each edge as a directed edge oriented away from the root vertex.  Such a digraph is called an \emph{out-tree} (or \emph{arborescence}).
	\end{enumerate}
	If $G_1 = (V_1,E_1)$ and $G_2 = (V_2,E_2)$ are digraphs, then a \emph{digraph homomorphism} from $G_1$ to $G_2$ is a map $\phi: V_1 \to V_2$ such that if $v \rightsquigarrow w$ in $G_1$, then $\phi(v) \rightsquigarrow \phi(w)$ in $G_2$.  We denote the set of homomorphisms by $\Hom(G_1,G_2)$.
	
	For each finite digraph $G=(V, E)$, we will define a convolution operation $\boxplus_G: \mathcal{P}(\R)^{\times V} \to \mathcal{P}(\R)$, called $G$-free convolution, where $\mathcal{P}(\R)$ denotes the space of Borel probability measures on $\R$ and $\mathcal{P}(\R)^{\times V}$ is the Cartesian product over the index set $V$.  If the input measures $\mu_v$ have compact support, they can be viewed as spectral distributions of bounded operators $X_v$, and then the convolution operation is defined by creating $G$-independent copies of $X_v$ through the explicit Hilbert space construction described in \cite{JekelLiu2020}, and in \S \ref{subsec: product space} below.  The convolution can also be described by calculating its moments (see Theorem \ref{thm:momentformula} and Lemma \ref{lem: moments of sum}).  For the case of measures $\mu_v$ with unbounded support, the convolution can be described complex-analytically, or obtained by writing the measures $\mu_v$ as weak-$*$ limits of compactly supported probability measures (see Proposition \ref{prop: K transform equation}, Definition \ref{def: convolution general}, and the surrounding discussion).  The $G$-free additive convolution of $(\mu_v)_{v \in V}$ is denoted by $\boxplus_G((\mu_v)_{v \in V})$, and $\boxplus_G(\mu)$ denotes the  $G$-free additive convolution of copies of the same measure $\mu$ indexed by $V$.
	
	Our main result is the following.  Note that in this work, convergence of measures always refers to weak-$*$ or vague convergence, or equivalently convergence in the L{\'e}vy distance.
	
	\begin{theorem} \label{thm:limit}
		Let $G_n = (V_n,E_n)$ be a sequence of finite digraphs such that $\lim_{n \to \infty} |V_n| = \infty$.  Suppose that for every finite out-tree $G' = (V',E')$, the limit
		\[
		\beta_{G'} := \lim_{n \to \infty} \frac{1}{|V_n|^{|V'|}} |\Hom(G',G_n)| \text{ exists.}
		\]
		Let $\mu \in \mathcal{P}(\R)$ and let $(\mu_n)_{n \in \N}$ be a sequence of probability measures on $\R$ such that $\lim_{n \to \infty} \mu_n^{\uplus |V_n|} = \mu$, where $\mu^{\uplus t}=\mu_t$ denotes the $t$-transformation or $t$-th Boolean convolution power of $\mu$ for $t \in (0,\infty)$ (see \cite{MBJW1998, MBJW2001}).  Then
		\[
		\lim_{n \to \infty} \boxplus_{G_n}(\mu_n) \text{ exists,}
		\]
		Furthermore, $\lim_{n \to \infty} \boxplus_{G_n}(\mu_n)$ depends only upon $\mu$ and the coefficients $\beta_{G'}$ for finite out-trees $G'$.
	\end{theorem}
	
	The proof of this theorem (see \S \ref{sec: limit theorem}) proceeds in two stages.  We first show that it holds for compactly supported measures using moment computations (see \S \ref{subsec: limit theorem compact support}). Then we extend it to arbitrary measures using the results of \cite{JDW2021} (see \S \ref{subsec: limit general}).
	
	Next, we turn to applications of the main theorem, exhibiting several classes of examples where the limits exist.  Many of these  are obtained by viewing the digraphs $G_n$ as discretizations of a digraphon (the directed version of a graphon).  A digraphon is a measurable digraph $(\Omega,\rho,\mathcal{E})$, that is, a complete probability measure space $(\Omega,\rho)$ representing some set of vertices, and a measurable subset $\mathcal{E} \subseteq \Omega \times \Omega$ representing the set of directed edges.  For instance, $\Omega$ could be $[0,1]$ and $\mathcal{E}$ could be a subset of $[0,1]^{2}$ defined by some inequalities.  One may approximate $(\Omega,\mathcal{E})$ by discreted digraphs $G_n=(V_n, E_n)$ by partitioning $\Omega$ into measurable subsets $(A_{n,v})_{v \in V_n}$ of measure $1/|V_n|$ and choosing subset $E_n \subseteq V_n \times V_n$ such that $\tilde{\mathcal{E}}_n = \bigcup_{(v,w) \in E_n} A_v \times A_w$ converges to $\mathcal{E}$ in measure as $n \to \infty$.  In this case, the $G_n$'s will satisfy the hypotheses of Theorem \ref{thm:limit}.  More precisely, we have the following result (see \S \ref{sec: examples}).
	
	\begin{proposition} \label{prop: continuum limit}
		Let $(\Omega,\rho)$ be a complete probability measure space and let $\mathcal{E} \subseteq \Omega \times \Omega$ be measurable.  For each $n \in \N$, let $G_n = (V_n,E_n)$ be a finite digraph.  Let $(A_{n,v})_{v \in V_n}$ be a measurable partition of $\Omega$ into sets of measure $1/|V_n|$, and let $\tilde{\mathcal{E}}_n = \bigcup_{(v,w) \in E_n} A_{n,v} \times A_{n,w}$.  Suppose that $(\rho\times \rho)(\tilde{\mathcal{E}}_n \Delta \mathcal{E}) \to 0$.  Then for every digraph $G' = (V',E')$, we have
		\[
		\lim_{n \to \infty} \frac{\Hom(G',G_n)}{|V_n|^{|V'|}} = \rho^{\times V'}(\Hom(G',(\Omega,\mathcal{E}))),
		\]
		where $\Hom(G',(\Omega,\mathcal{E})$ is viewed as a subset of the Cartesian product $\Omega^{\times V'}$ and $\rho^{\times V'}$ denotes the product measure.  In particular, if $\mu \in \mathcal{P}(\R)$ and $(\mu_n)_{n \in \N}$ is a sequence of probability measures on $\R$ such that $\lim_{n \to \infty} \mu_n^{\uplus |V_n|} = \mu$, then $\lim_{n \to \infty} \boxplus_{G_n}(\mu_n)$ exists.
	\end{proposition}

	Using Proposition \ref{prop: continuum limit} and similar techniques, we show how Theorem \ref{thm:limit} applies to several families of examples:
	\begin{enumerate}
		\item BM-independences described by positive symmetric cones as in \cite{AKJW2010,LOJW1} (see \S \ref{subsec: BM limit theorem}).
		\item Iterated compositions of the same digraph in the sense of \cite{JekelLiu2020} (see \S \ref{subsec: iterated composition}).
		\item Regular digraphs and, more generally, multi-regular digraphs (see \S \ref{subsec: multiregular}).
		\item Sparse graphs (see \S \ref{subsec: sparse graphs}).
	\end{enumerate}
	
	Given that, in the situation of Proposition \ref{prop: continuum limit}, the coefficients $\beta_{G'}$ can often be described as measures of the set of homomorphisms into some measurable digraph, it is natural to model the limiting measures using operators on a continuum analog of the $G$-product Hilbert space.  This leads to a Fock space associated to $(\Omega,\mathcal{E})$, which generalizes the BM Fock spaces in \cite{AKJW2010,JW2010} (see Example \ref{ex: BM Fock space}).  The construction also overlaps with the Fock spaces in \cite{JekelLiu2020} (see Example \ref{ex: Fock for iterated composition}), which in turn encompasses Fock spaces in the free \cite{Voiculescu1985}, Boolean \cite{BGS2004}, and monotone cases \cite{Lu1997,Muraki1997} (for further background see also the unified explanation of these three cases in \cite{Franz2003}).
	
	\subsection{Remarks and Questions}
	
	Very recent work of C{\'e}bron, Oliveira Santos, and Youssef \cite{COSY2024} gives somewhat analogous results for the setting of $\varepsilon$-independence, which is a mixture of tensor and free independence.  Their main result is a central limit theorem for $\varepsilon$-additive convolutions associated to a sequence of graphs $G_n$ such that the normalized homomorphism counts from any fixed graph $G'$ into $G_n$ converge to $|\Hom(G',G)|$ for some graphon $G$ on $\Omega = [0,1]$.  The central limit distribution also has a graphon Fock-space model \cite[\S 2]{COSY2024} analogous to that in \S \ref{sec: Fock space} of our work.  Note that, unlike the situation in our Theorem \ref{thm:limit}, one cannot restrict the test graph $G'$ to be a tree or forest because the moment formulas for $\varepsilon$-independence involve partitions with crossings.  Moreover, in the setting of $\varepsilon$-independence, it is not clear how to prove limit theorems for measures with unbounded support since we lack a good analytical theory of additive convolution.
	
	Another natural question is how to define a suitable mixed independence that allows the pairwise relations to be any combination of tensor, free, Boolean, and monotone independence.  Is it possible to define such a relation that extends either the BMT independence of \cite{OASRMJVB2025} or the $G$-independence in this paper?  Furthermore, as mentioned earlier, free, Boolean, and monotone independence can all arise in various ways from random matrix models, and so can $\varepsilon$-independence, so do the mixtures of independence studied in this paper arise in random matrix theory as well?  Finally, we propose establishing a multiplicative analog of Theorem \ref{thm:limit} as a problem for future work.
	
	\subsection{Organization}
	
	In \S \ref{sec: prelim}, we present elementary background on non-commutative probability spaces, digraphs, and non-crossing partitions. In \S \ref{sec: independence}, we give a self-contained explanation of the Hilbert space construction and moment formulas for $G$-independence.  In \S \ref{sec: limit theorem}, we prove Theorem \ref{thm:limit}.  In \S \ref{sec: examples}, we prove Proposition \ref{prop: continuum limit} and study several families of examples for which limit distributions exist. In \S \ref{sec: Fock space}, we describe the Fock space construction associated to digraphons, and more generally measure spaces equipped with a two-variable weight function.
	
	
	\section{Preliminaries} \label{sec: prelim}
	
	\subsection{Non-commutative probability spaces}
	
	We assume familiarity with basic $\mathrm{C}^*$-algebra theory.  In particular, recall that if $\mathcal{A}$ is a unital $\mathrm{C}^*$-algebra, then a state $\phi$ is a linear functional $A \to \C$ satisfying $\phi(1) = 1$ and $\phi(a^*a) \geq 0$.  For each state $\phi$, one can define a semi-inner product $\ip{a,b}_\phi = \phi(a^*b)$.  Quotienting by the kernel $\{a: \phi(a^*a)\}$ and taking the completion produces the \emph{GNS Hilbert space} $H_\phi$; for each $a \in \mathcal{A}$, the corresponding element of $H_\phi$ is denoted by $\widehat{a}$.  There is also a $*$-homomorphism $\pi_\phi: \mathcal{A} \to B(H_\phi)$ satisfying that $\pi_\phi(a) \widehat{b} = \widehat{ab}$, and $\pi_\phi$ is called the \emph{GNS respresentation} associated to $\phi$.
	
	A \emph{$\mathrm{C}^*$-probability space} refers to a unital $\mathrm{C}^*$-algebra $\mathcal{A}$ with a state $\phi$, such that the GNS representation $\pi_\phi: \mathcal{A} \to B(H_\phi)$ induced by $\phi$ is faithful,\footnote{Note that various authors either require the stronger hypothesis that the \emph{state} is faithful ($\phi(a^*a) = 0$ implies $a = 0$) or do not assume any faithfulness condition at all.} or equivalently $\phi(bac) = 0$ for all $b, c \in \mathcal{A}$ implies that $a = 0$.
	
	If the GNS respresentation is faithful, then $\mathcal{A}$ is isomorphic to the image $\pi_\phi(\mathcal{A})$, which is a $\mathrm{C}^*$-subalgebra of $B(H_\phi)$, and the state $\phi$ can be recovered as $\phi(a) = \ip{\xi, \pi_\phi(a)\xi}$ where $\xi = \widehat{1} \in H_\phi$.  Thus, faithfulness of the GNS representation is a natural non-degeneracy condition which guarantees that the state eventually captures all the information about the $\mathrm{C}^*$-algebra.  In particular, it allows us to define $G$-product of $\mathrm{C}^*$-probability spaces by means of defining the appropriate Hilbert space on which it acts (see Definition \ref{def: G free product probability space}).
	
	\subsection{Digraphs} \label{subsec: digraphs}
	
	\textbf{Digraphs and graphs:} Recall that a \emph{digraph} is a pair $(V,E)$ where $V$ is a set of \emph{vertices} set and $E \subseteq V \times V$ is a set of \emph{directed edges} that does not intersect the diagonal.  We write $v \rightsquigarrow w$ and $w \leftsquigarrow v$ when $(v,w) \in E$, and we write $E^{\dagger} = \{(v,w): (w,v) \in E\}$.  A \emph{walk} in a directed graph is a sequence of vertices $v_0$, \dots, $v_m$ such that $v_0 \rightsquigarrow v_1 \rightsquigarrow \dots \rightsquigarrow v_m$, where $m \geq 0$, and walk is said to be a \emph{path} if no vertices are repeated.  
	
	A (simple undirected) \emph{graph} can be defined as a digraph such $E$ is \emph{symmetric}, or $(v,w) \in E$ if and only if $(w,v) \in E$.  Undirected graphs are sometimes equivalently defined by specifying $E$ as a set of unordered pairs.
	
	\textbf{Rooted trees and out-trees:}  A \emph{tree} is a graph $G$ such that any two vertices are connected by a \emph{unique} path.  A \emph{rooted tree} is a tree $G$ together with a chosen root vertex $r$.  A \emph{out-tree} is a digraph $G = (V,E)$ such that the following holds:  there exists a vertex $r$ such that for every $v \in V$, there is a unique (directed) walk from $r$ to $v$.  An \emph{out-forest} is any disjoint union of out-trees.
	
	There is a well-known bijection between rooted undirected trees and out-trees on a given vertex set $V$, described as follows. Consider a rooted undirected tree $(G,r)$.  For each vertex $v$, let $d(v,r)$ be the distance of $v$ from the root vertex (that is, the length of the unique path from $v$ to $r$).  Because of the uniqueness of such a path, one sees that if $v$ and $w$ are adjacent vertices, then $|d(v,r) - d(w,r)| = 1$.  Let $G' = (V,E')$ be the digraph given by $E' = \{ (v,w) \in E: d(v,r) + 1 = d(w,r) \}$.  One can check that there is a unique walk from $r$ to any given vertex $v$, and hence $G'$ is an out-tree.
	
	Conversely, given an out-tree $G = (V,E')$, there is some vertex $r$ such that for every vertex $v$ there is a unique walk from $r$ to $v$.  We claim first that $r$ is unique.  Note that there cannot be any edges going into $r$; otherwise, if $(v,r) \in E$, then by joining this edge to a walk from $r$ to $v$, we would obtain another walk from $r$ to $r$.  Therefore, the vertex $r$ can be characterized as the unique vertex of $G$ that does not have any ingoing edges (any other vertex $v$ has an ingoing edge coming from a walk from $r$ to $v$).  Let $G$ be the undirected graph $(V,E' \cup (E')^\dagger)$.  We then claim that $G$ is a undirected tree.  Indeed, the out-tree $(V,E')$ must satisfy $|E'| = |V| - 1$ since each vertex besides $r$ must have a unique ingoing edge.  We deduce that $G$ has exactly $|V| - 1$ undirected edges; it is also connected, and hence it must be a tree.  It is straightforward to check that the maps from rooted undirected trees to out-trees and vice versa are inverse to each other.
	
	\textbf{Posets:}  A \emph{partially ordered set} or \emph{poset} is a set $V$ together with a partial order $\preceq$, that is, a relation that is reflexive, antisymmetric, transitive.  A \emph{strict poset} is a set $V$ together with a \emph{strict partial order}, that is, a relation that is irreflexive, antisymmetric, and transitive.  There is a well-known bijection between (non-strict) posets and strict posets on a given set $V$.  Given a poset $(V,\preceq)$, the corresponding strict poset is defined using the relation $\prec$ given by $(\prec) = (\preceq) \setminus (=)$ as subsets of $V \times V$.  Conversely, given a strict poset, the non-strict poset is given by $(\preceq) = (\prec) \cup (=)$.
	
	When we handle posets in the remainder of the paper, we will implicitly use both of these two equivalent representations; thus, $\prec$ and $\preceq$ will refer to the strict and non-strict versions of ``the same'' partial order.  We also consider posets as a subclass of digraphs, where the edge set $E$ is given by the strict partial order $(\prec) \subseteq V \times V$.

	\subsection{Non-crossing partitions}
	
	Non-crossing partitions are a combinatorial tool that has been used to describe moments in non-commutative probability since the work of Speicher \cite{Speicher1994,Speicher1998}.  We recall some relevant definitions and facts here.
	
	\begin{definition}[Partitions]
		We use the notation $[k] = \{1,\dots,k\}$.  A \emph{partition} of $[k]$ is a collection $\pi$ of non-empty subsets of $[k]$, called \emph{blocks}, such that $[k] = \bigsqcup_{B \in \pi} B$.  The set of partitions of $[k]$ is denoted $\mathcal{P}_k$.
	\end{definition}
	
	\begin{definition}[Non-crossing partitions]
		For a partition $\pi$ of $[k]$, a \emph{crossing} is a sequence of indices $i < i' < j < j'$ such that $i$ and $j$ are in some block $B$ and $i'$ and $j'$ are in some block $B' \neq B$.  We say that $\pi$ is \emph{non-crossing} if it has no crossings.  We denote by $\mathcal{NC}_k$ the set of non-crossing partitions of $[k]$.
	\end{definition}
	
	\begin{remark}
		Visually, a partition is non-crossing if, after arranging points labeled $1$, \dots, $k$ on a horizontal line, it is possible to connect all the points in the same blocks by curves above the horizontal line such that the curves associated to points in two different blocks never cross each other.  See Figure \ref{fig: NC partition} for an example.
	\end{remark}
	
	\begin{figure}
		
		\begin{center}
			
			\begin{tikzpicture}[scale=0.9]
				
				\node[circle,fill] (1) at (1,0) [label=below:$1$] {};
				\node[circle,fill] (2) at (2,0) [label=below:$2$] {};
				\node[circle,fill] (3) at (3,0) [label=below:$3$] {};
				\node[circle,fill] (4) at (4,0) [label=below:$4$] {};
				\node[circle,fill] (5) at (5,0) [label=below:$5$] {};
				\node[circle,fill] (6) at (6,0) [label=below:$6$] {};
				\node[circle,fill] (7) at (7,0) [label=below:$7$] {};
				\node[circle,fill] (8) at (8,0) [label=below:$8$] {};
				\node[circle,fill] (9) at (9,0) [label=below:$9$] {};
				\node[circle,fill] (10) at (10,0) [label=below:$10$] {};
				\node[circle,fill] (11) at (11,0) [label=below:$11$] {};
				\node[circle,fill] (12) at (12,0) [label=below:$12$] {};
				\node[circle,fill] (13) at (13,0) [label=below:$13$] {};
				\node[circle,fill] (14) at (14,0) [label=below:$14$] {};
				\node[circle,fill] (15) at (15,0) [label=below:$15$] {};
				\node[circle,fill] (16) at (16,0) [label=below:$16$] {};
				
				\draw (1) -- (1,2) -- (5,2) -- (5);
				\draw (2) -- (2,1) -- (4,1) -- (4);
				\draw (6) -- (6,3) -- (16,3) -- (16);
				\draw (9) -- (9,3);
				\draw (15) -- (15,3);
				\draw (7) -- (7,1) -- (8,1) -- (8);
				\draw (10) -- (10,2) -- (14,2) -- (14);
				\draw (13) -- (13,2);
				\draw (11) -- (11,1) -- (12,1) -- (12);
				
			\end{tikzpicture}
			
		\end{center}
		
		\caption{A non-crossing partition $\pi$ of $[16]$ with blocks $\{1,5\}$, $\{2,4\}$, $\{3\}$, $\{6,9,15,16\}$, $\{7,8\}$, $\{10,13,14\}$, and $\{11,12\}$.} \label{fig: NC partition}
		
	\end{figure}
	
	\begin{definition}[Nested blocks]
		If $\pi$ is a partition of $[k]$ and $B', B \in \pi$, we say that $B'$ is \emph{nested inside} $B$ if there exist $i, j \in B$ such that $i < j$ and $B' \subseteq \{i+1,\dots,j-1\} \subseteq [k] \setminus B$; in other words, there are no intervening indices of $B$ between $i$ and $j$, and $B'$ lies entirely between $i$ and $j$.  In this case, we write $B \prec B'$.
	\end{definition}
	
	\begin{definition}[Separated blocks]
		If $\pi$ is a partition of $[k]$ and $B, B' \in \pi$, we say that $B$ and $B'$ are \emph{separated} if there exists $j \in [k]$ such that either $B \subseteq \{1,\dots,j\}$ and $B' \subseteq \{j+1,\dots,k\}$ or $B' \subseteq \{1,\dots,j\}$ and $B \subseteq \{j+1,\dots,k\}$.
	\end{definition}
	
	\begin{lemma} \label{lem:trichotomy}
		Let $\pi$ be a non-crossing partition of $[k]$ and $B$, $B'$  distinct blocks of $\pi$. Then either $B'$ is nested inside $B$, $B$ is nested inside $B'$, or $B$ and $B'$ are separated; and these cases are mutually exclusive.
	\end{lemma}
	
	\begin{proof}
		Suppose that there exists $i, j \in B$ and $i' \in B'$ such that $i < i' < j$.  Without loss of generality, assume that $i$ is the largest index in $B$ to the left of $i'$, and $j$ is the smallest index in $B$ to the right of $i'$.  Then $\{i+1,\dots,j-1\} \subseteq [k] \setminus B$.  If $B'$ had some element $j'$ that was not contained in $\{i+1,\dots,j-1\}$, then $i < i' < j < j'$ would be a crossing.  Hence, $B' \subseteq \{i+1,\dots,j-1\}$, so $B'$ is nested inside $B$.
		
		Similarly, if there exists $i', j' \in B'$ and $j \in B$ such that $i' < j < j'$, then $B$ is nested inside $B'$.
		
		If neither of the two cases above holds, then either all the indices of $B$ are less than those of $B'$, or vice versa, hence $B$ and $B'$ are separated.  It is a straightforward exercise that $B \prec B'$, $B' \prec B$, and $B$ and $B'$ separated are mutually exclusive cases.
	\end{proof}
	
	\begin{corollary}
		For $\pi \in \mathcal{NC}_k$, the nesting relation $\prec$ is a strict partial order on $\pi$.
	\end{corollary}
	
	\begin{proof}
		The previous lemma shows that $B \prec B'$ and $B' \prec B$ are mutually exclusive.  It is immediate that $B \prec B'$ implies $B \neq B'$, and straightforward to check that $\prec$ is transitive.
	\end{proof}
	
	Thus, $(\pi,\prec)$ is a poset.  Recall that for a poset, the \emph{covering relation} is the relation $R$ given by $xRy$ if $x < y$ and there is no $z$ with $x < z < y$.  In this case, we call $x$ a \emph{predecessor} of $y$.
	
	\begin{definition} \label{def: nesting forest}
		For $\pi \in \mathcal{NC}_k$, let $\mathrm{F}(\pi)$ be the digraph with vertex set $\pi$ and edges given by the covering relation of $(\pi,\prec)$.  That is, $B \rightsquigarrow B'$ if $B \prec B'$ and there is no $B''$ with $B \prec B'' \prec B'$.
	\end{definition}
	
	\begin{figure}
		\begin{center}
			\begin{tikzpicture}
				\node[circle,fill] (B1) at (0,0) [label=left: {$\{1,5\}$}] {};
				\node[circle,fill] (B2) at (0,-1) [label=left: {$\{2,4\}$}]{};
				\node[circle,fill] (B3) at (0,-2) [label=left: {$\{3\}$}]{};
				\node[circle,fill] (B4) at (4,0) [label=right: {$\{6,9,15,16\}$}]{};
				\node[circle,fill] (B5) at (3,-1) [label=left: {$\{7,8\}$}] {};
				\node[circle,fill] (B6) at (5,-1) [label=right: {$\{10,13,14\}$}] {};
				\node[circle,fill] (B7) at (5,-2) [label=right: {$\{11,12\}$}] {};
				
				\draw[->] (B1) to (B2);
				\draw[->] (B2) to (B3);
				\draw[->] (B4) to (B5);
				\draw[->] (B4) to (B6);
				\draw[->] (B6) to (B7);
			\end{tikzpicture}
		\end{center}
		
		\caption{The nesting forest of the partition $\pi$ from Figure \ref{fig: NC partition}.}
		\label{fig: nesting forest}
	\end{figure}
	
	\begin{lemma}
		Let $\pi \in \mathcal{NC}_k$ and $B \in \pi$. Then either $B$ is minimal with respect to $\prec$ or there is a unique $B'$ such that $B' \rightsquigarrow B$ in $\mathrm{F}(\pi)$.  In particular, $\mathrm{F}(\pi)$ is an out-forest with edges oriented away from the root of each component.
	\end{lemma}
	
	\begin{proof}
		Suppose that $B$ is not minimal.  Then $\{B': B' \prec B\}$ is a finite poset and hence has a maximal element, so there exists some $B'$ with $B' \rightsquigarrow B$ in $\mathrm{F}(\pi)$.  To show that this $B'$ is unique, consider some other $B''$ with $B'' \prec B$.  By Lemma \ref{lem:trichotomy}, either $B' \prec B''$ or $B'' \prec B'$ or $B'$ and $B''$ are separated.  The case $B' \prec B'' \prec B$ cannot happen because we assumed that $B' \rightsquigarrow B$.  If $B'' \prec B'$, then we cannot have $B' \rightsquigarrow B$.  Finally, if $B'$ and $B''$ are separated, then there exists a partition of $[k]$ into two intervals $I'$ and $I''$ with $B' \subseteq I'$ and $B'' \subseteq I''$.  Since $B' \prec B$ and $I'$ is an interval, we have $B \subseteq I'$.  Hence, $B$ and $B''$ are separated, which contradicts $B'' \prec B$, so the case where $B'$ and $B''$ are separated also cannot happen.  This completes the proof of the first claim.
		
		To show that $\mathrm{F}(\pi)$ is a forest, one uses the first claim to construct a backward walk from any given $B$ to some $B'$ which is minimal in $\prec$ (that is, a directed walk from $B'$ to $B$) and check that this walk is unique.
	\end{proof}
	
	\begin{definition}[Nesting forest]
		For $\pi \in \mathcal{NC}_k$, we call $\mathrm{F}(\pi)$ the \emph{nesting (out-)forest} of $\pi$.  See Figure \ref{fig: nesting forest} for an example.
	\end{definition}
	
	\begin{notation}
		We denote by $\depth(B)$ the depth of a block $B$ in the out-forest $\mathrm{F}(\pi)$.  If $B$ is minimal, then $\depth(B) = 1$.  If $\depth(B) > 1$, then we denote by $\pred(B)$ the predecessor of $B$ (which is unique because $\mathrm{F}(\pi)$ is a out-forest).
	\end{notation}
	
	\section{$G$-independence} \label{sec: independence}
	
	Often non-commutative probability papers define  notions of independence of $*$-subalgebras $\mathcal{A}_1, \ldots, \mathcal{A}_n$ first through a condition on moments, and then use a Hilbert space model to show that, given any family of $\mathrm{C}^*$-probability spaces $(\mathcal{A}_1,\phi_1)$, \dots, $(\mathcal{A}_n,\phi_n)$, there exists some $(\mathcal{A},\phi)$ containing independent copies of all $(\mathcal{A}_j,\phi_j)$.  However, since the moment formula for $G$-independence (and even for mixtures of free and Boolean independences \cite{AKJW2010}) is much more complicated to state, we will begin with the Hilbert space model, explain how the moment formula arises naturally from the Hilbert space structure, and use this for our definition of independence.
	
	The digraph construction described here includes the the BM-product Hilbert space in \cite[\S 2]{JW2007}, \cite[\S 3.2]{JW2010} and \cite[\S 3.1]{LOJW3} as the special case when $G$ is given by a strict partial order (see also Proposition \ref{prop: equivalent definitions of BM independence}).  It also includes the BF-product Hilbert space from \cite[\S 2.2]{AKJW2010} (see Remark \ref{rem: BF independence}). The digraph construction itself is a special case of the more general tree construction of \cite{JekelLiu2020}, corresponding to the case when the tree arises as the set of walks on the digraph.  However, we want to present a self-contained explanation of the digraph case by itself, without the general tree framework or the operator-value probability setup, in order to reduce the number of prerequisites and simplify the intuition for the definition.  We also give more details for the proof of the moment formula than \cite{JekelLiu2020}.
	
	\subsection{Digraph products of pointed Hilbert spaces} \label{subsec: product space}
	
	\begin{definition}
		A \emph{pointed Hilbert space} is a pair $(\mathcal{H},\xi)$ where $\mathcal{H}$ is a Hilbert space and $\xi \in \mathcal{H}$ is a unit vector.  If $(\mathcal{H},\xi)$ is a pointed Hilbert space, we denote by $\mathcal{H}^{\circ}$ the orthogonal complement of $\C \xi$ in $\mathcal{H}$.
	\end{definition}
	
	\begin{notation}
		For a digraph $G = (V,E)$ by $E_m^\dagger$ we  will denote the set of reversed directed walks of length $m$:
		\[
		E_m^{\dagger} = \{(v_0,\dots,v_m): v_0 \leftsquigarrow v_1 \leftsquigarrow v_2 \leftsquigarrow \dots \leftsquigarrow v_m\}.
		\]
		Note that $E_0^\dagger = V$ and $E_1$ is the reversed edge set $E^\dagger$.
	\end{notation}
	
	\begin{definition}[$G$-product of pointed Hilbert spaces] \label{def:productHilbertspace}
		Let $G = (V,E)$ be a digraph, and let $(\mathcal{H}_v,\xi_v)_{v \in V}$ be a collection of pointed Hilbert spaces indexed by $V$.  We define $\assemb_G[(\mathcal{H}_v,\xi_v)_{v \in V}]$ as the pointed Hilbert space $(\mathcal{H},\xi)$ given by
		\begin{equation} \label{eq: G free product}
			\mathcal{H} = \C \xi \oplus \bigoplus_{m \geq 0} \bigoplus_{(v_0,\dots,v_m) \in E_m^\dagger} \mathcal{H}_{v_0}^\circ \otimes \dots \otimes \mathcal{H}_{v_m}^\circ.
		\end{equation}
	\end{definition}
	
	Here we can think of $\mathcal{H}_v$ as sitting inside $\mathcal{H}$ by identifying $\xi_v$ with $\xi$.
	
	\begin{definition} \label{def:inclusionmaps}
		Continuing with the notation of the previous definition, we define for each $v \in V$ a $*$-homomorphism $\iota_v: B(\mathcal{H}_v) \to B(\mathcal{H})$ as follows.  Let
		\[
		\mathcal{H}_{\rightsquigarrow v} = \C \xi \oplus \bigoplus_{m \geq 0} \bigoplus_{\substack{(v_0,\dots,v_m) \in E_m^\dagger v_0 \rightsquigarrow v}} \mathcal{H}_{v_0}^\circ \otimes \dots \otimes \mathcal{H}_{v_m}
		\]
		and
		\[
		\mathcal{H}_{\perp v} = \bigoplus_{m \geq 0} \bigoplus_{\substack{(v_0,\dots,v_m) \in E_m^{\dagger} \\ v_0 \neq v \\ v_0 \text{ not } \rightsquigarrow v}} \mathcal{H}_{v_0}^\circ \otimes \dots \otimes \mathcal{H}_{v_m}.
		\]
		By distributing tensor products over direct sums, we have a unitary isomorphism
		\begin{equation} \label{eq:UV}
			u_v: \mathcal{H} \to [(\C \oplus \mathcal{H}_v^{\circ}) \otimes \mathcal{H}_{\rightsquigarrow v}] \oplus \mathcal{H}_{\perp v} \to [\mathcal{H}_v \otimes \mathcal{H}_{\rightsquigarrow v}] \oplus \mathcal{H}_{\perp v},
		\end{equation}
		where the first term $\C \otimes \mathcal{H}_{\rightsquigarrow v} = \mathcal{H}_{\rightsquigarrow v}$ corresponds to reverse walks that start with a vertex $v' \rightsquigarrow v$, the second term $\mathcal{H}_v^{\circ} \otimes \mathcal{H}_{\rightsquigarrow v}$ corresponds to reverse walks that start with $v$, and the third term $\mathcal{H}_{\perp v}$ corresponds to all other reverse walks.  Then we define
		\[
		\iota_v(a) = u_v \left( [a \otimes \id_{\mathcal{H}_{\rightsquigarrow v}}] \oplus 0_{\mathcal{H}_{\perp V}} \right) u_v^* \text{ for } a \in B(\mathcal{H}_v).
		\]
	\end{definition}
	
	The next lemma shows that $\iota_v$ is expectation-preserving.
	
	\begin{lemma} \label{lem:expectationpreserving}
		With the setup and notation of the previous two definitions,
		\[
		\ip{\xi, \iota_v(a) \xi} = \ip{\xi_v, a \xi_v} \text{ for all } a \in B(\mathcal{H}_v) \text{ and } v\in V.
		\]
	\end{lemma}
	
	\begin{proof}
		The subspace $\C \xi \oplus \mathcal{H}_v^{\circ} \cong \mathcal{H}_v$ of $\mathcal{H}$ corresponds to the term $\mathcal{H}_v \otimes \C \subseteq \mathcal{H}_v \otimes \mathcal{H}_{\rightsquigarrow v}$ in \eqref{eq:UV}.  From this we see that $\C \xi \oplus \mathcal{H}_v^{\circ}$ is an invariant subspace of $\iota_v(a)$ on which $\iota_v(a)$ acts in the same way as $a$ acts on $\mathcal{H}_v = \C \xi_v \oplus \mathcal{H}_v^{\circ}$.  The conclusion of the lemma is immediate from this.
	\end{proof}
	
	This also allows us to define the digraph product of $\mathrm{C}^*$-probability spaces as follows.
	
	\begin{definition}[$G$-product of $\mathrm{C}^*$-probability spaces] \label{def: G free product probability space}
		Let $G$ be a finite digraph and for $v \in V$, let $(\mathcal{A}_v,\phi_v)$ be a $\mathrm{C}^*$-probability space.  Let $\pi_v: \mathcal{A}_v \to B(H_v)$ be the GNS representation associated to $\mathcal{A}_v$ and $\phi_v$, and let $\xi_v = \widehat{1}$ be the corresponding state vector.  Let $(H,\xi)$ be the $G$-product Hilbert space and $\iota_v: B(H_v) \to B(H)$ the corresponding inclusions.  Then the $G$-product\footnote{Note that in the case where $G$ is a complete graph, this corresponds to a \emph{reduced} free product of $\mathrm{C}^*$-probability spaces.} is the pair $(\mathcal{A},\phi)$ where $\mathcal{A}$ is the unital $\mathrm{C}^*$-subalgebra of $B(H)$ generated by $\iota_v \circ \pi_v(\mathcal{A}_v)$ for $v \in V$ and $\phi$ is the state $\phi(a) = \ip{\xi, a\xi}$.
	\end{definition}
	
	\begin{lemma}
		The $G$-product defined above is a $\mathrm{C}^*$-probability space. In particular, the GNS representation associated to the state $\phi$ is faithful.
	\end{lemma}
	
	\begin{proof}
		It is clear that $\mathcal{A}$ is a unital $\mathrm{C}^*$-algebra and $\phi$ is a state.  In order to show faithfulness of the representation, it suffices to show that the state vector $\xi$ is cyclic for $\mathcal{A}$.  Since $\mathcal{A}_v$ is assumed to be unital, we see that $\{\widehat{a}: a \in \mathcal{A}_v, \phi_v(a) = 0\}$ is dense in $H_v^\circ$.  Consider a reversed walk $v_0$, \dots, $v_m$ in the graph $G$, and for $j = 1$, \dots, $m$, let $a_j \in \mathcal{A}_{v_j}$ with $\phi_{v_j}(a_j) = 0$.  Then
		\[
		\iota_{v_0}(\pi_{v_0}(a_0)) \dots \iota_{v_m}(\pi_{v_m}(a_m)) \xi = \widehat{a}_0 \otimes \dots \otimes \widehat{a}_m.
		\]
		Thus, $\overline{\mathcal{A} \xi}$ contains $H_{v_0}^\circ \otimes \dots \otimes H_{v_m}^{\circ}$.  Since this holds for all reversed walks, we see that $\overline{\mathcal{A} \xi} = H$, or equivalently $\xi$ is cyclic, as desired.
	\end{proof}
	
	\subsection{Computation of joint moments} \label{subsec: moments}
	
	Let $G = (V,E)$ be a digraph and let $(\mathcal{H}_v,\xi_v)_{v \in V}$ be a family of pointed Hilbert spaces indexed by $V$.  Let $(\mathcal{H},\xi)$ be the $G$-product Hilbert space from Definition \ref{def:productHilbertspace} and let $\iota_v: B(\mathcal{H}_v) \to B(\mathcal{H})$ be the $*$-homomorphisms described in Definition \ref{def:inclusionmaps}.  Our goal is to compute
	\[
	\ip{\xi, \iota_{\ell(1)}(a_1) \dots \iota_{\ell(k)}(a_k) \xi},
	\]
	where $\ell: [k] \to V$ is a function and $a_j \in B(\mathcal{H}_{\ell(j)})$.  The resulting moment formula will be the basis for the definition of $G$-independence (see Definition \ref{def: G free independence}).
	
	Our formula expresses the joint moments in terms of Boolean cumulants and non-crossing partitions.  Although it may be counterintuitive to use solely Boolean cumulants to describe mixtures of free, Boolean, and monotone independence, the Boolean cumulants have already played an important role in analyzing free and c-free independence; see e.g.\ \cite{PMWJC2011,PV2013,FLKS2021}.  The situation is also analogous to the way that the moments for mixtures of free and classical convolution are expressed using only free cumulants in \cite[Theorem 5.2]{SpWys2016}.  We recall the definition of Boolean cumulants now.

	\begin{definition}
		A partition $\pi$ of $[k]$ is called an \emph{interval partition} if every block $B \in \pi$ has the form $B = \{l\in [k]:i\leq l \leq j\}$ for some $1 \leq i \leq j \leq k$.  We denote by $I_k$ the set of interval partitions of $[k]$.
	\end{definition}
	
	\begin{definition}
		Let $(\mathcal{A},\phi)$ be a $\mathrm{C}^*$-probability space.  We define Boolean cumulants $K_{\bool,k}: \mathcal{A}^k \to \C$ by
		\[
		K_{\bool,k}[a_1,\dots,a_k] = \sum_{\pi \in I_k} (-1)^{|\pi|-1} \prod_{B \in \pi} \phi \left( \prod_{j \in B}^{\longrightarrow} a_j \right),
		\]
		where $\vec{\prod}_{j \in B} a_j$ denotes the product of the $a_j$ for $j \in B$ written in order from left to right.
	\end{definition}
	
	We recall the following fact about Boolean cumulants (see \cite[Lemma 4.14]{JekelLiu2020}), which seems to be well known in combinatorial non-commutative probability; see for instance \cite[Lemma 2.9]{PV2013} and \cite[proof of Proposition 3.2.1]{ACG2023}.
	
	\begin{lemma} \label{lem: Boolean cumulants}
		Let $(\mathcal{A},\phi)$ be a $\mathrm{C}^*$-probability space and $(\mathcal{H},\xi)$ a pointed Hilbert space with $\mathcal{A} \subseteq B(\mathcal{H})$ and $\phi(a) = \ip{\xi, a \xi}$.  Let $P \in B(\mathcal{H})$ be the rank-one projection onto $\mathbb{C}\xi$ and $Q = 1 - P$.  Then
		\[
		K_{\bool,k}[a_1,\dots,a_k] = \ip{\xi, a_1 Q a_2 \dots Q a_k \xi}.
		\]
	\end{lemma}
	
	\begin{proof}
		Write
		\begin{align*}
			\ip{\xi, a_1 Q a_2 \dots Q a_k \xi} &= \ip{\xi, a_1(1 - P)a_2 \dots (1 - P)a_k \xi} \\
			&= \sum_{T_1, \dots, T_{k-1} \in \{1,-P\}} \ip{\xi, a_1 T_1 a_2 \dots T_{k-1} a_k \xi}.
		\end{align*}
		Consider the map from the set $I_k$ of interval partitions to the sequences $T_1$, \dots, $T_k$ from $\{1,-P\}$ that sends $\pi \in I_k$ to the sequence $T_1$, \dots, $T_k$ with $T_j = 1$ if $j$ and $j+1$ are in the same block of $\pi$ and $T_j = -P$ if $j$ and $j+1$ are in different blocks of $\pi$.  It is straightforward to check that this is a bijection. Moreover, if $T_1$, \dots, $T_k$ is the sequence associated to $\pi$, and if $j_1, \dots, j_{|\pi|-1}$ are the indices where $j$ and $j+1$ are in distinct blocks, then we have
		\begin{align*}
			\ip{\xi, a_1 T_1 a_2 \dots T_{k-1} a_k \xi} &= (-1)^{|\pi|-1} \ip*{\xi, a_1 \dots a_{j_1} P a_{j_1+1} \dots a_{j_2} \dots P a_{j_{|\pi|-1}} \dots a_k \xi} \\
			&= (-1)^{|\pi|-1} \ip{\xi, a_1 \dots a_{j_1} \xi} \dots \ip{\xi, a_{j_{|\pi|-1}} \dots a_k \xi} \\
			&= (-1)^{|\pi|-1} \prod_{B \in \pi} \phi \left( \prod_{j \in B}^{\longrightarrow} a_j \right).
		\end{align*}
		This is precisely the definition of $K_{\bool,k}[a_1,\dots,a_k]$, so the proof is complete.
	\end{proof}
	
	\begin{definition} \label{def: partition labeling compatible}
		Let $G = (V,E)$ be a digraph.  By a \emph{labeling of $[k]$} (by elements of  $V$), we mean a function $\ell: [k] \to V$.  For every such labeling, we denote by $\mathcal{NC}_k(\ell)$ the set of $\pi \in \mathcal{NC}_k$ such that $\ell$ is constant on each block of $\pi$. We say that $\pi$ and $\ell$ are \emph{compatible} if $\pi \in \mathcal{N}C_k(\ell)$.  See Figure \ref{fig: partition and labeling} for an example.
	\end{definition}
	
	\begin{figure}
		\begin{center}
			\begin{tikzpicture}
				\node[circle,fill] (1) at (1,0) [label=below:$1$] {};
				\node[circle,fill] (2) at (2,0) [label=below:$2$] {};
				\node[circle,fill] (3) at (3,0) [label=below:$3$] {};
				\node[circle,fill] (4) at (4,0) [label=below:$4$] {};
				\node[circle,fill] (5) at (5,0) [label=below:$5$] {};
				\node[circle,fill] (6) at (6,0) [label=below:$6$] {};
				\node[circle,fill] (7) at (7,0) [label=below:$7$] {};
				\node[circle,fill] (8) at (8,0) [label=below:$8$] {};
				\node[circle,fill] (9) at (9,0) [label=below:$9$] {};
				\node[circle,fill] (10) at (10,0) [label=below:$10$] {};
				
				\node at (1,-0.8) {$a$};
				\node at (2,-0.8) {$d$};
				\node at (3,-0.8) {$d$};
				\node at (4,-0.8) {$a$};
				\node at (5,-0.8) {$b$};
				\node at (6,-0.8) {$c$};
				\node at (7,-0.8) {$b$};
				\node at (8,-0.8) {$a$};
				\node at (9,-0.8) {$c$};
				\node at (10,-0.8) {$c$};
				
				\draw (1) -- (1,2) -- (8,2) -- (8);
				\draw (4) -- (4,2);
				\draw (2) -- (2,1) -- (3,1) -- (3);
				\draw (5) -- (5,1) -- (7,1) -- (7);
				\draw (9) -- (9,1) -- (10,1) -- (10);
			\end{tikzpicture}
		\end{center}
		
		\caption{A non-crossing partition $\pi = \{\{1,4,8\}, \{2,3\}, \{5,7\}, \{6\}, \{9,10\}\}$ compatible with a certain labeling $\ell: [10] \to V = \{a,b,c,d\}$.  See Definition \ref{def: partition labeling compatible}.}
		\label{fig: partition and labeling}
	\end{figure}
	
	\begin{definition} \label{def: partition labeling graph compatible}
		Let $G = (V,E)$ be a digraph and $\ell: [k] \to V$.  Let $\pi \in \mathcal{NC}_k(\ell)$.  Let $\tilde{\ell}: \pi \to V$ be the map given by $\tilde{\ell}(B) = \ell(j)$ for $j \in B$.  We define $\mathcal{NC}_k(\ell,G)$ as the set of $\pi \in \mathcal{NC}_k(\ell)$ such that $\tilde{\ell}$ defines a digraph homomorphism $\mathrm{F}(\pi) \to G$.  In this case, we say that \emph{$\pi$, $\ell$, and $G$ are compatible}.  See Figure \ref{fig: partition and labeling and graph} for an example.
	\end{definition}
	
	\begin{figure}
		\begin{center}
			\begin{tikzpicture}
				\begin{scope}
					\node at (1,0.8) {$\mathrm{F}(\pi)$};
					
					\node[circle,fill] (B1) at (0,0) [label=left:{$\{1,4,8\}$},label=right:$a$] {};
					\node[circle,fill] (B2) at (-1,-1) [label=left:{$\{2,3\}$},label=right:$d$] {};
					\node[circle,fill] (B3) at (1,-1) [label=left:{$\{5,7\}$},label=right:$b$] {};
					\node[circle,fill] (B4) at (1,-2) [label=left:{$\{6\}$},label=right:$c$] {};
					\node[circle,fill] (B5) at (3,0) [label=left:{$\{9,10\}$},label=right:$c$] {};
					
					\draw[->] (B1) to (B2);
					\draw[->] (B1) to (B3);
					\draw[->] (B3) to (B4);
				\end{scope}
				
				\node at (4,0.8) [label=above:$\ell$] {$\longrightarrow$};
				
				\begin{scope}[shift = {(7,0)}]
					\node at (0,0.8) {$G$};
					
					\node[circle,fill] (A) at (0,-0.2) [label=above:$a$] {};
					\node[circle,fill] (B) at (-1,-0.8) [label=left:$b$] {};
					\node[circle,fill] (C) at (0,-1.4) [label=below:$c$] {};
					\node[circle,fill] (D) at (1,-0.8) [label=right:$d$] {};
					
					\draw[->] (A) to (B);
					\draw[->] (B) to (C);
					\draw[->] (C) to (A);
					\draw[->] (A) to (D);
				\end{scope}
			\end{tikzpicture}
		\end{center}
		
		\caption{Compatibility between a non-crossing partition $\pi$, labeling $\ell$, and graph $G$.  Left: the nesting forest of the partition $\pi$ from Figure \ref{fig: partition and labeling}; each vertex has the block written on the left and the label on the right.  Right:  The graph $G$ on vertex set $V = \{a,b,c,d\}$.  Note that that if $B \rightsquigarrow B'$ in $\pi$, then $\ell(B) \rightsquigarrow \ell(B')$ in $G$.  See Definition \ref{def: partition labeling graph compatible}.}
		\label{fig: partition and labeling and graph}
	\end{figure}
	
	In the rest of the paper, we will slightly  abuse notation and use the same name $\ell$ for both $\ell$ and $\tilde{\ell}$ in the above definition.  In other words, when $\pi$ is compatible with $\ell$, we will view $\ell$ sometimes as defined on the indices of $[k]$ and sometimes as defined on the blocks of $\pi$.
	
	\begin{theorem} \label{thm:momentformula}
		Let $G = (V,E)$ be a digraph.  Let $(\mathcal{H}_v,\xi_v)_{v \in V}$ be a collection of pointed Hilbert spaces indexed by $v \in V$, let $(\mathcal{H},\xi)$ be the $G$-product as in Definition \ref{def:productHilbertspace}, and let $\iota_v: B(\mathcal{H}_v) \to B(\mathcal{H})$ for $v \in V$ be the $*$-homomorphisms given in Definition \ref{def:inclusionmaps}.  Let
		\begin{align*}
			\phi_v: & B(\mathcal{H}_v) \to \C, a \mapsto \ip{\xi_v, a \xi_v} \\
			\phi: & B(\mathcal{H}) \to \C, a \mapsto \ip{\xi, a\xi},
		\end{align*}
		so that $(B(\mathcal{H}_v),\phi_v)$ and $(B(\mathcal{H}),\phi)$ are $\mathrm{C}^*$-probability spaces and $\iota_v: B(\mathcal{H}_v) \to B(\mathcal{H})$ is expectation-preserving by Lemma \ref{lem:expectationpreserving}.
		
		For $k \in \N$ let $\ell: [k] \to V$ be a labeling, and let $a_j \in B(\mathcal{H}_{\ell(j)})$ for $j = 1$, \dots, $k$.  Then
		\[
		\ip{\xi, \iota_{\ell(1)}(a_1) \dots \iota_{\ell(k)}(a_k) \xi} = \sum_{\pi \in \mathcal{NC}_k(\ell,G)} \prod_{B \in \pi} K_{\bool,|B|}[a_j: j \in B].
		\]
		where for each block $B$, $K_{\bool,|B|}$ denotes the $|B|$th Boolean cumulant associated to $(B(\mathcal{H}_{\ell(B)}),\phi_{\ell(B)})$, and the arguments $a_j: j \in B$ are written in increasing order of their indices from left to right.
	\end{theorem}
	
	\begin{remark}
		For each block $B$ of a partition $\pi$ in the above formula, since the map $\iota_{\ell(B)}$ is expectation-preserving, we could equivalently write
		\[
		\ip{\xi, \iota_{\ell(1)}(a_1) \dots \iota_{\ell(k)}(a_k) \xi} = \sum_{\pi \in \mathcal{NC}_k(\ell,G))} \prod_{B \in \pi} K_{\bool,|B|}[\iota_{\ell(j)}(a_j): j \in B].
		\]
	\end{remark}
	
	As explained in \S \ref{subsec: tree independence}, this result, and its proof, are a special case of \cite[Theorem 4.21]{JekelLiu2020}.  For the sake of exposition, we explain the proof in more detail for the particular case of $G$-independence.
	
	\begin{proof}[Proof of Theorem \ref{thm:momentformula}]
		As a notational convenience, let us reindex the operators $\iota_{\ell(1)}(a_1)$, \dots, $\iota_{\ell(k)}(a_k)$ in reverse order, so that $\iota_{\ell(1)}(a_1)$ is the right-most operator, i.e.\ it is applied to $\xi$ first.  Thus, we want to prove that
		\begin{equation} \label{eq: moment formula reverse}
			\ip{\xi, \iota_{\ell(k)}(a_k) \dots \iota_{\ell(1)}(a_1) \xi} = \sum_{\pi \in \mathcal{NC}_k(\ell,G))} \prod_{B \in \pi} K_{\bool,|B|}[a_j: j \in B],
		\end{equation}
		with the indices $a_j: j \in B$ now in \emph{decreasing} order for each block.
		
		Let $P_v \in B(\mathcal{H}_v)$ be the rank-one projection onto $\C \xi_v$, and let $Q_v = 1 - P_v$.  Let
		\begin{align*}
			a_j^{(0,0)} &= P_{\ell(j)} a_j P_{\ell(j)} \\
			a_j^{(0,1)} &= P_{\ell(j)} a_j Q_{\ell(j)} \\
			a_j^{(1,0)} &= Q_{\ell(j)} a_j P_{\ell(j)} \\
			a_j^{(1,1)} &= Q_{\ell(j)} a_j Q_{\ell(j)}.
		\end{align*}
		Then $a_j = a_j^{(0,0)} + a_j^{(0,1)} + a_j^{(1,0)} + a_j^{(1,1)}$.  We may thus write
		\begin{equation} \label{eq: 4 term expansion}
			\ip{\xi, \iota_{\ell(k)}(a_k) \dots \iota_{\ell(1)}(a_1) \xi} = \sum_{\delta_1,\epsilon_1,\dots,\delta_k,\epsilon_k \in \{0,1\}} \ip{\xi, \iota_{\ell(k)}(a_k^{(\delta_k,\epsilon_k)}) \dots \iota_{\ell(1)}(a_1^{(\delta_1,\epsilon_1)}) \xi}.
		\end{equation}
		Our goal is to show that certain of the terms in the sum vanish, while the others correspond to non-crossing partitions and evaluate to the product of Boolean cumulants in the asserted formula.  Note that $a_j^{(\delta,\epsilon)}$ annihilates $\mathcal{H}_{\ell(j)}^\circ$ when $\epsilon = 0$ and annihilates $\C \xi_{\ell(j)}$ when $\epsilon = 1$, and its image is contained in $\C \xi_{\ell(j)}$ when $\delta = 0$ and $\mathcal{H}_{\ell(j)}^\circ$ when $\delta = 1$.
		
		Examining the definition of the maps $\iota_v$ in Definition \ref{def:inclusionmaps}, we conclude the following.
		
		\begin{fact} \label{obs: where they map} ~
			\begin{itemize}
				\item $\iota_{\ell(j)}(a^{(0,0)})$ maps $\mathcal{H}_{v_0}^\circ \otimes \dots \otimes \mathcal{H}_{v_m}$ into itself if $v_0 \rightsquigarrow \ell(j)$, and vanishes on $\mathcal{H}_{v_0}^\circ \otimes \dots \otimes \mathcal{H}_{v_m}$ otherwise.
				\item $\iota_{\ell(j)}(a^{(1,0)})$ maps $\mathcal{H}_{v_0}^\circ \otimes \dots \otimes \mathcal{H}_{v_m}$ into $\mathcal{H}_{\ell(j)}^\circ \otimes \mathcal{H}_{v_0}^\circ \otimes \dots \otimes \mathcal{H}_{v_m}$ if $v_0 \rightsquigarrow \ell(j)$, and vanishes on $\mathcal{H}_{v_0}^\circ \otimes \dots \otimes \mathcal{H}_{v_m}$ otherwise.
				\item $\iota_{\ell(j)}(a^{(0,1)})$ maps $\mathcal{H}_{v_0}^\circ \otimes \dots \otimes \mathcal{H}_{v_m}$ into itself if $v_0 \rightsquigarrow \ell(j)$, and vanishes on $\mathcal{H}_{v_0}^\circ \otimes \dots \otimes \mathcal{H}_{v_m}$ otherwise.
				\item $\iota_{\ell(j)}(a^{(1,1)})$ maps $\mathcal{H}_{v_0}^\circ \otimes \dots \otimes \mathcal{H}_{v_m}$ into itself if $v_0 = \ell(j)$, and vanishes on $\mathcal{H}_{v_0}^\circ \otimes \dots \otimes \mathcal{H}_{v_m}$ otherwise.
			\end{itemize}
		\end{fact}
		
		With this information in mind, we can then consider the effect of applying several operators $\iota_{\ell(j)}(a^{(\delta_j,\epsilon_j)})$ consecutively to the state vector $\xi$, and thus determine which direct summand of the Hilbert space $\mathcal{H}$ contains the vector
		\[
		\iota_{\ell(j)}(a_j^{(\delta_1,\epsilon_1)}) \dots \iota_{\ell(1)}(a_1^{(\delta_1,\epsilon_1)}) \xi
		\]
		for each $j \leq k$.   First, to keep track of the number of tensorands, we introduce a \emph{height function} $h$ associated to the sequence of indices $(\delta_j,\epsilon_j)$.  Let
		\[
		h(m) = \sum_{j=1}^m (\delta_i - \epsilon_i).
		\]
		Note that $h(0) = 0$, and $h(j+1) - h(j) \in \{-1,0,1\}$.  By inductive application of the observations above, one can show that $\iota_{\ell(j)}(a_j) \dots \iota_{\ell(1)}(a_1) \xi$ is contained in one of the $h(j)$-fold tensor products among the direct summands in the definition of $\mathcal{H}$, provided that $h(i) \geq 0$ for $i \leq j$.  If $h(j)$ is ever $-1$, then the first time that $h(j) = -1$, we are applying an `annihilation operator' $\iota_{\ell(j)}(a_j^{(0,1)})$ to a multiple of the state vector $\xi$, which results in $\iota_{\ell(j)}(a_j^{(\delta_j,\epsilon_j)}) \dots \iota_{\ell(1)}(a_1^{(\delta_1,\epsilon_1)}) \xi = 0$.  Hence also, if $h(i) < 0$ for any $i \leq j$, then $\iota_{\ell(j)}(a_j^{(\delta_j,\epsilon_j)}) \dots \iota_{\ell(1)}(a_1^{(\delta_1,\epsilon_1)}) \xi = 0$.  Furthermore, at the last step, for the inner product to be nonzero, $\iota_{\ell(k)}(a_k^{(\delta_k,\epsilon_k)}) \dots \iota_{\ell(1)}(a_1^{(\delta_1,\epsilon_1)}) \xi$ must be in $\C \xi$, and hence $h(k) = 0$.
		
		Therefore, in the expansion \ref{eq: 4 term expansion}, only the summands which have a nonnegative height function $h$ with $h(k) = 0$ will remain.  We want to express these in terms of non-crossing partitions.  Thus, we recall the following fact, which is a generalization of the well-known bijection between non-crossing \emph{pair} partitions and Dyck paths.  We will not give the proof here in detail, since a similar argument is given in \cite[Lemma 4.24]{JekelLiu2020}.  However, note here that we are picturing the indices $1$, \dots, $k$ as running from \emph{right} to \emph{left}.
		
		\begin{lemma} \label{lem: bijection}
			There is a bijection between
			\begin{enumerate}[(1)]
				\item sequences $(\delta_1,\epsilon_1)$, \dots, $(\delta_k,\epsilon_k)$ whose height function $h$ is nonnegative and satisfies $h(k) = 0$, and
				\item non-crossing partitions $\pi \in \mathcal{NC}_k$,
			\end{enumerate}
			described by the following relationship:
			\begin{itemize}
				\item $\{j\}$ is a singleton in $\pi$ if and only if $(\delta_j,\epsilon_j) = (0,0)$.
				\item $\{j\}$ is the upper (left) endpoint of a non-singleton block in $\pi$ if and only if $(\delta_j,\epsilon_j) = (0,1)$.
				\item $\{j\}$ is the lower (right) endpoint of a non-singleton block in $\pi$ if and only if $(\delta_j,\epsilon_j) = (1,0)$.
				\item $\{j\}$ is in a non-singleton block and not an endpoint of the block if and only if $(\delta_j,\epsilon_j) = (1,1)$.
			\end{itemize}
			See Figure \ref{fig: partition and Motzkin path} for an example.
		\end{lemma}
		
		\begin{figure}
			\begin{center}
				\begin{tikzpicture}[xscale=1.2]
					\node[circle,fill] (0) at (0,0) [label=below:{$(0,1)$}] {};
					\node[circle,fill] (1) at (1,0) [label=below:{$(0,0)$}] {};
					\node[circle,fill] (2) at (2,0) [label=below:{$(1,1)$}] {};
					\node[circle,fill] (3) at (3,0) [label=below:{$(0,1)$}] {};
					\node[circle,fill] (4) at (4,0) [label=below:{$(1,0)$}] {};
					\node[circle,fill] (5) at (5,0) [label=below:{$(1,0)$}] {};
					
					\draw (0) -- (0,2) -- (5,2) -- (5);
					\draw (2) -- (2,2);
					\draw (3) -- (3,1) -- (4,1) -- (4);
					
					\draw (-0.5,-3) node[circle,fill,inner sep=2pt] {} to (0.5,-2) node[circle,fill,inner sep=2pt] {} to (1.5,-2) node[circle,fill,inner sep=2pt] {} to (2.5,-2) node[circle,fill,inner sep=2pt] {} to (3.5,-1) node[circle,fill,inner sep=2pt] {} to (4.5,-2) node[circle,fill,inner sep=2pt] {} to (5.5,-3) node[circle,fill,inner sep=2pt] {};
				\end{tikzpicture}
			\end{center}
			
			\caption{Correspondence between non-crossing partitions and a sequences of indices $(\delta_j,\epsilon_j)$ with $h \geq 0$, as in Lemma \ref{lem: bijection}.  The partition is shown above the sequence of indices, and the height function is shown below.}
			\label{fig: partition and Motzkin path}
		\end{figure}
		
		Now given a non-crossing partition $\pi$, we need to evaluate the corresponding term $\ip{\xi, \iota_{\ell(k)}(a_k^{(\delta_k,\epsilon_k)}) \dots \iota_{\ell(1)}(a_1^{(\delta_1,\epsilon_1)}) \xi}$.  In particular, we must show it is zero unless $\pi$, $\ell$, and $G$ are compatible.
		
		We aim to evaluate $\iota_{\ell(j)}(a_j^{(\delta_j,\epsilon_j)}) \dots \iota_{\ell(1)}(a_1^{(\delta_1,\epsilon_1)}) \xi$ by induction on $j$.  To this end, we introduce more notation.  Let $\pi_j$ be the restriction of $\pi$ to $[j]$, which is a non-crossing partition.  Each block of $\pi_j$ is thus $B \cap [j]$ for some block $B$ in $\pi$.  A block of $\pi_j$, say $B \cap [j]$, is called \emph{finished} if $B \cap [j] = B$ and \emph{unfinished} otherwise.  Let $\mathrm{F}_j$ be the out-forest where there is an edge from $B \cap [j]$ to $B' \cap [j]$ in $\mathrm{F}_j$ if and only if there is an edge from $B$ to $B'$ in $\mathrm{F}(\pi)$, which is a subgraph of $\mathrm{F}(\pi)$.
		
		\begin{lemma} \label{lem: evaluation}
			If the labeling $\ell$ is constant on each block of $\pi_j$ and defines a homomorphism from $\mathrm{F}_j$ to $G$, then
			\begin{multline} \label{eq: expansion}
				\iota_{\ell(j)}(a_j^{(\delta_j,\epsilon_j)}) \dots \iota_{\ell(1)}(a_1^{(\delta_1,\epsilon_1)}) \xi \\ = \bigotimes_{r=m}^1 \left[ \left( \prod_{s \in B_r}^{\longleftarrow} Q_{\ell(B_r)} a_s \right) \xi_{\ell(B_r)} \right] \prod_{B \in \pi_j \text{ finished}} K_{\bool,|B|}[a_s: s \in B].
			\end{multline}
			where $B_1$, \dots, $B_m$ are the unfinished blocks of $\pi_j$, ordered by $\min B_1 < \dots < \min B_m$, and the terms in $\prod\limits_{s \in B_r}^{\longleftarrow} Q_{\ell(B_r)} a_s$ are multiplied from left to right in decreasing order of the index $s$.  Here, as above in \eqref{eq: moment formula reverse}, the terms $a_s: s \in B$ in the Boolean cumulant also run in decreasing order from left to right. In all other cases, $\iota_{\ell(j)}(a_j^{(\delta_j,\epsilon_j)}) \dots \iota_{\ell(1)}(a_1^{(\delta_1,\epsilon_1)}) \xi = 0$.
		\end{lemma}
		
		\begin{proof}
			We proceed by induction.  The base case $j = 0$ is immediate; all the products are empty and so both sides evaluate to $\xi$. 
			For the induction step, suppose the claim is true for $j$ and we will prove it for $j + 1$.  For simplicity, let us denote by ($*$) the condition that the labeling $\ell$ is constant on each block of $\pi_j$ and defines a homomorphism from $\mathrm{F}_j$ to $G$.
			
			If ($*$) fails for $j$, then it also fails for $j + 1$.  By induction hypothesis,
			\[
			\iota_{\ell(j)}(a_j^{(\delta_j,\epsilon_j)}) \dots \iota_{\ell(1)}(a_1^{(\delta_1,\epsilon_1)}) \xi = 0,
			\]
			and hence also
			\[
			\iota_{\ell(j+1)}(a_{j+1}^{(\delta_{j+1},\epsilon_{j+1})}) \dots \iota_{\ell(1)}(a_1^{(\delta_1,\epsilon_1)}) \xi = 0.
			\]
			Thus, the claim holds for $j + 1$.
			
			Now suppose that ($*$) holds for $j$.  Note that
			\[
			\zeta_j := \iota_{\ell(j)}(a_j^{(\delta_j,\epsilon_j)}) \dots \iota_{\ell(1)}(a_1^{(\delta_1,\epsilon_1)}) \xi \in \bigotimes_{r=m}^1 \mathcal{H}_{\ell(B_r)}^\circ.
			\]
			Using the same notation as \eqref{eq: expansion} for the unfinished blocks, express $B_1$, \dots, $B_m$ as $[j] \cap B_1'$, \dots, $[j] \cap B_m'$ for blocks $B_1'$, \dots, $B_m'$ in $\pi$.  Note that $B_{s+1}'$ is nested inside $B_s'$ because $\min B_{s+1} > \min B_s$ but $B_s'$ contains an element greater than $\min B_{s+1}'$ because $B_s$ is unfinished in $\pi_j$.  Similar elementary reasoning with non-crossing conditions shows that there is no block strictly between $B_{s+1}$ and $B_s$ in the nesting order, so that $B_s' \rightsquigarrow B_{s+1}'$ in $\mathrm{F}(\pi)$, hence also $B_s \rightsquigarrow B_{s+1}$ in $\mathrm{F}_j$.
			
			We consider cases based on $(\delta_{j+1},\epsilon_{j+1})$.  
			\begin{enumerate}[(1)]
				\item Suppose $(\delta_{j+1},\epsilon_{j+1}) = (0,0)$, so that $a^{(0,0)}$ is a multiple of $P_{\ell(j+1)}$.  Thus, $\{j+1\}$ is a singleton block in $\pi$ that is nested inside $B_m'$. Moreover, $\{j+1\}$ is a finished block in $\pi_{j+1}$, and it is the only new vertex in $\mathrm{F}_{j+1}$ that was not in $\mathrm{F}_j$.  Thus, $\ell$ defines a homomorphism $\mathrm{F}_{j+1} \to G$ if and only if $\ell(B_m) \rightsquigarrow \ell(j+1)$.  Therefore, if $\ell(B_m) \not \rightsquigarrow \ell(j+1)$, then ($*$) fails for $j+1$ and $\iota_{\ell(j+1)}(a^{(0,0)}) \zeta_j = 0$ by Fact \ref{obs: where they map}.  On the other hand, if $\ell(B_m) \rightsquigarrow \ell(j+1)$, then since $\zeta_j \in \bigotimes_{j=m}^1 \mathcal{H}_{\ell(B_j)}^\circ$, we obtain $\iota_{\ell(j+1)}(a_{j+1}^{(0,0)}) \zeta_j = K_{\bool,1}(a_{j+1}) \zeta_j$; meanwhile, on the right-hand side of \eqref{eq: expansion}, a new term of $K_{\bool,1}(a_{j+1})$ is added for the new finished block $\{j+1\}$ in $\pi_{j+1}$.
				
				\item Suppose $(\delta_{j+1},\epsilon_{j+1}) = (1,0)$.  In this case $\{j+1\}$ is a singleton block in $\pi_{j+1}$ that is the right endpoint of a block in $\pi$.  Similar to case (1), $\ell$ defines a homomorphism $\mathrm{F}_{j+1} \to V$ if and only if $\ell(B_m) \rightsquigarrow \ell(j+1)$.  Therefore, if $\ell(B_m) \not \rightsquigarrow \ell(j+1)$, then ($*$) fails for $j+1$ and $\iota_{\ell(j+1)}(a^{(1,0)}) \zeta_j = 0$ by Fact \ref{obs: where they map}.  On the other hand, if $\ell(B_m) \rightsquigarrow \ell(j+1)$, then since $\zeta_j \in \bigotimes_{j=m}^1 \mathcal{H}_{\ell(B_j)}^\circ$, we obtain $\iota_{\ell(j+1)}(a_{j+1}^{(1,0)}) \zeta_j = a_{j+1} \xi_{\ell(j+1)} \otimes \zeta_j$; meanwhile, on the right-hand side of \eqref{eq: expansion}, a new term of $a_{j+1} \xi_{\ell(j+1)}$ is added in the tensor product expansion corresponding to the new unfinished block $\{j+1\}$ in $\pi_{j+1}$.
				
				\item Suppose that $(\delta_{j+1},\epsilon_{j+1}) = (0,1)$.  Then $j+1$ is added to the most recent unfinished block $B_m$ in $\pi_j$, and this block is now finished in $\pi_{j+1}$.  Thus, $\ell$ defines a homomorphism $\mathrm{F}_{j+1} \to V$ if and only if $\ell(j+1) = \ell(B_m)$.  If $\ell(j+1) \neq \ell(B_m)$, then $\iota_{\ell(j+1)}(a_{j+1}^{(1,1)}) \zeta_j = 0$ by Fact \ref{obs: where they map}.  On the other hand, if $\ell(j+1) = \ell(B_m)$, then
				\begin{multline*}
					\iota_{\ell(j+1)}(a_{j+1}^{(0,1)}) \bigotimes_{r=m}^1 \left[ \left( \prod_{s \in B_m}^{\longleftarrow} Q_{\ell(B_m)} a_s \right) \xi_{\ell(B_r)} \right] \\
					= \ip*{ \xi_{\ell(j+1)}, a_{j+1} \left( \prod_{s \in B_m}^{\longleftarrow} Q_{\ell(B_m)} a_s \right) \xi_{\ell(B_m)} } \bigotimes_{r=m-1}^1 \left[ \left( \prod_{s \in B_r}^{\longleftarrow} Q_{\ell(B_r)} a_s \right) \xi_{\ell(B_r)} \right] \\
					= K_{\bool,|B_m|+1}[a_s: s \in B_m \cup \{j+1\}] \bigotimes_{r=m-1}^1 \left[ \left( \prod_{s \in B_r}^{\longleftarrow} Q_{\ell(B_r)} a_s \right) \xi_{\ell(B_r)} \right].
				\end{multline*}
				This change is accounted for on the right-hand side of \eqref{eq: expansion} by removing the block $B_m$ from the tensor product expansion for the unfinished blocks, and adding a new term for $B_1 \cup \{j+1\}$ in the product expansion for the finished blocks.
				
				\item Suppose that $(\delta_{j+1},\epsilon_{j+1}) = (1,1)$.  In this case, $j+1$ is added to the most recent unfinished block $B_m$ in $\pi_j$.  Thus, $\ell$ defines a homomorphism $\mathrm{F}_{j+1} \to V$ if and only if $\ell(j+1) = \ell(B_m)$.  If $\ell(j+1) \neq \ell(B_m)$, then $\iota_{\ell(j+1)}(a_{j+1}^{(1,1)}) \zeta_j = 0$ by Fact \ref{obs: where they map}.  On the other hand, if $\ell(j+1) \neq \ell(B_m)$, then
				\begin{multline*}
					\iota_{\ell(j+1)}(a_{j+1}^{(1,1)}) \bigotimes_{r=m}^1 \left[ \left( \prod_{s \in B_m}^{\longleftarrow} Q_{\ell(B_m)} a_s \right) \xi_{\ell(B_r)} \right] \\
					= Q_{\ell(j+1)} a_{j+1} \left( \prod_{s \in B_m}^{\longleftarrow} Q_{\ell(B_m)} a_s \right) \xi_{\ell(B_m)} \bigotimes_{r=m-1}^1 \left[ \left( \prod_{s \in B_r}^{\longleftarrow} Q_{\ell(B_r)} a_s \right) \xi_{\ell(B_r)} \right].
				\end{multline*}
				This change is accounted for on the right-hand side of \eqref{eq: expansion} by adding a new term corresponding to $j+1$ onto the product of $a_s$'s for the block $B_m$.
				
			\end{enumerate}
			In each case, the induction proceeds and completes the proof of the lemma.
		\end{proof}
		
		Now looking at the result of the lemma in the case where $j = k$, there are no unfinished blocks, and hence no tensor product terms.  Thus, \eqref{eq: expansion} reduces to $\prod_{B \in \pi} K_{\bool,|B|}[a_j: j \in B]$.  Thus, by Lemmas \ref{lem: bijection} and \ref{lem: evaluation}, the terms that survive in \eqref{eq: 4 term expansion} correspond to partitions $\pi$ that are compatible with $\ell$ and $G$.  Therefore, we obtain \eqref{eq: moment formula reverse}, which completes the proof of Theorem \ref{thm:momentformula}.
	\end{proof}
	
	
	\subsection{Definition and examples of $G$-independence} \label{subsec: G independence}
	
	Now that we understand the combinatorics of moments for the $G$-product, we define independence as follows:
	
	\begin{definition}[$G$-independence] \label{def: G free independence}
		Let $(\mathcal{A},\phi)$ be a $\mathrm{C}^*$-probability space, let $G = (V,E)$ be a graph, and let $(\mathcal{A}_v)_{v \in V}$ be $*$-subalgebras.  We say that $(\mathcal{A}_v)_{v \in V}$ are \emph{$G$-independent} if for every $k \in \N$, for every labeling $\ell: [k] \to V$, and for all $a_j \in \mathcal{A}_{\ell_j}$ for $j = 1, \ldots, k$, we have
		\begin{equation} \label{eq: G free independence}
			\phi(a_1 \dots a_k) = \sum_{\pi \in \mathcal{NC}_k(\ell,G))} \prod_{B \in \pi} K_{\bool,|B|}[a_j: j \in B].
		\end{equation}
	\end{definition}
	
	In other words, $G$-independence means by definition that the algebras $\mathcal{A}_v$ have joint moments satisfying the combinatorial formula from Theorem \ref{thm:momentformula}.  If we assume that $(\mathcal{A}_v,\phi|_{\mathcal{A}_v})$ is $\mathrm{C}^*$-probability space, independence means that the joint moments of elements of $A_v$ agree with those of the corresponding operators on a $G$-product $\mathrm{C}^*$-probability space in Definition \ref{def: G free product probability space}.  More precisely, let $\phi_v$ be the restriction of the state $\phi$ to $A_v$, let $H_v$ be the GNS Hilbert space associated to $(A_v,\phi_v)$, let $\xi_v$ the vector obtained from the unit, and let $\pi_v: A_v \to B(H_v)$ be the GNS representation.  Let $(H,\xi)$ be the $G$-product of $(H_v,\xi_v)$ and $\iota_v$ the corresponding $*$-homomorphism $B(H_v) \to B(H)$ constructed in Definitions \ref{def:productHilbertspace} and \ref{def:inclusionmaps}.  Then, by Theorem \ref{thm:momentformula}, $G$-independence means that the joint moments of operators $A_v$ in $A$ with respect to $\phi$ are the same as the joint moments of the operators $\iota_v \circ \pi_v(A_v)$ with respect to the vector state given by $\xi$.
	
	Although for a general $G$ we do not know how to describe $G$-independence using any simpler condition  than \eqref{eq: G free independence}, special cases of $G$-independence have more tractable characterizations, such as vanishing of certain mixed moments, multiplicativity of the state on certain products, or vanishing of mixed cumulants.  For instance, \cite[\S 4.6, \S 7.3]{JekelLiu2020} discusses in depth how the familiar moment conditions used to define Boolean, monotone, and free independence relate to the general moment formula \eqref{eq: G free independence}.  Here we will focus on BM-independence, as well as commenting on BF-independence and comparing our construction with BMT independence.
	
	\begin{definition}[BM-independence {\cite{JW2010}}] \label{def: BM independence}
		Let $(V,\prec)$ be a strict poset and for $v,w\in V$ by $v\nsim w$ we denote the incomparability of the elements.  We say that a family $(\mathcal{A}_v)_{v \in V}$ in a $\mathrm{C}^*$-probability space $(\mathcal{A},\phi)$ is \emph{BM-independent} if the following conditions hold:
		\begin{description}
			\item[\textbf{BM1}] Suppose that $v_1 \prec v_2 \succ v_3$ or $v_1 \prec v_2 \nsim v_3$ or $v_1 \nsim v_2 \succ v_3$ and $a_1 \in \mathcal{A}_{v_1}$, $a_2 \in \mathcal{A}_{v_2}$, and $a_3 \in \mathcal{A}_{v_3}$.  Then
			\begin{equation}\label{bm1}
				a_1 a_2 a_3 = \phi(a_2) a_1 a_3.
			\end{equation}
			\item[\textbf{BM2}] For $k \in \N$, $\ell: [k] \to V$, and $a_j \in \mathcal{A}_{\ell(j)}$ for $j = 1, \ldots, k$, if $\ell(1) \succ \dots \succ \ell(s) \nsim \ell(s+1) \nsim \dots \nsim
			\ell(t) \prec  \ell(t+1) \prec \dots \prec \ell(k)$ for some $1\le s \leq t \le k$, then
			\begin{equation}\label{bm2}
				\phi(a_{1}\ldots a_{k}) = \prod_{j=1}^{k}\phi(a_{j}).
			\end{equation}
		\end{description}
	\end{definition}
	
	The conditions  \textbf{BM1} and \textbf{BM2} above allow one to compute all joint moments $\phi(a_1\cdots a_k)$ of bm-independent random variables $a_1, \ldots, a_k$ by \cite[Lemmas~2.3,2.4]{JW2010} and an algorithm to evaluate joint moments using these conditions is given in \cite[Remark 2.3]{LOJW1}.
	
	Recall from \S \ref{subsec: digraphs} that we view (strict) posets as a subclass of digraphs.  We will show that the definition of BM independence in \cite{JW2008} agrees with our more general definition of $G$-independence when $G$ is a strict poset.  It will be useful first to observe the following alternative description of $\mathcal{NC}_k(\ell,G)$ in this case.
	
	\begin{definition}[Strict BM order; {\cite[Definitions 3.8]{LOJW1}}]
		Let $(V,\prec)$ be a strict poset, let $\ell: [k] \to V$ be a labeling, and let $\pi \in \mathcal{NC}_k(\ell)$.  We say that $\ell$ \emph{establishes strict BM order on} $\pi$ if $B \prec B'$ in $\pi$ implies that $\ell(B) \prec \ell(B')$ in $V$.
	\end{definition}
	
	\begin{fact} \label{obs: BM homomorphism equivalence}
		Let $G = (V,E) = (V,\prec)$ be a strict partial order, which we also view as a digraph.  Let $\ell: [k] \to V$ and let $\pi \in \mathcal{NC}_k(\ell)$.  Then the following are equivalent:
		\begin{enumerate}
			\item $\ell$ establishes strict BM order on $\pi$.
			\item $\ell$ defines a digraph homomorphism $F(\pi) \to G$.
			\item $\ell$ defines a strict poset homomorphism\footnote{A strict poset homomorphism by definition is a map that preserves strict inequality $\prec$.} $\pi \to (V,\prec)$, where $\pi$ is equipped with the strict partial order given by nesting.
		\end{enumerate}
	\end{fact}
	
	\begin{proof}
		(1) $\iff$ (3) is immediate from the definitions.  Moreover, (2) $\iff$ (3) is immediate from the definitions and transitivity of $\prec$.
	\end{proof}
	
	Now we prove the equivalence of two definitions of independence given by a finite poset.
	
	\begin{proposition} \label{prop: equivalent definitions of BM independence}
		Let $G = (V,E)$ be a digraph such that $E = (\prec)$ defines a strict partial order.  Let $(\mathcal{A},\phi)$ be a $\mathrm{C}^*$-probability space and let $(\mathcal{A}_v)_{v \in V}$ be $*$-subalgebras that generate $\mathcal{A}$.  Then $(\mathcal{A}_v)_{v \in V}$ are BM independent in the sense of Definition \ref{def: BM independence} if and only if they are $G$-independent in the sense of Definition \ref{def: G free independence}.
	\end{proposition}
	
	\begin{proof}
		First, suppose that Definition \ref{def: G free independence} holds.  We first need to show that $a_1 a_2 a_3 = \phi(a_2) a_1 a_3$ in the situation of \textbf{BM1}.  Because the GNS representation is assumed to be faithful, it suffices to show that $\phi(b_1 a_1 a_2 a_3 b_2) = \phi(a_2) \phi(b_1 a_1 a_3 b_2)$ for all $b_1$, $b_2 \in \mathcal{A}$, and by density and linearity, it suffices to consider $b_1$ and $b_2$ that are products of elements from the individual $\mathcal{A}_v$'s.  Hence, after changing notation, it suffices to prove the following claim:
		
		\textbf{BM1':} Let $k \in \N$, $\ell: [k] \to V$, and $a_j \in \mathcal{A}_{\ell(j)}$ for $j = 1, \ldots, k$.  Fix an index $j$, and suppose $\ell(j-1) \not \succeq \ell(j) \not \preceq \ell(j+1)$ or $\ell(j-1) \prec \ell(j) \nsim \ell(j+1)$ or $\ell(j-1) \nsim \ell(j) \succ \ell(j+1)$.  Then
		\begin{equation}\label{bm1a}
			\phi(a_1 \dots a_n) = \phi(a_j) \phi(a_1 \dots a_{j-1} a_{j+1} \dots a_n)
		\end{equation}

		Assume the hypothesis of \textbf{BM1'}, and we will show \eqref{bm1a}.  We will evaluate $\phi(a_1 \dots a_k)$ using \eqref{eq: G free independence} and show that it agrees with $\phi(a_j) \phi(a_1 \dots a_{j-1} a_{j+1} \dots a_k)$.
		
		We claim that for every partition $\pi$ appearing in \eqref{eq: G free independence}, $\{j\}$ must be a singleton in $\pi$.  Recall that a partition $\pi$ appears in \eqref{eq: G free independence} if and only if $\pi$ is consistently labelled by $\ell$ and the labeling defines a digraph homomorphism from $\mathrm{F}(\pi)$ to $G$, or equivalently it defines a strict poset homomorphism, that is, $B \prec B'$ in $\pi$ implies that $\ell(B) \prec \ell(B')$.   Now let $B$ be the block containing $j$.  Suppose for contradiction that there is some $i < j$ in $B$.  Since $\ell(j-1) \neq \ell(j)$, we see that $i \neq j-1$ and the block containing $j-1$ is nested immediately inside $B$, and so we would need $\ell(j) \prec \ell(j-1)$, but this contradicts our assumption that $\ell(j-1) \prec \ell(j)$ or $\ell(j-1) \nsim \ell(j)$.  Similarly, if we assume for contradiction that there is some $i > j$ in $B$, then we obtain a contradiction by a symmetrical argument since the block of $j+1$ would be nested immediately inside $B$.
		
		Since $\pi$ has a singleton block at $j$, we obtain a non-crossing partition $\pi' = \pi \setminus \{\{j\}\}$ of $[k] \setminus \{j\}$.  Note that $\pi'$ is compatible with $\ell' = \ell_{[k] \setminus \{j\}}$ and $G$.  Conversely, we claim that every partition $\pi'$ compatible with $\ell'$ and $G$ arises in this way, or equivalently, for every such $\pi'$, the partition $\pi' \cup \{\{j\}\}$ of $[k]$ is compatible with $\ell$ and $G$.  To this end, we must consider some blocks $B_1$ and $B_2$ in $\pi$ with $B_2$ immediately nested inside $B_1$.  Since $\pi'$ is already compatible with $\ell'$ and $G$, the only case to check is when $B_2 = \{j\}$.  Note that either $\ell(j-1) \prec \ell(j)$ or $\ell(j+1) \prec \ell(j)$.  Suppose that $\ell(j-1) \prec \ell(j)$. We have two subcases:
		\begin{itemize}
			\item If $B_1$ contains $\ell(j-1)$, then $\ell(B_1) = \ell(j-1) \prec \ell(j) = \ell(B_2)$, so we are done.
			\item If $B_1$ does not contain $\ell(j-1)$, then the block $B_3$ containing $\ell(j-1)$ is nested inside $B_1$, and hence $\ell(B_1) \prec \ell(B_3) = \ell(j-1) \prec \ell(j) = \ell(B_2)$, so again we are done.
		\end{itemize}
		In the case where $\ell(j+1) \prec \ell(j)$, the argument is symmetrical.
		
		Therefore, we obtain that
		\begin{align*}
			\phi(a_1 \dots a_k) &= \sum_{\pi \in \mathcal{NC}_k(\ell,G)} \prod_{B \in \pi} K_{\bool,|B|}[a_i: i \in B] \\
			&= K_{\bool,1}[a_j] \sum_{\pi' \in \mathcal{NC}_{[k] \setminus \{j\}}(\ell',G)} \prod_{B \in \pi}K_{\bool,|B|}[a_i: i \in B], \\
			&= \phi(a_j) \phi(a_1 \dots a_{j-1} a_{j+1} \dots a_k),
		\end{align*}
		where $\mathcal{NC}_{[k] \setminus \{j\}}(\ell',G)$ denotes the set of non-crossing partitions of $[k] \setminus \{j\}$ that are compatible with $G$ and $\ell'$.
		
		Next to check \textbf{BM2}, suppose that $\ell(1) \succ \dots \succ \ell(s) \nsim \ell(s+1) \nsim \dots \nsim
		\ell(t) \prec  \ell(t+1) \prec \dots \prec \ell(k)$ for some $1\le s \leq t \le k$.  Let $\pi$ be a partition compatible with $G$ and $\ell$.  We claim that $\pi$ consist entirely of singletons.  Suppose for contradiction that $i$ and $j$ are in the same block $B_1$ and $i < j$. Since $i < j$ and $s \leq t$, we must have either $i < t$ or $j > s$.  Suppose that $i < t$.  Let $B_2$ be the block containing $i+1$.  Then $B_1 \neq B_2$ since our assumption on $\ell$ implies that consecutive indices have distinct labels.  Since $i < t$, we have that $\ell(i) \not \preceq \ell(i+1)$ by our assumptions on $\ell$, and this contradicts the condition $\ell(B_1) \prec \ell(B_2)$ needed for $\pi$ to be compatible with $\ell$ and $G$.  If $j > s$, we obtain a contradiction by a symmetrical argument.  Thus, the only possibility is that $\pi$ consists of singletons, and therefore \eqref{eq: G free independence} reduces to $\phi(a_1 \dots a_k) = K_{\bool,1}[a_1] \dots K_{\bool,1}[a_k] = \phi(a_1) \dots \phi(a_k)$.
		
		Therefore, we have shown that Definition \ref{def: G free independence} implies Definition \ref{def: BM independence}.  Conversely, suppose that Definition \ref{def: BM independence} holds.  Let $\psi_v = \phi|_{\mathcal{A}_v}$.  Construct another $\mathrm{C}^*$-probability space $(\mathcal{B},\psi)$ as the $G$-product of $(\mathcal{A}_v,\psi_v)$, and let $\mathcal{B}_v$ be the image of $\mathcal{A}_v$ in $\mathcal{B}$.  Then the $\mathcal{B}_v$'s are $G$-independent by Theorem \ref{thm:momentformula}.  Therefore, also the $\mathcal{B}_v$'s are BM-independent by the preceding argument.  By \cite[Lemmas~2.3,2.4]{JW2010}, BM-independence uniquely determines the joint moments of elements from the different algebras. 
		Since the $\mathcal{A}_v$'s and the $\mathcal{B}_v$'s are both BM-independent, the joint moments of elements $a_j \in \mathcal{A}_{\ell(j)}$ for $j = 1$, \dots, $k$ viewed inside $\mathcal{A}$ must be the same as their joint moments when viewed inside $\mathcal{B}$.  Thus, since the $\mathcal{B}_v$'s are $G$-independent in $\mathcal{B}$, it follows that the $\mathcal{A}_v$'s are $G$-independent in $\mathcal{A}$.
	\end{proof}
	
	\begin{remark}[BF-independence] \label{rem: BF independence}
		The BF-independence construction from \cite{AKJW2013} is a special case of our construction in \S \ref{subsec: product space}.  BF-products of Hilbert spaces \cite[\S 2.2]{AKJW2013} are defined based on a poset $(V,\prec)$; the product space includes all summands of the form $H_{v_0}^\circ \otimes \dots \otimes H_{v_m}^\circ$ such that $v_j$ is comparable with but not equal to $v_{j+1}$.  Hence, this is the digraph product corresponding to the digraph $(V, (\prec) \cup (\succ))$ where $v \rightsquigarrow w$ if and only if $v \prec w$ or $w \prec v$.
		
		The partitions $\mathcal{NC}_k(\ell,G)$ are also related to the BF-ordered partitions in \cite[Definition 3.3]{AKJW2013}.  Specifically, if $\ell$ is a labeling, then a partition $\pi$, compatible with $\ell$, is BF-ordered if whenever $B$ is directly nested inside $B'$, then $\ell(B)$ is comparable with $\ell(B')$.  On the other hand, $\pi \in \mathcal{NC}_k(\ell,G)$ requires that $\ell(B) \rightsquigarrow \ell(B')$, or that $\ell(B)$ and $\ell(B')$ are comparable \emph{and not equal}.  Thus, $\mathcal{NC}(\ell,G)$ is the ``strict'' version of BF-ordered partitions where we forbid equality in the underlying relation.
		
		Unlike the work on BM-independence, the work on BF-independence did not define a general moment condition for independence, only a Hilbert space construction.  Hence, we do not know any analog of Definition \ref{def: BM independence} and Proposition \ref{prop: equivalent definitions of BM independence} for the BF case.  However, our work does provide a moment formula for the BF independence by specializing Theorem \ref{thm:momentformula} and Definition \ref{def: G free independence} to this case.
	\end{remark}
	
	\begin{remark}[Comparison with BMT independence] \label{rem: BMF versus BMT}
		Arizmendi, Mendoza, and Vazquez-Becerra \cite{OASRMJVB2025} have defined BMT independence which provides a mixture of Boolean, monotone, and tensor independence given by digraphs.  In the case where the digraph is a partial order, this extends the notion of BM independence from \cite{JW2007,JW20072} (see \cite[Theorem 3.9]{OASRMJVB2025}) and hence also agrees with the notion of digraph independence in the present work.  However, we caution that the direction of the edges is reversed; here because we use $E_k^{\dagger}$ in \S \ref{subsec: product space} (following \cite[Definition 3.18]{JekelLiu2020}), monotone independence of $n$ elements, for example, corresponds to the digraph $([n],<)$, while in \cite{OASRMJVB2025} monotone independence corresponds to the digraph $([n],>)$.  Correspondingly, \cite[Remark 3.8]{OASRMJVB2025} reverses the direction of edges when turning a partial order into a digraph.
		
		For a digraph $G = (V,E)$ where each pair of vertices has an edge in at most one direction (or each pair of algebras is Boolean, monotone, or anti-monotone independent), both our construction and the construction of \cite{OASRMJVB2025} make sense.  However, in general, they give \emph{different} notions of independence.  For example, let $G$ be a $3$-cycle with vertices $V = \{1,2,3\}$ and $E = \{\{1,2\},\{2,3\},\{3,1\}\}$, and let $G^\dagger$ be the reversed graph $(V,E^\dagger)$.  Let $(\cA_j,\phi_j)$ be $\mathrm{C}^*$-probability spaces for $j = 1, 2, 3$ and let $(\mathcal{H}_j,\xi_j)$ be the corresponding GNS space.  Let $(H,\xi)$ be the $G$-product space, let $\cA$ be the $\mathrm{C}^*$-probability space generated by $\iota_j(\cA_j)$, $j = 1$, $2$, $3$, and let $\phi$ be the state induced by $\xi$.  Suppose $a_j \in \cA_j$ be self-adjoint with $\phi_j(a_j) = 0$ and $\phi_j(a_j^2) = 1$.  Let $\widehat{a}_j = a_j \xi_j$ denote the vector in the GNS space corresponding to the operator $a_j$.  Note that $\mathcal{H}$ contains a summand $\mathcal{H}_3^\circ \otimes \mathcal{H}_2^\circ \otimes \mathcal{H}_1^\circ$ since $(3,2,1) \in E_2^\dagger$.  Therefore,
		\begin{align*}
			\phi(\iota_1(a_1) \iota_2(a_2) \iota_3(a_3) \iota_3(a_3) \iota_2(a_2) \iota_1(a_1)) &= \norm{\iota_3(a_3) \iota_2(a_2) \iota_1(a_1) \xi}_{\mathcal{H}}^2 \\
			&= \norm{\widehat{a}_3 \otimes \widehat{a}_2 \otimes \widehat{a}_1}_{\mathcal{H}_3^\circ \otimes \mathcal{H}_2^\circ \otimes \mathcal{H}_1^\circ}^2 \\
			&= \norm{\widehat{a}_3}_{\mathcal{H}_3^\circ}^2 \norm{\widehat{a}_2}_{\mathcal{H}_2^\circ}^2 \norm{\widehat{a}_1}_{\mathcal{H}_1^\circ}^2 \\
			&= \phi_3(a_3^2) \phi_2(a_2^2) \phi_1(a_1^2) \\
			&= 1.
		\end{align*}
		On the other hand, we claim that the analogous moment in the BMT construction is zero.  Applying \cite[Definition 4.1]{OASRMJVB2025} with the $3$-cycle digraph $G^{\dagger}$, we obtain the three inclusions $\pi_j: \cA_j \to B(\mathcal{H}_1 \otimes \mathcal{H}_2 \otimes \mathcal{H}_3)$:
		\begin{align*}
			\pi_1(a_1) &= a_1 \otimes P_2 \otimes 1 \\
			\pi_2(a_2) &= 1 \otimes a_2 \otimes P_3 \\
			\pi_3(a_3) &= P_1 \otimes 1 \otimes a_3,
		\end{align*}
		where $P_j \in B(H_j)$ is the rank-one projection onto $\xi_j$.  The state for BMT independence is given by the vector $\xi_1 \otimes \xi_2 \otimes \xi_3$.  Therefore,
		\begin{align*}
			\phi(\pi_1(a_1) \pi_2(a_2) \pi_3(a_3) \pi_3(a_3) \pi_2(a_2) \pi_1(a_1)) &= \norm{\pi_3(a_3) \pi_2(a_2) \pi_1(a_1) (\xi_1 \otimes \xi_2 \otimes \xi_3)}_{\mathcal{H}_1 \otimes \mathcal{H}_2 \otimes \mathcal{H}_3}^2 \\
			&= \norm{\pi_3(a_3) \pi_2(a_2)(a_1 \xi_1 \otimes \xi_2 \otimes \xi_3)}_{\mathcal{H}_1 \otimes \mathcal{H}_2 \otimes \mathcal{H}_3}^2 \\
			&= \norm{\pi_3(a_3)(a_1 \xi_1 \otimes a_2 \xi_2 \otimes \xi_3)}_{\mathcal{H}_1 \otimes \mathcal{H}_2 \otimes \mathcal{H}_3}^2 \\
			&= \norm{P_1 a_1 \xi_1 \otimes a_2 \xi_2 \otimes a_3 \xi_3}_{\mathcal{H}_1 \otimes \mathcal{H}_2 \otimes \mathcal{H}_3}^2 \\
			&= 0
		\end{align*}
		since $P_1 a_1 \xi_1 = \ip{\xi_1, a_1 \xi_1} \xi_1 = \phi_1(a_1) \xi_1 = 0$.
	\end{remark}
	
	\subsection{Relationship with general tree independence} \label{subsec: tree independence}
	
	Now let us explain the relationship with tree independence from \cite{JekelLiu2020} and \cite{JDW2021}, which is necessary since we will use results from \cite{JDW2021} later on.  Note that in \cite{JekelLiu2020} the construction was done in the $\cB$-valued setting, while here we only consider the scalar-valued setting where $\cB = \C$.
	
	Let $\mathcal{T}_{\operatorname{free},n}$ be the rooted tree described as follows.  The vertices are the alternating strings on the alphabet $[n]$, including the empty string.  The empty string is the root vertex of $\mathcal{T}_{\operatorname{free},n}$. Each vertex $j_m \ldots j_1$ in the tree has edges to the vertices $j j_m \dots j_1$ for $j \neq j_m$.
	
	Let $\mathcal{T}$ be a connected subtree of $\mathcal{T}_{\operatorname{free},n}$, and let $(\mathcal{H}_1,\xi_1)$, \dots, $(\mathcal{H}_n,\xi_n)$ be pointed Hilbert spaces.  Then define $\assemb_{\mathcal{T}}[(\mathcal{H}_1,\xi_1),\dots,(\mathcal{H}_n,\xi_n)]$ as the pair $(\mathcal{H},\xi)$ where
	\begin{equation} \label{eq: T free product}
		\mathcal{H} = \bigoplus_{m \geq 0} \bigoplus_{j_m \dots j_1 \in \mathcal{T}} \mathcal{H}_{j_m}^\circ \otimes \dots \otimes \mathcal{H}_{j_1}^\circ.
	\end{equation}
	Here, by convention, the only word of length zero is the empty string $\emptyset$, and the corresponding summand in $\mathcal{H}$ is $\C \xi$, which is the ``empty tensor product.''  This is a generalization of the construction we already explained for digraphs in \S \ref{subsec: product space} .  Indeed, if $G = (V,E)$ is a digraph on the vertex set $V = \{1,\ldots,n\}$, and we let
	\[
	\operatorname{Walk}(G) = \{j_m \dots j_1: m \geq 0, j_m \leftsquigarrow j_{m-1} \leftsquigarrow \dots \leftsquigarrow j_1\} = \{\varnothing\} \sqcup \bigsqcup_{m \geq 0} E_m^{\dagger}
	\text{ , }\]
	then, taking $\mathcal{T} = \operatorname{Walk}(G)$, the $\cT$-free product \eqref{eq: T free product} will reduce to $\mathcal{T}$-free product \eqref{eq: G free product}.  See also \cite[Definition 3.18]{JekelLiu2020}.
	
	In the general setting of $\cT$-free products, the inclusion maps $B(\mathcal{H}_j) \to B(\mathcal{H})$ are given as follows.  Let
	\begin{align*}
		S_j &= \{ j_m \dots j_1 \in \cT \text{ such that } j j_m \dots j_1 \in \cT \} \\
		S_j' &= \{j_m \dots j_1 \in \cT \text{ such that } j \neq j_m \text{ and } j j_m \dots j_1 \not \in \cT \}.
	\end{align*}
	Note that in the case $\cT = \Walk(G)$, then $S_j$ is the set of reverse walks such that the leftmost vertex $j_m \rightsquigarrow j$, and $S_j'$ is the set of reverse walks such that $j_m \neq j$ and $j_m$ is not $\rightsquigarrow j$.  Thus, the generalization of $\mathcal{H}_{\rightsquigarrow v}$ and $\mathcal{H}_{\perp v}$ are respectively
	\[
	\mathcal{H}_{S_j} = \bigoplus_{j_m \dots j_1 \in S_j} \mathcal{H}_{j_m}^\circ \otimes \dots \otimes \mathcal{H}_{j_1}^\circ \text{ and } \mathcal{H}_{S_j'} = \bigoplus_{j_m \dots j_1 \in S_j'} \mathcal{H}_{j_m}^\circ \otimes \dots \otimes \mathcal{H}_{j_1}^\circ.
	\]
	Then just as in \eqref{eq:UV}, we have a unitary isomorphism
	\[
	u_j: \mathcal{H} \xrightarrow{\cong} [\mathcal{H}_j \otimes \mathcal{H}_{S_j}] \oplus \mathcal{H}_{S_j'},
	\]
	and define the map $\iota_j: B(\mathcal{H}_j) \to B(\mathcal{H})$ by
	\[
	\iota_j(a) = u_j[[a \otimes \id_{\mathcal{H}_{S_j}}] \oplus 0_{\mathcal{H}_{S_j'}}]u_j^*.
	\]
	In the case $\cT = \operatorname{Walk}(G)$, this reduces to Definition \ref{def:inclusionmaps}.
	
	Next, we turn to the generalization of Theorem \ref{thm:momentformula}.  This will again express $\ip{\xi, \iota_{\ell(1)}(a_1) \dots \iota_{\ell(k)}(a_k) \xi}$, where $a_j \in B(\mathcal{H}_{\ell(j)})$ and $\ell(1)$, \dots, $\ell(k)$ is alternating, though a sum of Boolean cumulants indexed by partitions compatible with the given tree $\cT$.  Compatibility is described as follows.
	
	Given a labeling $\ell: [k] \to [n]$ and a compatible partition $\pi$ in the sense of Definition \ref{def: partition labeling compatible}, we say that $\pi$ and $\ell$ are \emph{compatible with $\cT$} if the following condition holds:  For each block $B \in \pi$, let $B_0 \rightsquigarrow B_1 \rightsquigarrow \dots \rightsquigarrow B_m = B$ be the unique path from a minimal block up/down to $B$ in the nesting forest $\mathrm{F}(\pi)$.  Then for every block $B$, we have $\ell(B_m) \dots \ell(B_0) \in \cT$.
	
	Another interpretation of this statement is as follows.  As in \cite[Definition 4.15]{JekelLiu2020}, we can make $\mathrm{F}(\pi)$ into an out-tree $\operatorname{graph}(\pi)$ by adding a new vertex $\varnothing$, which will be the root and have edges to all the minimal blocks in $\mathrm{F}(\pi)$.  This is analogous to the way that $\operatorname{Walk}(G)$ has the empty walk $\varnothing$ added as the root vertex.  Then compatibility of $\pi$, $\ell$, and $\cT$ means precisely that there is a digraph homomorphism $\phi: \operatorname{graph}(\pi) \to \cT$ preserving the root, such that for every block $B$, the first letter of $\phi(B)$ is $\ell(B)$; see also \cite[Remark 4.20]{JekelLiu2020}.  In the case where $\cT = \operatorname{Walk}(G)$, then after deleting the root vertex, we get a digraph homomorphism from $\mathrm{F}(\pi)$ to $\operatorname{Walk}(G) \setminus \{\varnothing\}$.  Now $\operatorname{Walk}(G) \setminus \{\varnothing\}$ is a union of $n$ branches, each branch representing the walks starting at a vertex $v \in [n]$; this construction is a version of the universal cover of a digraph where the walks are always directed.  Just like in the case of the universal cover for undirected graphs, homomorphisms from the $\operatorname{graph}(\pi)$ into $\operatorname{Walk}(G)$ correspond to homomorphisms $\mathrm{F}(\pi) \to G$.  Thus, homomorphisms from $\operatorname{graph}(\pi)$ into $\cT$ as in \cite[Remark 4.20]{JekelLiu2020} reduce in the case of $\cT = \operatorname{Walk}(G)$ to homomorphisms from $\mathrm{F}(\pi)$ into $G$ as in Definition \ref{def: partition labeling graph compatible}.
	
	\section{Convolution and limit theorems} \label{sec: limit theorem}
	
	\subsection{The compactly supported case} \label{subsec: limit theorem compact support}
	
	Given a digraph $G = (V,E)$ and compactly supported measures $(\mu_v)_{v \in V}$, we define the $G$-free convolution $\boxplus_G((\mu_v)_{v \in V})$ as follows.  Recall by the spectral theorem that any self-adjoint, or more generally normal, element in a $\mathrm{C}^*$-probability space has as well-defined spectral distribution with respec to the state (see e.g. \cite[Proposition 3.13]{NS2006}); moreover, any compactly supported $\mu \in \mathcal{P}(\R)$ can be realized as the spectral distribution of some self-adjoint element in a $\mathrm{C}^*$-probability space, for instance by taking $\cA = C(\operatorname{supp}(\mu))$, $\phi(f) = \int f\,d\mu$, and $x$ to be the identity function.
	
	Now fix $G$ and $(\mu_v)_{v \in V}$.  Let $(\mathcal{A}_v,\phi_v)$ be a $\mathrm{C}^*$-probability space and $x_v \in \mathcal{A}_v$ self-adjoint such that the spectral distribution of $x_v$ with respect to $\phi_v$ is $\mu_v$.  Let $(\mathcal{A},\phi)$ be the $G$-product of $(\mathcal{A}_v,\phi_v)_{v \in V}$ and let $\iota_v: \mathcal{A}_v \to \mathcal{A}$ the corresponding inclusion (see Definition \ref{def: G free product probability space}).  Then $\boxplus_G((\mu_v)_{v \in V})$ is defined to be the spectral distribution of $\sum_{v \in V} \iota_v(x_v)$ with respect to the state $\phi$.  For this to be well-defined, one should verify that the specific choice of $(\mathcal{A}_v,\phi_v)$ and $x_v$ does not affect the final result, so long as $x_v$ has the distribution $\mu_v$.  Since $\sum_{v \in V} \iota_v(x_v)$ is a bounded operator, its spectral distribution is uniquely determined by its moments.  Thus, it suffices to show that the moments of $\sum_{v \in V} \iota_v(x_v)$ are uniquely determined by the moments of $x_v$. This will follow from the next result, where we compute the moments of $x$ using Theorem \ref{thm:momentformula}.
	
	Given the role of Boolean cumulants in Theorem \ref{thm:momentformula}, we often use the Boolean cumulants of a single operator and of a probability measure, and hence we use the following notation.
	
	\begin{notation}~
		\begin{itemize}
			\item For a $\mathrm{C}^*$-probability space $(\cA,\phi)$ and $x \in \cA$, we write $\kappa_{\bool,k}(x) = K_{\bool,k}(x,\dots,x)$.        
			\item For $\mu \in \mathcal{P}(\R)$ compactly supported, we write $\kappa_{\bool,k}(\mu)$ for the $k$th Boolean cumulant of any self-adjoint $x$ in $(\cA,\phi)$ whose spectral distribution is $\mu$, which of course only depends on $\mu$. 
			\item Similarly, for a partition $\pi$ of $[k]$ and a compactly supported $\mu \in \mathcal{P}(\R)$, we write $\kappa_{\bool,\pi}(\mu) = \prod_{B \in \pi} \kappa_{\bool,|B|}(\mu)$.
			\item Moreover, for $\mu \in \mathcal{P}(\R)$ compactly supported, we denote by $m_k(\mu)$ the $k$th moment of $\mu$.
		\end{itemize}
	\end{notation}
	
	\begin{lemma} \label{lem: moments of sum}
		Let $G = (V,E)$ be a digraph.  Let $(\mathcal{A},\phi)$ be the $G$-product of $\mathrm{C}^*$-probability spaces $(\mathcal{A}_v,\phi_v)$.  Let $x_v \in \mathcal{A}_v$ be self-adjoint.  Let $x = \sum_{v \in V} \iota_v(x_v)$.  Then for $k \in \N$,
		\[
		\phi(x^k) = \sum_{\ell: [k] \to V}  \sum_{\pi \in \mathcal{NC}_k(\ell,G)} \prod_{B \in \pi} \kappa_{\bool,|B|}(x_{\ell(B)}).
		\]
		Here $\ell$ is required to be constant on each block $B$, and so $\ell(B)$ denotes the constant value on that block.
	\end{lemma}
	
	\begin{proof}
		Using multilinearity,
		\[
		\phi(x^k) = \phi \left( \left( \sum_{v \in V} \iota_v(x_v) \right)^k \right) = \sum_{\ell: [k] \to V} \phi(\iota_{\ell(1)}(x_{\ell(1)}) \dots \iota_{\ell(k)}(x_{\ell(k)})).
		\]
		By Theorem \ref{thm:momentformula}, this equals 
		\[
		\sum_{\ell: [k] \to V}  \sum_{\pi \in \mathcal{NC}_k(\ell,G))} \prod_{B \in \pi} K_{\bool,|B|}[x_{\ell(j)}: j \in B].
		\]
		Now $\ell$ must be constant on each block $B$ in the above expression and hence we can write $K_{\bool,|B|}[x_{\ell(j)}: j \in B]$ equivalently as the $|B|$th cumulant of $x_{\ell(B)}$.
	\end{proof}
	
	For our limit theorems, we focus on repeated convolutions of same measure.  For simplicity of notation, we denote by $\boxplus_G(\mu)$ the $G$-free convolution of $(\mu_v)_{v \in V}$ where all the $\mu_v$'s are equal to $\mu$.  The previous lemma implies the following.
	
	\begin{lemma} \label{lem: moments of convolution power}
		Let $\mu \in \mathcal{P}(\R)$ be compactly supported.  Let $G = (V,E)$ be a finite digraph.  Then we have
		\[
		m_k(\boxplus_G(\mu)) = \sum_{\pi \in \mathcal{NC}_k} |\Hom(\mathrm{F}(\pi),G)| \kappa_{\bool,\pi}(\mu).
		\]
	\end{lemma}
	
	\begin{proof}
		For each $v \in V$, let $(\mathcal{A}_v,\phi_v)$ be a $\mathrm{C}^*$-probability space and $x_v \in \mathcal{A}_v$ self-adjoint with distribution $\mu$; of course, one may take $\mathcal{A}_v = L^\infty(\R,\mu)$ with the state given by integration for every $v$, but for the purposes of notation, we want to distinguish the spaces for various $v$.  Let $x = \sum_{v \in V} \iota_v(x_v)$ in the $G$-product space as in the previous lemma. By the previous lemma,
		\begin{align*}
			m_k(\boxplus_{G}(\mu)) = \phi(x^k) &= \sum_{\ell: [k] \to V}  \sum_{\pi \in \mathcal{NC}_k(\ell,G))} \prod_{B \in \pi} \kappa_{\bool,|B|}(x_{\ell(B)}) \\
			&= \sum_{\ell: [k] \to V}  \sum_{\pi \in \mathcal{NC}_k(\ell,G))} \prod_{B \in \pi} \kappa_{\bool,|B|}(\mu). 
		\end{align*}
		Now we exchange the order of summation over $\ell$ and $\pi$ to get
		\begin{align*}
			m_k(\boxplus_{G}(\mu)) &= \sum_{\pi \in \mathcal{NC}_k} \sum_{\substack{ \ell: [k] \to V \\ \pi \in \mathcal{NC}_k(\ell,G)}} \prod_{B \in \pi} \kappa_{\bool,|B|}(\mu) \\
			&= \sum_{\pi \in \mathcal{NC}_k} |\{ \ell: [k] \to V, \pi \in \mathcal{NC}_k(\ell,G)\}| \, \prod_{B \in \pi} \kappa_{\bool,|B|}(\mu).
		\end{align*}
		Recall that $\pi \in \mathcal{NC}_k(\ell,G)$ if and only if $\ell$ defines a digraph homomorphism $\mathrm{F}(\pi) \to G$.  Therefore,
		\[
		|\{ \ell: [k] \to V | \\ \pi \in \mathcal{NC}_k(\ell,G)\}| = |\Hom(\mathrm{F}(\pi),G)|,
		\]
		which establishes the desired formula.
	\end{proof}
	
	The next lemma is the first step of Theorem \ref{thm:limit}.  In fact, it is a special case of the theorem when $\mu_n = \mu^{\uplus 1/|V_n|}$ and $\mu$ is compactly supported.
	
	\begin{lemma} \label{lem:cpctsuppconvergence}
		Suppose that $G_n = (V_n,E_n)$ is a sequence of digraphs such that for every finite out-tree $G' = (V',E')$,
		\[
		\lim_{n \to \infty} \frac{1}{|V_n|^{|V'|}} |\Hom(G',G_n)| = \beta_{G'}.
		\]
		For an out-forest $G'$ that is the disjoint union of out-forests $G_1', \ldots, G_k'$, let us write
		\[
		\beta_{G'} = \beta_{G_1'} \dots \beta_{G_k'}.
		\]
		Then for every compactly supported measure $\mu$, we have
		\begin{equation} \label{eq: limit moments}
			\lim_{n \to \infty} m_k(\boxplus_{G_n}(\mu^{\uplus 1/|V_n|})) = \sum_{\pi \in \mathcal{NC}_k} \beta_{\mathrm{F}(\pi)} \kappa_{\bool,\pi}(\mu).
		\end{equation}
		Moreover, denoting by $\rad(\mu)$ the radius of the support of the measure $\mu$, we have $\rad(\boxplus_{G_n}(\mu^{\uplus 1/|V_n|})) \leq 4 \rad(\mu)$.  Hence, $\lim_{n \to \infty} \boxplus_{G_n}(\mu^{\uplus 1 / |V_n|})$ exists in $\mathcal{P}(\R)$.
	\end{lemma}
	
	\begin{proof}
		First, note that $\mu^{\uplus 1/|V_n|}$ is a compactly supported probability measure.  Indeed, from \cite[equation (5)]{MBJW1998}, we see that its Cauchy transform is analytic in a neighborhood of $\infty$, which is equivalent to the measure having compact support.  Hence also $\boxplus_{G_n}(\mu^{\uplus 1/|V_n|})$ is a well-defined compactly supported probability measure.
		
		Now note that if $G'$ is an out-forest which is a disjoint union of out-trees $G_1'$, \dots, $G_k'$, then a digraph homomorphism from $G' \to G_n$ is equivalent to a $k$-tuple of digraph homomorphisms $G_i' \to G_n$ for $i = 1$, \dots, $k$, and thus
		\[
		|\Hom(G',G_n)| = |\Hom(G_1',G_n)| \dots |\Hom(G_k',G_n)|.
		\]
		Moreover, since $|V'| = |V_1'| + \dots + |V_k'|$, we have
		\[
		\frac{|\Hom(G',G_n)|}{|V_n|^{|V'|}} = \frac{|\Hom(G_1',G_n)|}{|V_n|^{|V_1'|}} \dots \frac{|\Hom(G_k',G_n)|}{|V_n|^{|V_k'|}}.
		\]
		Hence,
		\[
		\lim_{n \to \infty} \frac{|\Hom(G',G_n)|}{|V_n|^{|V'|}} = \beta_{G_1'} \dots \beta_{G_k'} = \beta_G.
		\]
		In other words, the hypothesis that we assumed to be true when $G'$ is an out-tree extends automatically to the case when $G'$ is an out-forest.
		
		By the previous lemma,
		\[
		m_k(\boxplus_{G_n}(\mu^{\uplus 1/|V_n|})) = \sum_{\pi \in \mathcal{NC}_k} |\Hom(\mathrm{F}(\pi),G_n)| \prod_{B \in \pi} \kappa_{\bool,|B|}(\mu^{\uplus 1/|V_n|}).
		\]
		By definition of the Boolean convolution powers, $\kappa_{\bool,|B|}(\mu^{\uplus 1/|V_n|}) = (1/|V_n|) \kappa_{\bool,|B|}(\mu)$.  Thus,
		\begin{align*}
			m_k(\boxplus_{G_n}(\mu^{\uplus 1/|V_n|})) &= \sum_{\pi \in \mathcal{NC}_k} |\Hom(\mathrm{F}(\pi),G_n)| \prod_{B \in \pi} \frac{1}{|V_n|} \kappa_{\bool,|B|}(\mu) \\
			&= \sum_{\pi \in \mathcal{NC}_k} \frac{|\Hom(\mathrm{F}(\pi),G_n)|}{|V_n|^{|\pi|}} \prod_{B \in \pi} \kappa_{\bool,|B|}(\mu).
		\end{align*}
		Here $|\pi|$ is the number of blocks in $\pi$, which is the same as the number of vertices in $\mathrm{F}(\pi)$.  By the foregoing argument, for each $\pi$, we have
		\[
		\lim_{n \to \infty} \frac{|\Hom(\mathrm{F}(\pi),G_n)|}{|V_n|^{|\pi|}} = \beta_{\mathrm{F}(\pi)}.
		\]
		Therefore, \eqref{eq: limit moments} holds.
		
		Finally, we prove our estimate on the support radius of $\boxplus_{G_n}(\mu^{\uplus 1/|V_n|})$.  First, by Lemma \ref{lem: Boolean cumulants}, one can see that
		\[
		|\kappa_{\bool,k}(\mu)| = |K_{\bool,k}(x_v,\dots,x_v)| \leq \norm{x_v}^k = \rad(\mu)^k,
		\]
		where $x_v$ is the operator of multiplication by $x$ in $\mathcal{A}_v = L^\infty(\R,\mu)$.  Thus, we estimate
		\begin{align*}
			|m_k(\boxplus_{G_n}(\mu^{\uplus 1/|V_n|}))| &\leq \sum_{\pi \in \mathcal{NC}_k} \frac{|\Hom(\mathrm{F}(\pi),G_n)|}{|V_n|^{|\pi|}} \prod_{B \in \pi} |\kappa_{\bool,|B|}(\mu)| \\
			&\leq \sum_{\pi \in \mathcal{NC}_k} \frac{|\Hom(\mathrm{F}(\pi),G_n)|}{|V_n|^{|\pi|}} \prod_{B \in \pi} \rad(\mu)^{|B|} \\
			&= \sum_{\pi \in \mathcal{NC}_k} \frac{|\Hom(\mathrm{F}(\pi),G_n)|}{|V_n|^{|\pi|}} \rad(\mu)^k.
		\end{align*}
		Since homomorphisms $\mathrm{F}(\pi) \to G_n$ are functions from $\pi$ to $V_n$, we have
		\[
		\frac{|\Hom(\mathrm{F}(\pi),G_n)|}{|V_n|^{|\pi|}} \leq 1.
		\]
		Moreover, the number of non-crossing partitions of $[k]$ is the $k$th Catalan number $C_k$, which satisfies $C_k \leq 4^k$.  Hence, overall $|m_k(\boxplus_{G_n}(\mu^{\uplus 1/|V_n|}))| \leq 4^k \rad(\mu)^k$.  Because this holds for all $k$, we conclude that $\rad(\boxplus_{G_n}(\mu^{\uplus 1/|V_n|})) \leq 4 \rad(\mu)$.
		
		Thus, the support radius of $\boxplus_{G_n}(\mu^{\uplus 1/|V_n|})$ is uniformly bounded for all $n$, and hence convergence in moments for this sequence is equivalent to convergence in $\mathcal{P}(\R)$.  This concludes the final claim of the lemma.
	\end{proof}
	
	\subsection{The general case} \label{subsec: limit general}
	
	In order to define the $G$-free convolution for probability measures $(\mu_v)_{v \in V}$ that do not necessarily have bounded support, we first want to express the convolution operation using complex-analytic transforms.  For a probability measure $\mu$ on $\R$, write its Cauchy-Stieltjes transform 
	\[
	G_\mu(z) = \int_{\R} \frac{1}{z - t} \,d\mu(t) \text{ for } z \in \C \setminus \R,
	\]
	and define 
	\[
	K_\mu(z) = z - 1/G_\mu(z).
	\]
	If $\mu$ is compactly supported, then the $K$-transform $K_\mu$ is related to Boolean cumulants $\kappa_{\bool,k}(\mu)$ by the formula
	\[
	K_\mu(z) = \sum_{k=1}^\infty z^{-(k-1)} \kappa_{\bool,k}(\mu);
	\]
	see \cite[end of \S 2]{SpW1997}.  In the case $G$-independence, we have the following result.
	
	\begin{proposition}[{\cite[Proposition 6.9]{JekelLiu2020}}] \label{prop: K transform equation}
		Let $G$ be a digraph on the vertex set $[n]$.  Let $\mu_1$, \dots, $\mu_n$ be compactly supported measures.  For each vertex $j$, let $\operatorname{Walk}(G,j)$ be the tree whose vertices are the empty walk and all reverse walks that start at $j$.  Let
		\[
		\nu_j = \boxplus_{\operatorname{Walk}(G,j)}(\mu_1,\dots,\mu_n)
		\]
		be the convolution of $\mu_1$, \dots, $\mu_n$ with respect to the tree $\operatorname{Walk}(G,j)$ as in \cite{JekelLiu2020}.  Then $\nu_1,\dots,\nu_n$ satisfy the equations
		\begin{equation} \label{eq: fixed point}
			K_{\nu_i}(z) = K_{\mu_i}\left( z - \sum_{j \leftsquigarrow i} K_{\nu_j}(z) \right),
		\end{equation}
		and we have
		\[
		K_{\boxplus_G(\mu_1,\dots,\mu_n)}(z) = \sum_{i=1}^n K_{\nu_i}(z).
		\]
	\end{proposition}
	
	The system \eqref{eq: fixed point} is a fixed-point equation for $(K_{\nu_1}(z),\dots,K_{\nu_n}(z))$, which suggests a way to extend the definition of $\boxplus_G(\mu_1,\dots,\mu_n)$ to general probability measures $\mu_1$, \dots, $\mu_n$.  It suffices to show that the solution to the fixed point equation exists, is unique, and depends continuously on the input measure.  This was done in \cite[Theorem 4.1]{JDW2021} in the more general setting of $\mathcal{T}$-free independence and convolution, using the Earle-Hamilton theorem.  This argument implies in particular that there is a unique $(\nu_1,\dots,\nu_n)$ satisfying \eqref{eq: fixed point}.  Therefore, the following definition is consistent with the compactly supported case.
	
	\begin{definition}[$G$-free convolution for general probability measures] \label{def: convolution general}
		Let $G$ be a digraph on vertex set $[n]$.  For $\mu_1$, \dots, $\mu_n \in \mathcal{P}(\R)$, define $\boxplus_G(\mu_1,\dots,\mu_n)$ to be $\nu_1 \uplus \dots \uplus \nu_n$, where $(\nu_1,\dots,\nu_n)$ are the unique probability measures satisfying \eqref{eq: fixed point}.  Moreover, in the case when the measures $\mu_j$ are the same, write $\boxplus_G(\mu) = \boxplus_G(\mu,\dots,\mu)$.
	\end{definition}
	
	Continuous dependence of the measures $\nu_1$, \dots, $\nu_n$ and consequently $\boxplus_G(\mu_1,\dots,\mu_n)$ upon the inputs $\mu_1$, \dots, $\mu_n$ also follows from \cite[Theorem 4.1]{JDW2021}.  In fact, there is a stronger equicontinuity result \cite[Theorem 6.2]{JDW2021} that we will need for the proof of Theorem \ref{thm:limit} in the case of measures with unbounded support.  Here we will use the \emph{L{\'e}vy distance} on $\mathcal{P}(\R)$ given by
	\[
	d_L(\mu,\nu) := \inf \Bigl\{\epsilon > 0: \mu((-\infty,x-\epsilon)) - \epsilon \leq \nu((-\infty,x)) \leq \mu((-\infty,x+\epsilon)) + \epsilon) \text{ for all } x \in \R \Bigr\}.
	\]
	The distance $d_L$ makes $\mathcal{P}(\R)$ into a complete metric space, and the induced topology is the same as the weak-$*$ topology from viewing $\mathcal{P}(\R)$ inside the dual of $C_0(\R)$ (see for instance \cite[Theorem 6.8]{Billingsley1999}).  Here we state \cite[Theorem 6.2]{JDW2021} specialized to $G$-free convolutions.
	
	\begin{proposition}[Uniform equicontinuity; {\cite[Theorem 6.2]{JDW2021}}] \label{prop:equicontinuity}
		Let $d_L$ be the L\'evy distance on $\mathcal{P}(\R)$.  For every $Y \subseteq \mathcal{P}(\R)$ compact and $\epsilon > 0$, there exists $\delta > 0$ such that for every digraph $G = (V,E)$ and every $\mu \in Y$ and $\nu \in \mathcal{P}(\R)$,
		\[
		d_L(\mu,\nu) < \delta \implies d_L(\boxplus_G(\mu^{\uplus 1/|V|}), \boxplus_G(\nu^{\uplus 1/|V|})) < \epsilon.
		\]
	\end{proposition}
	
	Now we can conclude the proof of the main theorem for probability measures with unbounded support.
	
	\begin{proof}[Proof of Theorem \ref{thm:limit}]
		Let $G_n = (V_n,E_n)$ be a sequence of digraphs such that $\lim_{n \to \infty} |V_n| = \infty$ and for every finite out-tree $G' = (V',E')$, the limit $\beta_{G'} = \lim_{n \to \infty} |\Hom(G',G)| / |V_n|^{|V'|}$ exists.  Let $(\mu_n)_{n \in \N}$ be a sequence of probability measures such that $\lim_{n \to \infty} \mu_n^{\uplus |V_n|} = \mu$, and write $\nu_n = \mu_n^{\uplus |V_n|}$.  In order to show that $(\boxplus_{G_k}(\mu_n))_{n \in \N}$ is a Cauchy sequence in $d_L$, fix $\epsilon > 0$.  Then $Y = \{\nu_n: n \in \N\} \cup \{\mu\}$ is compact.  By Proposition \ref{prop:equicontinuity}, there exists $\delta > 0$ such that for all $\nu$, we have
		\[
		d_L(\nu, \nu_n) < \delta \implies d_L(\boxplus_{G_n}(\nu^{\uplus 1/|V_n|}), \boxplus_{G_n}(\nu_n^{\uplus 1/|V_n|})) < \frac{\epsilon}{4},
		\]
		and the same holds with $\nu_n$ replaced by $\mu$.  Let $\sigma_R = \mu([-R,R])^{-1} \mu|_{[-R,R]}$.  By choosing $R$ sufficiently large, we can arrange that $d_L(\mu,\sigma_R) < \delta$.  For sufficiently large $n$, we also have $d_L(\nu_n,\sigma_R) < \delta$ as well.  Hence,
		\[
		d_L(\boxplus_{G_n}(\mu_n), \boxplus_{G_n}(\sigma_R^{\uplus 1/ |V_n|})) = d_L(\boxplus_{G_n})(\nu_n^{\uplus 1/|V_n|}), \boxplus_{G_n}(\sigma_R^{\uplus 1/ |V_n|})) < \frac{\epsilon}{4}.
		\]
		By Lemma \ref{lem:cpctsuppconvergence}, $\sigma_R' := \lim_{n \to \infty} \boxplus_{G_n}(\sigma_R^{\uplus 1/|V_n|})$ exists, and therefore for sufficiently large $n$,
		\[
		d_L(\boxplus_{G_n}(\sigma_R^{\uplus 1/ |V_n|}), \sigma_R') < \frac{\epsilon}{4}.
		\]
		Using the triangle inequality, for sufficiently large $n$ and $m$,
		\[
		d_L(\boxplus_{G_n}(\mu_n), \boxplus_{G_m}(\mu_m)) < \epsilon.
		\]
		Hence, $(\boxplus_{G_n}(\mu_n))_{n \in \N}$ is Cauchy in $d_L$ and hence converges to some limit $\mu'$.
		
		Similar reasoning shows that if $\epsilon$ and $\delta$ are as above and $d(\sigma_R,\mu) < \delta$, then
		\[
		d_L(\mu', \sigma_R') < \epsilon.
		\]
		Hence,
		\[
		\lim_{n \to \infty} \boxplus_{G_n}(\mu_n) = \mu' = \lim_{R \to \infty} \sigma_R'.
		\]
		Since $\sigma_R$ is the truncation of $\mu$, it only depends on $\mu$.  Moreover, $\sigma_R'$ given by Lemma \ref{lem:cpctsuppconvergence} only depends on $\mu$ and the coefficients $\beta_{G'}$.  Therefore, $\mu'$ only depends on $\mu$ and the coefficients $\beta_{G'}$.
	\end{proof}
	
	\section{Examples and applications} \label{sec: examples}
	
	In this section, we describe several classes of examples to which Theorem \ref{thm:limit} applies.
	
	\subsection{Digraphon limit method} \label{subsec: continuum limit method}
	
	In \cite{AKJW2010,JW2010,LOJW1,LOJW2,LOJW3}, certain limit theorems for BM independence associated to cones, which described the limiting behavior of finite index sets given as discretizations of a bounded region in a cone; several of the proofs evaluated moments as a sum over a certain families of partitions and then compared this to an integral using Riemann sum approximations.  Motivated by these results, we will now give a continuum limit method in the more general measure-theoretic context of digraphons, namely, Proposition \ref{prop: continuum limit}.  We caution that although the statement and proof of Proposition \ref{prop: continuum limit} are based on the same general ideas as earlier results in the BM settings, the proposition and its proof are not an exact generalization of them.  In \S \ref{subsec: BM limit theorem}, we give a more precise discussion of limit theorems for three families of cones.  We also remark that a similar limit theorem was given in \cite{AKJW2013} for BF independence, though the proof was not written in terms of discretizations per se; we leave further investigation of this case for future research.
	
	Let $(\Omega,\rho)$ be a complete probability measure space.  Let $\mathcal{E} \subseteq \Omega \times \Omega$ be measurable.  We view $\Omega$ as a vertex set and $\mathcal{E}$ as an edge set, so $(\Omega,\mathcal{E})$ is a \emph{digraphon}.  (Often digraphons are defined specifically with $\Omega = [0,1]$ and $\rho$ Lebesgue measure, but here we want the flexibility to use whatever measure space is convenient.)
	
	Now fix a digraph $G' = (V',E')$.  Note that homomorphisms $(V',E')$ to $(\Omega,\mathcal{E})$ can be described as functions $\omega: V' \to \Omega$ such that if $(v,w) \in E$, then $(\omega(v),\omega(w)) \in \mathcal{E}$.  Functions $V' \to \Omega$ may be identified with the Cartesian product $\Omega^{\times V'}$, and so we obtain
	\[
	\Hom(G', (\Omega,\mathcal{E})) = \{\omega \in \Omega^{\times V'}: (v,w) \in E \implies (\omega_v,\omega_w) \in \mathcal{E} \}.
	\]
	Recall that $\rho$ induces a probability measure $\rho^{\times V'}$ on the Cartesian product $\Omega^{\times V'}$ equipped with the product $\sigma$-algebra; $(\Omega^{\times V'},\rho^{\times V'})$ can also be completed to a complete measure space if desired.  Furthermore, the space of homomorphisms $\Hom(G', (\Omega,\mathcal{E}))$ is a measurable subset of $\Omega^{\times V'}$ since it can be expressed as
	\[
	\bigcap_{(v,w) \in E'} \{ \omega \in \Omega^{\times V'}: (\omega_v, \omega_w) \in \mathcal{E} \},
	\]
	and each of the sets in the intersection is a Cartesian product of measurable sets.  Moreover, the measure of $\Hom(G', (\Omega,\mathcal{E}))$ can be evaluated as
	\[
	\rho^{\times V'}(\Hom(G', (\Omega,\mathcal{E}))) = \int_{\Omega^{\times V'}} \prod_{(v,w) \in E'} \mathbbm{1}_\mathcal{E}(\omega_v, \omega_w)\,d\rho^{\times V'}(\omega),
	\]
	where $\omega = (\omega_v)_{v \in V'} \in \Omega^{\times V'}$.
	
	We will show that if $G_n$ is a sequence of graphs giving a discretization of $(\Omega,\mathcal{E})$, then the normalized count of homomorphisms from $G'$ to $G_n$ as in Theorem \ref{thm:limit} converges to $\rho^{\times V'}(\Hom(G', (\Omega,\mathcal{E})))$; see Proposition \ref{prop: continuum limit} above.
	
	We can relate finite digraphs and digraphons as follows.  Let $G = (V,E)$ be a finite digraph.  Let $(A_v)_{v \in V}$ be a partition of $\Omega$ into measureable sets with $\rho(A_v) = 1/|V|$ for all $v \in V$.  (For example, if $\Omega = [0,1]$ and $V = \{1,\dots,k\}$, we could take $A_j = [(j-1)/k,j/k)$.)  Let
	\begin{equation} \label{eq: E set}
		\tilde{\mathcal{E}} = \bigcup_{(v,w) \in E} A_v \times A_w \subseteq \Omega \times \Omega.
	\end{equation}
	Then we claim that
	\begin{equation} \label{eq: homomorphism discretization}
		\rho^{\times V'}(\Hom(G', (\Omega,\tilde{\mathcal{E}}))) = \frac{|\Hom(G',G)|}{|V|^{|V'|}}.
	\end{equation}
	To see this, suppose $\omega \in \Omega^{\times V'}$ and note there is a unique $\phi: V' \to V$ such that $\omega_v \in A_{\phi(v)}$ for each $v \in V$.  Moreover, $\omega \in \Hom(G',(\Omega,\tilde{\mathcal{E}}))$ if and only if $\phi \in \Hom(G',G)$.  Thus,
	\[
	\Hom(G', (\Omega,\tilde{\mathcal{E}})) = \bigsqcup_{\phi \in \Hom(G',G)} \prod_{v \in V'} A_{\phi(v)},
	\]
	where the product here is a Cartesian product.  Since $\prod_{v \in V'} A_{\phi(v)}$ has measure $1/|V|^{|V'|}$, we obtain \eqref{eq: homomorphism discretization}.
	
	\begin{fact} \label{obs: edge set difference}
		Let $(\Omega,\rho)$ be a probability measure space.  Let $\mathcal{E}, \mathcal{E}' \subseteq \Omega \times \Omega$ be measurable sets and $\mathcal{E} \Delta \mathcal{E}'$ their symmetric difference.  Let $G' = (V',E')$ be a finite digraph.  Then
		\begin{align*}
			|\rho^{\times V'}(\Hom(G',(\Omega,\mathcal{E}))) &- \rho^{\times V'}(\Hom(G',(\Omega,\mathcal{E}')))| \\
			&\leq
			\rho^{\times V'}(\Hom(G',(\Omega,\mathcal{E})) \Delta \Hom(G',(\Omega,\mathcal{E}'))) \\
			&\leq |E'| \rho^{\times 2}(\mathcal{E} \Delta \mathcal{E}').
		\end{align*}
	\end{fact}
	
	\begin{proof}
		The first inequality is immediate.  For the second, note that
		\[
		\rho^{\times V'}(\Hom(G',(\Omega,\mathcal{E})) \Delta \Hom(G',(\Omega,\mathcal{E}'))) = \norm*{\prod_{(v,w) \in E'} \mathbbm{1}_\mathcal{E}(v,w) - \prod_{(v,w) \in E'} \mathbbm{1}_{\mathcal{E}'}(v,w)}_{L^1(\rho^{\times V'})}.
		\]
		We swap out each $\mathbbm{1}_{\mathcal{E}}$ for $\mathbbm{1}_{\mathcal{E}'}$ one instance at a time.  Each swap produces an error of at most $\norm{\mathbbm{1}_\mathcal{E} - \mathbbm{1}_{\mathcal{E}'}}_{L^1(\rho^{\times 2})} = \rho^{\times 2}(\mathcal{E} \Delta \mathcal{E}')$ because the product of the other terms is zero or one.  Overall there are $|E'|$ swaps, and so the error is at most $|E'| \rho^{\times 2}(\mathcal{E} \Delta \mathcal{E}')$.
	\end{proof}
	
	\begin{proof}[Proof of Proposition \ref{prop: continuum limit}]
		The first claim follows from \eqref{eq: homomorphism discretization} and Fact \ref{obs: edge set difference}, and the second claim follows from the first claim and Theorem \ref{thm:limit}.
	\end{proof}
	
	\subsection{Limit theorems for BM-independence associated to cones} \label{subsec: BM limit theorem}
	
	In this subsection, we revisit some of the BM limit theorems from \cite{AKJW2010,JW2010,LOJW1} using Theorem \ref{thm:limit} and Proposition \ref{prop: continuum limit} from this paper.  While \cite{AKJW2010,JW2010} focused on the central limit theorem and \cite{LOJW1} focused on the Poisson limit theorem, we obtain from Theorem \ref{thm:limit} limit theorems for general sequences of measures for discretizations of three families of cones (positive orthant, light-cones, and real positive definite matrices).  We remark that the complex and quaternionic positive definite matrices were also studied in \cite{JW2010,LOJW2}, but in order to keep things technically simple we do not handle these cases here.\footnote{Complex and quaternionic positive definite matrices can be handled in the same way if one can extend Proposition \ref{prop: volume characteristic} on volume characteristic to these cases.  This has not been explicitly proved in the literature, but was stated in the complex case in \cite[Remark 3]{AKJW2010}.}
	
	First, we recall some terminology relating to convex cones.  We say $\Pi \subseteq \R^d$ is a \emph{convex cone} if it is closed under addition and positive scalar multiples.  We assume that $\Pi$ is \emph{closed} and that it is \emph{salient}, meaning that $\Pi \cap -\Pi = \{0\}$.  In this case, the relation $\preceq$ on $\R^d$ defined by setting $\xi \preceq \eta$ if and only if $\eta - \xi \in \Pi$ is a non-strict partial order (this of course also leads to a strict partial order $\prec$ as described in \S \ref{subsec: digraphs}).  We define the \emph{interval}
	\[
	[\xi,\eta] = \{\rho \in \R^d: \xi \preceq \rho \preceq \eta\}
	\]
	(which is nonempty if and only if $\xi \preceq \eta$).  We will be concerned especially with the three families of \emph{positive, symmetric} cones first studied in the context of BM-independence in \cite[\S 4, examples (1) - (3)]{JW2010}:
	\begin{itemize}
		\item the positive orthant $\R_+^d \subseteq \R^d$;
		\item the Lorentz light-cone $\Lambda_d^1 = \{ (t; \mathbf{x}) \in \R^{d+1}: t \geq \norm{x} \}$;
		\item the positive semidefinite matrices $M_d(\R)_+$ which is a subset of the space of symmetric matrices $\Sym_d(\R) \cong \R^{d(d+1)/2}$.
	\end{itemize}
	For background and classification of positive symmetric cones, see \cite{FK1994}.
	
	Given a salient closed convex cone $\Pi \subseteq \R^d$, one can obtain \emph{finite} posets by considering $I_\xi = [0, \xi] \cap \Z^d$ for $\xi \in \Pi$.  Then, as in \S \ref{subsec: G independence}, one can consider BM-independent random variables indexed by $I_\xi$, which by Proposition \ref{prop: equivalent definitions of BM independence} is equivalent to $I_\xi$-independent variables where we view $I_\xi$ is a digraph.  We will study the behavior of $\boxplus_{I_\xi}(\mu)$ as $\xi \xrightarrow{\Pi} \infty$, or as $\xi$ tends to infinity in the cone $\Pi$.  Here we recall that if $f$ is a function on the cone $\Pi$, we say that $f(\xi) \to L$ as $\xi \xrightarrow{\Pi} \infty$, if for every $\epsilon > 0$, there exists $\xi_0 \in \Pi$ such that for all $\xi \succeq \xi_0$ we have $|f(\xi) - L| < \epsilon$.  The meaning of $\xi \xrightarrow{\Pi} \infty$ in the specific cases of $\R_+^d$, $\Lambda_d^1$, and $M_d(\R)_+$ is explained in \cite[Definition 1.5]{LOJW1}.
	
	By Theorem \ref{thm:limit}, we need to study the limit as $\xi \xrightarrow{\Pi} \infty$ of $|\Hom(G',I_\xi)| / |I_\xi|^{|V'|}$ for a finite out-forest $G' = (V'E')$.  We remark that $\Hom(G',I_\xi)$ is equivalently the set of strict poset homomorphisms from $G'$ to $I_\xi$ by Fact \ref{obs: BM homomorphism equivalence}, and in the case that $G' = \mathrm{F}(\pi)$ for some non-crossing partition $\pi$, this is exactly the set of $I_\xi$ labelings that establish strict BM order on $\pi$; this is denoted by $\BMO(\pi;\xi)$ in \cite[Definition 4.1]{LOJW1}.  The limit of $|\BMO(\pi;\xi)| / |I_\xi|^{|\pi|}$ is described in \cite[Theorem 4.4, Corollary 4.5]{LOJW1}.  We will rederive this result here in two steps, first applying a variant of the continuum limit method, and then computing the volume of the limiting set explicitly using the volume characteristic of \cite{AKJW2010}.
	
	\begin{lemma} \label{lem: BM continuum limit}
		Let $\Pi$ be one of the cones as above.  Let $G' = (V',E')$ be a finite out-forest.  Let
		\[
		\Hom(G',[0,\xi]) = \{ \eta \in [0,\xi]^{\times V'}: v \rightsquigarrow w \text{ in } G' \implies \eta_v \prec \eta_w \text{ in } \Pi \} \subseteq (\R^d)^{\times V'}.
		\]
		Then
		\[
		\lim_{\xi \xrightarrow{\Pi} \infty} \left| \frac{|\Hom(G',I_\xi)|}{|I_\xi|^{|V'|}} - \frac{\vol(\Hom(G',[0,\xi]))}{\vol([0,\xi])^{|V'|}} \right| = 0.
		\]
	\end{lemma}
	
	The idea of this lemma is the same as Proposition \ref{prop: continuum limit}, but here we do not have a \emph{fixed} continuum limit, since the continuum object $\Hom(G',[0,\xi])$ also depends on $\xi$.  And of course, we are taking the limit as $\xi \xrightarrow{\Pi} \infty$ rather than only a limit as $N \to \infty$.  Thus, we must proceed carefully to define the discretized set and estimate the symmetric difference.  Here we will leave some details to the reader since \cite{LOJW1} already gave another argument for the limit in Lemma \ref{lem: BM continuum limit}.
	
	\begin{lemma} \label{lem: BM symmetric difference}
		Let $\Pi \subseteq \R^d$ be one of the cones above.  For $\xi \in \R^d$, let $Q_\xi$ be the unit cube $\prod_{j=1}^d [\xi_j-1/2,\xi_j+1/2]$.  Let
		\[
		A_\xi = \bigcup_{\eta \in I_\xi} Q_\eta
		\]
		Then
		\[
		\lim_{\xi \xrightarrow{\Pi} \infty} \frac{\vol(A_\xi \Delta [0,\xi])}{\vol([0,\xi])} = 0.
		\]
		Similarly, let
		\begin{align*}
			E_\xi &= \{(\eta_1,\eta_2) \in \Z^d \times \Z^d: 0 \preceq \eta_1 \prec \eta_2 \preceq \xi\} \\
			B_\xi &= \bigcup_{(\eta_1,\eta_2) \in E_\xi \cap \Z^{2d}} Q_{(\eta_1,\eta_2)}.
		\end{align*}
		Then
		\[
		\lim_{\xi \xrightarrow{\Pi} \infty} \frac{\vol(B_\xi \Delta E_\xi)}{\vol([0,\xi])^2} = 0.
		\]
	\end{lemma}
	
	\begin{proof}
		For the first claim, we note that $A_\xi \, \Delta \, [0,\xi]$ is contained in the union of the cubes that intersect the boundary $\partial [0,\xi]$.  Hence, in particular, letting $N_\delta(\partial [0,\xi])$ be the closed $\delta$-neighborhood of the boundary with respect to the $\ell^\infty$ metric on $\R^d$, we have
		\[
		\vol(A_\xi \Delta [0,\xi]) \leq \vol(N_{1/2}(\partial [0,\xi])),
		\]
		and so the claim reduces to proving that
		\[
		\lim_{\xi \xrightarrow{\Pi} \infty} \frac{\vol(N_{1/2}(\partial [0,\xi]))}{\vol[0,\xi]}.
		\]
		This can be proved by explicit estimates in each of the three cases of $\Pi$ under consideration here.  We leave the details to the reader.  For the second claim, one can similarly reduce to the showing that
		\[
		\lim_{\xi \xrightarrow{\Pi} \infty} \frac{\vol(N_{1/2}(\partial E_\xi))}{\vol[0,\xi]^2} = 0,
		\]
		and then perform direct estimates for each case of $\Pi$.
	\end{proof}
	
	Lemma \ref{lem: BM continuum limit} follows from Lemma \ref{lem: BM symmetric difference} by similar reasoning as we used in Fact \ref{obs: edge set difference}.
	
	It remains to compute the volume of $|\Hom(G',[0,\xi])|$ appearing in Lemma \ref{lem: BM continuum limit}.  This computation drastically simplifies due to the special geometric structure of the cones under consideration, as shown by Kula and the third author in  \cite{AKJW2010}.
	
	\begin{proposition}[Existence of volume characteristic; {\cite[Theorem 2]{AKJW2010}}] \label{prop: volume characteristic}
		For each of the positive symmetric cones $\Pi$ we consider (namely $\R_+^d$, $\Lambda_d^1$, and $M_d(\R)_+$)\footnote{The same also holds for $M_d(\C)$ by \cite[Remark 3]{AKJW2010}, but the proof is not given in detail.} there exists a sequence $\displaystyle (\gamma_n(\Pi))_{n\geq 1}$ such that for any $\xi\in \Pi$ and any $n \in \N$
		\[
		\int_{\rho\in [0, \xi]} \vol([0,\rho])^{n-1} d(\rho) = \gamma_n(\Pi) \vol([0,\xi])^n.
		\]
	\end{proposition}
	
	The sequence $\gamma_n(\Pi)$ is called the \emph{volume characteristic sequence} for the cone $\Pi$.  It allows for a recursive computation of the volume of $\Hom(G',[0,\xi])$, as described in \cite[Theorem 4.4, Corollary 4.5]{LOJW1}.  Here we express the result of the computation explicitly as a product rather than giving a recursive description as in \cite{LOJW1}.
	
	\begin{lemma} \label{lem: BM homomorphism volume}
		Let $G' = (V',E')$ be a finite out-forest, let $\prec$ denote the strict partial order obtained as the transitive closure of $E'$ as a relation on $V'$, and let $\preceq$ be the corresponding non-strict partial order.  For the cones $\Pi$ under consideration and $\xi \in \Pi$, we have
		\[
		\vol(\Hom(G',[0,\xi])) = \vol([0,\xi])^{|V'|} \prod_{v \in V'} \gamma_{k(v)},
		\]
		where $k(v) = |\{w \in V': w \succeq v \}|$.
	\end{lemma}
	
	\begin{proof}
		We proceed by induction on $V'$.  If $|V'| = 1$, then both sides are equal to $\vol[0,\xi]$.
		
		Next, suppose that $|V'| > 1$, and suppose that $G'$ has more than one connected component (here components are defined by forgetting the orientation of the edges).  Write $G'$ as the disjoint union of components $G_1'$, \dots, $G_k'$.  Then
		\[
		\Hom(G',[0,\xi]) \cong \Hom(G_1',[0,\xi]) \times \dots \times \Hom(G_k',[0,\xi]).
		\]
		By applying the induction hypothesis to $G_j'$, we get
		\[
		\vol(\Hom(G',[0,\xi])) = \prod_{j=1}^k \left( \vol([0,\xi])^{|V_j'|} \prod_{v \in V_j'} \gamma_{k(v)} \right) = \vol([0,\xi])^{|V'|} \prod_{v \in V_j'} \gamma_{k(v)}.
		\]
		
		Finally, suppose that $|V'| > 1$ and that $G'$ has only one component, i.e.\ $G'$ is a out-tree.  Let $r$ be the root vertex (see \S \ref{subsec: digraphs}) and let $v_1$, \dots, $v_k$ be its neighbors, so that $r \rightsquigarrow v_j$.  Let $G_j'$ be the subtree under $v_j$, or the out-tree with vertex set $\{w: w \succeq v_j\}$.  Observe that
		\[
		\Hom(G',[0,\xi]) = \{ (\eta,\eta_1,\dots,\eta_k): \eta \in [0,\xi], \eta_j \in \Hom(G_j', (\eta,\xi]) \},
		\]
		where $(\eta,\xi] = \{\zeta: \eta \prec \zeta \preceq \xi\}$; this follows by first choosing the point $\eta$ where the root $r$ is mapped and then restricting the homomorphism to each of the subtrees $G_j$.  It follows from the Fubini--Tonelli theorem that
		\[
		\vol(\Hom(G',[0,\xi])) = \int_{[0,\xi]} \prod_{j=1}^k \vol(\Hom(G_j', (\eta,\xi]))\,d\eta.
		\]
		By ignoring the boundary, we can use $[\eta,\xi]$ instead of $(\eta,\xi]$.  Now perform the change of variables $\eta \mapsto \xi - \eta$ to obtain
		\[
		\vol(\Hom(G',[0,\xi])) = \int_{[0,\xi]} \prod_{j=1}^k \vol(\Hom(G_j', [0,\eta])) \,d\eta.
		\]
		Applying the induction hypothesis to $G_j'$, we obtain
		\begin{align*}
			\vol(\Hom(G',[0,\xi])) &= \int_{[0,\xi]} \prod_{j=1}^k \left( \vol([0,\eta])^{|V_j'|} \prod_{v \in V_j'} \gamma_{k(v)} \right) \,d\eta \\
			&= \int_{[0,\xi]} \vol([0,\eta])^{|V'|-1} \,d\eta \cdot \prod_{v \in V' \setminus \{r\}} \gamma_{k(v)}.
		\end{align*}
		By Proposition \ref{prop: volume characteristic},
		\begin{align*}
			\vol(\Hom(G',[0,\xi])) &= \vol([0,\xi])^{|V'|} \gamma_{|V'|}  \prod_{v \in V' \setminus \{r\}} \gamma_{k(v)} \\
			&= \vol([0,\xi])^{|V'|} \prod_{v \in V'} \gamma_{k(v)}
		\end{align*}
		since $k(r) = |V'|$.
	\end{proof}
	
	\begin{remark}
		Note that analogous computations in \cite{LOJW1} are written in terms of the partition $\pi$ rather than the out-forest $G'$, and thus correspond to taking $G' = \mathrm{F}(\pi)$.  The case of several connected components $G_1'$, \dots, $G_k'$ corresponds to when $\pi$ is the disjoint union or concatenation of partitions $\pi_1$, \dots, $\pi_k$.  Similarly, if $G'$ has only one component and we look at the branches $G_j'$, this corresponds to taking a partition $\pi$ with only one minimal block $B$ and looking at the subpartitions $\pi_1$, \dots, $\pi_k$ in between consecutive elements of the block $B$.
	\end{remark}
	
	Putting together Lemmas \ref{lem: BM continuum limit} and \ref{lem: BM homomorphism volume} with Theorem \ref{thm:limit}, we obtain the following result.
	
	\begin{theorem}[BM limit theorems for positive symmetric cones] \label{thm: BM limit theorem}
		Let $\Pi$ be one of the cones $\R_+^d$, $\Lambda_d^1$, or $M_d(\R)_+$, and let $I_\xi$ be as above.  For a finite out-forest $G' = (V',E')$, we have
		\[
		\lim_{\xi \xrightarrow{\Pi} \infty} \frac{|\Hom(G',I_\xi)|}{|I_\xi|^{|V'|}} = \prod_{v \in V'} \gamma_{k(v)}.
		\]
		In particular, by Theorem \ref{thm:limit}, if $\mu_\xi$ is a family of probability measures such that $\lim_{\xi \xrightarrow{\Pi} \infty} \mu_\xi^{\uplus |I_\xi|} = \mu$, then $\widehat{\mu} = \lim_{\xi \xrightarrow{\Pi} \infty} \boxplus_{I_\xi}(\mu)$ exists.  Moreover, in light of Lemma \ref{lem:cpctsuppconvergence}, if $\mu$ is compactly supported, then so is $\widehat{\mu}$, and
		\[
		m_k(\widehat{\mu}) = \sum_{\pi \in \mathcal{NC}_k} \prod_{v \in \mathrm{F}(\pi)} \gamma_{k(v)} \prod_{B \in \pi} \kappa_{\bool,|B|}(\mu).
		\]
	\end{theorem}
	
	\begin{example}[BM central limit theorems]
		To obtain BM central limit theorems for positive symmetric cones, we must plug in the Boolean central limit distribution $(1/2)(\delta_{-1} + \delta_1)$ for $\mu$ into Theorem \ref{thm: BM limit theorem}.  Thus, the Boolean cumulants of $\mu$ are all zero except for the second cumulant which is one.  Hence, the central limit distribution satisfies
		\[
		m_k(\widehat{\mu}) = \sum_{\pi \in \mathcal{NC}_k^{(2)}} \prod_{v \in \mathrm{F}(\pi)} \gamma_{k(v)},
		\]
		where $\mathcal{NC}_k^{(2)}$ is the set of non-crossing \emph{pair} partitions of $[k]$.  Note that this agrees with the moment formula for central limit measures associated to positive cones from \cite[Theorem 7]{AKJW2010} in light of \cite[\S 8, item (1)]{AKJW2010}, since for a pair partition, the number of blocks nested inside some block is the same as half of the number of indices inside it\footnote{Note that \cite[Theorem 7]{AKJW2010} lists the cones $\Lambda_d^1$, $M_d(\R)_+$, and $M_d(\C)_+$ while we focus on $\R_+^d$, $\Lambda_d^1$, and $M_d(\R)_+$.  The reasoning in \cite{AKJW2010} applies equally well to $\R_+^d$.  Similarly, by \cite[Remark 3]{AKJW2010}, our result can be applied to $M_d(\C)_+$.}
	\end{example}
	
	\begin{example}[BM Poisson limit theorems] \label{ex: BM LSN}
		The law of small numbers or Poisson limit theorem for BM independence studied in \cite{LOJW1} is a special case of Theorem \ref{thm: BM limit theorem}.  For the Poisson limit theorem, we must plug in for $\mu$ the Boolean analog of the Poisson distribution, which turns out to be $\frac{1}{1 + \lambda} \delta_0 + \frac{\lambda}{1 + \lambda} \delta_{1 + \lambda}$.  Indeed, one can show by direct computation of $K$-transforms that
		\[
		[(1 - \lambda/n) \delta_0 + (\lambda/n) \delta_1]^{\uplus n} \to \frac{1}{1 + \lambda} \delta_0 + \frac{\lambda}{1 + \lambda} \delta_{1 + \lambda};
		\]
		and that the Boolean cumulants of $\frac{1}{1 + \lambda} \delta_0 + \frac{\lambda}{1 + \lambda} \delta_{1 + \lambda}$ are all equal to $\lambda$.  By Theorem \ref{thm: BM limit theorem}, if $\mu_\xi$ is some measure with $\mu_\xi^{\uplus |I_\xi|} \to \mu$, then we have $\boxplus_{I_\xi}(\mu_\xi) \to \widehat{\mu}$ where
		\[
		m_k(\widehat{\mu}) = \sum_{\pi \in \mathcal{NC}_k} \left( \prod_{v \in \mathrm{F}(\pi)} \gamma_{k(v)} \right) \lambda^{|\pi|}.
		\]
		This is the same result as \cite[Theorem 4.4]{LOJW1} up to some technical differences.
		
		Specifically, \cite[Theorem 4.4]{LOJW1} allowed the convolution of several different measures, rather than only copies of the same measure.  Our result Theorem \ref{thm:limit} could similarly be generalized to consider $\boxplus_{G_n}(\mu_{n,1},\dots,\mu_{n,|V_n|})$ where $\lim_{n \to \infty} \sup_{j=1,\dots,n} d_L(\mu_{n,j},\mu) = 0$ after generalizing \cite[Theorem 6.2]{JDW2021} to allow several different input measures.  However, that is beyond the scope of this work.
		
		Note also that the hypotheses and conclusion of \cite[Theorem 4.4]{LOJW1} use convergence of moments rather than weak-$*$ convergence of measures; these are equivalent for measure supported on a fixed compact set, but neither type of convergence implies the other in general.
	\end{example}

	
	
	\subsection{Iterated composition of digraphs} \label{subsec: iterated composition}
	
	Another motivating case of the continuum limit method is the setting of iterated composition of digraphs studied in \cite{JekelLiu2020} (which of course also worked in the more general setting of tree independences).
	
	First, we recall from \cite[\S 5.5]{JekelLiu2020} the composition operation on digraphs.  Let $\Digraph(n)$ be the set of directed graphs on $[n]$.  Let $G \in \Digraph(k)$ and let $G_j \in \Digraph(n_j)$ for $j = 1$, \dots, $k$.  Let $N = n_1 + \dots + n_k$ and let $\iota_j: [n_j] \to [N]$ be the inclusion $\iota_j(i) = n_1 + \dots + n_{j-1} + i$.  Then $G(G_1,\dots,G_k)$ is the digraph $G'$ on vertex set $[N]$ described by
	\[
	E' = \{ (\iota_j(v),\iota_j(w)): v \rightsquigarrow w \text{ in } G_j \} \cup \{(\iota_i(v),\iota_j(w)): i \rightsquigarrow j \text{ in } G, v \in V_i, w \in V_j \}.
	\]
	In other words, we create disjoint copies of $G_1$, \dots, $G_k$, and then whenever $i \rightsquigarrow j$ in $G$ we add edges from every vertex in $G_i$ to every vertex in $G_j$.  This composition operation defines a (symmetric) operad structure (see \cite{Leinster2004} for background on operads).
	
	We focus here on iterated compositions of a fixed graph $G \in \Digraph(n)$.  Define inductively $G^{\circ k}$ by $G^{\circ 1} = G$ and $G^{\circ(k+1)} = G(G^{\circ k},\dots,G^{\circ k})$.  Limit theorems for such iterated compositions of the same digraph are given in \cite{JekelLiu2020} and \cite{JDW2021}.  The idea is essentially a continuum limit construction, where the limiting measure space is an infinite product, and the finite approximants are given by cylinder sets.
	
	As motivation, let us describe the edge structure in $G^{\circ k}$, starting with $G^{\circ 2}$.  The vertex set of $G^{\circ 2}$ is $[n]^2$, which we view as $[n] \times [n]$, where the first coordinate describes the position in the \emph{outer} graph in the composition (i.e.\ \emph{which} of the $n$ copies of $G$ you are in), and the second coordinate describes the position in the \emph{inner} graph in the composition.  Then $(i_1,i_2) \rightsquigarrow (j_1,j_2)$ if and only if either $i_1 \rightsquigarrow i_2$ in the outer graph, or $i_1 = i_2$ and $j_1 \rightsquigarrow j_2$ in the inner graph.  Similarly, the vertex set of $G^{\circ k}$ can be described as $[n]^k$ where the first coordinate corresponds to the outermost graph and the last coordinate corresponds to the innermost graph in the composition.  To determine when $(i_1,\dots,i_k) \rightsquigarrow (j_1,\dots,j_k)$, one looks at the first coordinate where $i_t \neq j_t$ and then checks whether $i_t \rightsquigarrow j_t$ in the graph at the $t$th innermost level of the composition.
	
	Hence, to study the limit as $k \to \infty$, we use a digraphon on the infinite product space $\Omega = [n]^{\times \N}$.  Moreover, let $\rho$ be the infinite product of the uniform probability measure on $[n]$, which is a Radon measure on $\Omega$.  Let $\mathcal{E}$ be the set of pairs $(\mathbf{i},\mathbf{j}) \in \Omega \times \Omega$ such that if $t$ is the first index where $i_t \neq j_t$, then $i_t \rightsquigarrow j_t$ in $G$.  Letting $E_k$ be the edge set of $G^{\circ k}$, we view $E_k \times \Omega$ as a subset of $\Omega$, where $E_k$ determines the values of the first $k$ coordinates.  Then $E_{k+1} \times \Omega \subseteq E_k \times \Omega$, and we have
	\[
	\mathcal{E} = \bigcap_{k=1}^\infty E_k \times \Omega.
	\]
	Hence, by continuity of the measure,
	\[
	\lim_{k \to \infty} \rho^{\times 2}(\mathcal{E} \Delta (E_k \times \Omega))=0.
	\]
	Therefore, by Proposition \ref{prop: continuum limit}, we have the following result.
	
	\begin{theorem}[Limit theorem for iterated composition compare {\cite[Theorem 8.6]{JekelLiu2020}, \cite[Theorem 6.1]{JDW2021}}]
		Fix $G \in \Digraph(n)$, and let $G^{\circ k}$ be its $k$-fold iterated composition.  Let $\Omega = [n]^{\N}$ and let $\mathcal{E}$ be the edge set described above. 
		Let $G'$ be an out-forest.  Then
		\[
		\lim_{k \to \infty} \frac{|\Hom(G',G^{\circ k})|}{n^{k|V'|}} = \rho^{\times V'}(\Hom(G',(\Omega,\mathcal{E})).
		\]
		Therefore, if $\mu_k \in \mathcal{P}(\R)$ such that $\mu_k^{\uplus n^k} \to \mu$, then the limit $\widehat{\mu} = \lim_{k \to \infty} \boxplus_{G^{\circ k}}(\mu_k)$ exists.
	\end{theorem}
	
	\begin{remark}
		Although this result is contained in \cite[Theorem 6.1]{JDW2021}, the proof used here is different.  In \cite{JDW2021}, the proof is not based on a general result such as Theorem \ref{thm:limit}, but rather directly showing the sequence of measures is Cauchy using the uniform continuity estimates for convolution, and there is no hope of generalizing this technique to the setting of Theorem \ref{thm:limit}.  Our proof here goes by way of Theorem \ref{thm:limit} which relies on the moment formulas Theorem \ref{thm:momentformula}, which also gives information about the moments of $\widehat{\mu}$ in the compactly supported case.  Further, we remark that although the continuum limit construction and moment computations used here overlap with \cite{JekelLiu2020}, we avoid the cumulant machinery of \cite[\S 7]{JekelLiu2020}.
	\end{remark}
	
	\subsection{Multi-regular digraphs} \label{subsec: multiregular}
	
	In \cite{JekelLiu2020}, it was shown that for regular digraphs the central limit distribution (under iterated composition) only depends on the number of vertices and the degree.  For digraphs, ``regular'' means in this paper that the out-degree of each vertex is the same.  Here we will generalize this result to sequences of \emph{multi-regular} digraphs.
	
	Fix $n \in \N$.  Let $G_n = (V_n,E_n)$ be a digraph, and assume that
	\begin{equation} \label{eq: multiregular graph 1}
		V_n = \bigsqcup_{j=1}^m V_{n,j},
	\end{equation}
	and further that for $i, j \in [n]$, there is a constant $A_{n,i,j}$ such that
	\begin{equation} \label{eq: multiregular graph 2}
		\text{for } v \in V_{n,i}, \quad |\{w \in V_{n,j}: v \rightsquigarrow w\}| = A_{n,i,j},
	\end{equation}
	that is, each vertex in $V_{n,i}$ has $A_{n,i,j}$-many edges into $V_{n,j}$.
	
	Let $G' = (V',E')$ be an out-tree.  In order to compute $|\Hom(G',G_n)|$, we partition the set of homomorphisms based on which set $V_{n,j}$ contains the image of each vertex in $V'$.  More precisely, given a label function $\ell: V' \to [m]$, let $\Hom_\ell(G',G_n)$ be the set of $\phi \in \Hom(G',G_n)$ such that $\phi(v) \in G_{n,\ell(v)}$ for all $v \in V'$.  Then $|\Hom_\ell(G',G_n)|$ can be computed by counting the number of choices for where to map each vertex of $G'$ iteratively:  For the root vertex $r$, there $|V_{n,\ell(r)}|$ choices for $\phi(r)$.  For any non-root vertex $v$, let $v_-$ be its predecessor, i.e., the unique vertex with $v_- \rightsquigarrow v$.  Assuming that $\phi(v_-)$ has already been chosen, then $v$ must be mapped to some vertex in $V_{n,\ell(v)}$ which has an ingoing edge from $\phi(v_-) \in V_{n,\ell(v_-)}$, and hence there are $A_{n,\ell(v_-),\ell(v)}$ choices for $\phi(v)$.  Therefore, we have
	\[
	|\Hom_\ell(G',G_n)| = |V_{n,\ell(r)}| \prod_{v \in V' \setminus \{r\}} A_{n,\ell(v_-),\ell(v)}.
	\]
	This argument can be formalized as an induction on $|V'|$ where the inductive step considers removing one leaf from $G'$.  Now summing over $\ell: V' \to [n]$, we obtain
	\begin{equation} \label{eq: multiregular homomorphism}
		|\Hom(G',G_n)| = \sum_{\ell: V' \to [n]} |V_{n,\ell(r)}| \prod_{v \in V' \setminus \{r\}} A_{n,\ell(v_-),\ell(v)}.
	\end{equation}
	We remark that the output of this formula only depends on $|V_{n,j}|$'s and the $A_{n,i,j}$'s, and in particular any multi-regular digraph $\tilde{G}_n$ with these same constants will produce the same number of homomorphisms and hence satisfy $\boxplus_{\tilde{G}_n}(\mu) = \boxplus_{G_n}(\mu)$.
	
	Next, we consider limits as $n \to \infty$.  Assume that
	\begin{equation} \label{eq: multiregular limit}
		\lim_{n \to \infty} \frac{|V_{n,j}|}{|V_n|} = t_j > 0, \qquad \lim_{n \to \infty} \frac{A_{n,i,j}}{|V_n|} = a_{i,j}.
	\end{equation}
	Then we obtain from \eqref{eq: multiregular homomorphism} that
	\[
	\frac{|\Hom(G',G_n)|}{|V_n|^{|V'|}} = \sum_{\ell: V' \to [n]} \frac{|V_{n,\ell(r)}|}{|V_n|} \prod_{v \in V' \setminus \{r\}} \frac{A_{n,\ell(v_-),\ell(v)}}{|V_n|},
	\]
	hence we obtain the following result:
	
	\begin{proposition}[Limit theorem for multiregular digraphs] \label{prop: multiregular 1}
		Let $G_n = (V_n,E_n)$ be a multiregular graph satisfying \eqref{eq: multiregular graph 1} and \eqref{eq: multiregular graph 2} with respect to coefficients $A_{n,i,j}$, such that the limiting conditions \eqref{eq: multiregular limit} hold.  Then
		\begin{equation}
			\lim_{n \to \infty} \frac{|\Hom(G',G_n)|}{|V_n|^{|V'|}} = \sum_{\ell: V' \to [m]} t_{\ell(r)} \prod_{v \in V' \setminus \{r\}} a_{\ell(v_-),\ell(v)} =: \beta_{G'}.
		\end{equation}
		Hence, by Theorem \ref{thm:limit}, if $\mu_n \in \mathcal{P}(\R)$ and $\mu_n^{\uplus |V_n|} \to \mu$, then $\boxplus_{G_n}(\mu_n)$ converges.
	\end{proposition}
	
	Next, let us describe how to compute the measure $\widehat{\mu} = \lim_{n \to \infty} \boxplus_{G_n}(\mu_n)$ in this situation.  We start with the fixed point equations in Proposition \ref{prop: K transform equation} for the $K$-transforms.  Fix $n$.  Then Proposition \ref{prop: K transform equation} gives a system of equations for $K_{\nu_v}$ where $\nu_v$ is the convolution with respect to $\Walk(G_n,v)$ in the sense of \cite{JekelLiu2020}.  Since the graph is multiregular, the isomorphism class of $\Walk(G_n,v)$ is the same for all vertices $v$ in the same part $V_{n,j}$ of our partition.  We denote by $\nu_{n,j}$ the common value of $\boxplus_{\Walk(G_n,v)}(\mu)$ for $v \in V_{n,j}$.  Then Proposition \ref{prop: K transform equation} yields
	\begin{align*}
		K_{\nu_{n,i}}(z) &= K_{\mu_n}\left(z - \sum_{i=1}^n A_{n,i,j} K_{\nu_{n,j}}(z) \right) \\
		K_{\boxplus_{G_n}(\mu_n)}(z) &= \sum_{i=1}^n |V_{n,i}| K_{\nu_{n,i}}.
	\end{align*}
	Since we assume that $\mu_n^{\uplus 1/|V_n|}$ converges, we renormalize these equations as follows:
	\begin{align*}
		|V_n| K_{\nu_{n,i}}(z) &= |V_n| K_{\mu_n}\left(z - \sum_{j=1}^n \frac{A_{n,i,j}}{|V_n|} \cdot |V_n| K_{\nu_{n,j}}(z) \right) \\
		K_{\boxplus_{G_n}(\mu_n)}(z) &= \sum_{j=1}^n \frac{|V_{n,i}|}{|V_n|} \cdot |V_n| K_{\nu_{n,i}}.
	\end{align*}
	By assumption, $|V_n| K_{\mu_n} \to K_\mu$.  We will prove below that $|V_n| K_{\nu_{n,j}}(z)$ converges to some $K_{\nu_j}$, and that we can take the limit of the above equations.
	
	\begin{proposition}[Limit theorem for multiregular digraphs 2] \label{prop: multiregular 2}
		For $n\in\N$ let $G_n = (V_n,E_n)$ be a multiregular graph and assume \eqref{eq: multiregular graph 1}, \eqref{eq: multiregular graph 2}, \eqref{eq: multiregular limit}.  Let $\mu_n$ be a sequence of probability measures such that $\mu_n^{\uplus |V_n|} \to \mu$.  Let $\widehat{\mu} = \lim_{n \to \infty} \boxplus_{G_n}(\mu_n)$.  Then there exist unique probability measures $(\nu_i)_{i=1}^m$ satisfying
		\begin{equation} \label{eq: multiregular limit fixed point}
			K_{\nu_i}(z) = K_{\mu}\left(z - \sum_{j=1}^m a_{i,j} K_{\nu_j}(z) \right),
		\end{equation}
		and in fact $\nu_i = \lim_{n \to \infty} \nu_{n,i}^{\uplus |V_n|}$ where $\nu_i$ is as above.  Then the measure $\widehat{\mu}$ is given by
		\[
		K_{\widehat{\mu}}(z) = \sum_{i=1}^m t_i K_{\nu_i}(z).
		\]
	\end{proposition}
	
	\begin{proof}
		The proof of convergence of $|V_n| K_{\nu_{n,i}}$ will follow roughly the same outline as the proof of Theorem \ref{thm:limit}. 
		We first consider the case $\mu_n = \mu^{\uplus 1/|V_n|}$ where $\mu$ is compactly supported, and then extend to the general case by equicontinuity.
		
		First, recall from Proposition \ref{prop: K transform equation} that $\nu_{n,i}$ is the tree convolution of $\mu_n$ according to the tree $\Walk(G_n,v)$ where $v$ is any vertex in $V_{n,i}$.  Now let $\Walk(G_n,V_{n,i})$ be the tree whose vertices are the (reverse) walks that start at some vertex in $V_{n,i}$.  Note $\Walk(G_n,V_{n,i})$ as the union of $\Walk(G_n,v)$ for $v \in V_{n,i}$, where the root vertex $\emptyset$ is in their common intersection but otherwise they are disjoint.  Thus,
		\[
		\boxplus_{\Walk(G_n,V_{n,i})}(\mu_n) = \biguplus_{v \in V_{n,i}} \boxplus_{\Walk(G_n,v)}(\mu_n) = \nu_i^{\uplus V_{n,i}}.
		\]
		Thus, in particular, letting $\tilde{\nu}_{n,i} = \boxplus_{\Walk(G_n,V_{n,i})}(\mu_n)$, we have $K_{\tilde{\nu}_{n,i}} = |V_{n,i}| K_{\nu_{n,i}}$.  In terms of \S \ref{subsec: product space}, $\Walk(G_n,V_{n,i})$ produces a Hilbert space with summands $H_{v_k}^\circ \otimes \dots \otimes H_{v_1}^{\circ}$ where $v_1 \in V_{n,i}$ rather than $v_1$ being arbitrary.  Accordingly, the formula in Theorem \ref{thm:momentformula} is changed to include only partitions $\pi$ where the outer blocks of $\pi$ are labeled by vertices in $V_{n,i}$ instead of arbitrary vertices.  Hence, the formula for the moments of $\tilde{\nu}_{n,i}$ in this case is similar to Lemma \ref{lem: moments of sum} except that instead of all $\pi \in \mathcal{NC}_k(\ell,G_n)$, we only take $\pi$ in the set $\mathcal{NC}_k(\ell,G_n,V_{n,i})$ of all $\pi$ such that $\pi$ is compatible with $\ell$ and the labeling $\ell$ defines a homomorphism from $\mathrm{F}(\pi) \to G$ with all the outer blocks mapped to vertices in $V_{n,i}$.  For an out-forest $G'$, let us denote by $\Hom(G',G_n,V_{n,i})$ the set of homomorphisms from $G'$ to $G_n$ such that all the root vertices of $G'$ are mapped to vertices in $V_{n,i}$.  Then as in Lemma \ref{lem: moments of sum}, we get
		\begin{align*}
			m_k(\tilde{\nu}_{n,i}) &= \sum_{\pi \in \mathcal{NC}_k} |\Hom(\mathrm{F}(\pi),G_n,V_{n,i})| \prod_{B \in \pi} \kappa_{\bool,|B|}(\mu_n), \\
			&= \sum_{\pi \in \mathcal{NC}_k} \frac{|\Hom(\mathrm{F}(\pi),G_n,V_{n,i})|}{|V_n|^{|\pi|}} \prod_{B \in \pi} \kappa_{\bool,|B|}(\mu)
		\end{align*}
		since $\mu_n = \mu^{\uplus 1/|V_n|}$.  Our counting argument for homomorphisms in the multiregular case implies that
		\[
		\frac{|\Hom(\mathrm{F}(\pi),G_n,V_{n,i})|}{|V_n|^{|\pi|}} \to \sum_{\substack{\ell: V' \to [n] \\ \ell(r) = i}} t_{\ell(r)} \prod_{v \in V' \setminus \{r\}} a_{\ell(v_-),\ell(v)},
		\]
		by the same reasoning as in Proposition \ref{prop: multiregular 1}.  Thus, the same reasoning as in Lemma \ref{lem:cpctsuppconvergence} shows that $\tilde{\nu}_{n,i}$ converges as $n \to \infty$ to some $\tilde{\nu}_i$.
		
		Now for the case of general $\mu_n$ such that $\mu_n^{\uplus |V_n|} \to \mu$, we use the fact that the mapping $\sigma \mapsto \boxplus_{\Walk(G_n,V_{n,i})}(\sigma^{\uplus 1/|V_n|}$ is uniformly equicontinuous on any compact subset of $\mathcal{P}(\R)$, which follows from \cite[Theorem 6.2]{JDW2021}.  Thus, the same $3 \epsilon$ argument from the proof of Theorem \ref{thm:limit} in \S \ref{subsec: limit general} applies here.  Thus, we obtain convergence of $\tilde{\nu}_{n,i}$ in this case.
		
		This means that $|V_{n,i}| K_{\nu_{n,i}}$ converges as $n \to \infty$ to $K_{\tilde{\nu}_i}$.  Hence also $|V_n| K_{\nu_{n,i}} = \frac{|V_n|}{|V_{n,i}|} |V_{n,i}| K_{\nu_{n,i}}$ converges to $(1/t_i) K_{\tilde{\nu}_i}$ as $n \to \infty$.  Now let $\nu_i = \tilde{\nu}_i^{\uplus 1/t_i}$.  Now recall that
		\[
		|V_n| K_{\nu_{n,i}}(z) = |V_n| K_{\mu_n}\left(z - \sum_{j=1}^m \frac{A_{n,i,j}}{|V_n|} \cdot |V_n| K_{\nu_{n,j}}(z) \right).
		\]
		We now know that $|V_n| K_{\nu_{n,i}} \to K_{\nu_i}$ as $n \to \infty$.  Moreover, $|V_n| K_{\mu_n}$ converges to $K_\mu$.  Also, the functions $|V_n| K_{\mu_n}$ are equicontinuous because the measures $\mu_n^{\uplus |V_n|}$ inhabit a precompact subset of $\mathcal{P}(\R)$ since $\mu_n^{\uplus |V_n|} \to \mu$.  These facts together imply by a $3 \epsilon$ argument that we can take the limit of the above equation and obtain
		\[
		K_{\nu_i}(z) = K_\mu\left(z - \sum_{j=1}^m a_{i,j} K_{\nu_j}(z) \right).
		\]
		Similar reasoning shows that we can take the limit of the equation $K_{\boxplus_{G_n}(\mu_n)} = \sum_{i=1}^m |V_{n,i}| K_{\nu_{n,i}}$ to obtain $K_{\widehat{\mu}} = \sum_{i=1}^m t_i K_{\nu_i}$.
	\end{proof}
	
	Proposition \ref{prop: multiregular 2} allows for numerically tractable computations of limit measures associated to multiregular digraphs $G_n$.
	
	\begin{example}[Central limit distribution for multiregular digraphs]
		Suppose we want to find the central limit distribution.  Since the Boolean central limit distribution is $\mu = (1/2)(\delta_{-1} + \delta_1)$, the central limit distribution for the sequence $G_n$ will be given by the corresponding measure $\widehat{\mu}$.  Note that in this case $K_\mu(z) = 1/z$.  Thus, \eqref{eq: multiregular limit fixed point} reduces to 
		\[
		K_{\nu_i}(z) = \left(z - \sum_{j=1}^n a_{i,j} K_{\nu_j}(z) \right)^{-1},
		\]
		or equivalently
		\[
		z K_{\nu_i}(z) - \sum_{j=1}^n a_{i,j} K_{\nu_i}(z) K_{\nu_j}(z) = 1.
		\]
		In other words, the $m$ unknowns $(K_{\nu_1}(z),\dots,K_{\nu_m}(z))$ satisfy a quadratic system of $m$ equations.  Then $\widehat{\mu}$ is obtained by $K_{\widehat{\mu}} = \sum_{i=1}^n t_i K_{\nu_i}$.
		
		Figures \ref{fig: CLT example 1}, \ref{fig: CLT example 2}, \ref{fig: CLT example 3} show examples of numerical approximations of central limit densities using the fixed point equation \eqref{eq: multiregular limit fixed point}.  Changing the parameters produces symmetric distributions whose shape can be semicircular, become more flat, and then develop a concave shape in the middle with two bumps on the boundary, somewhat resembling the arcsine distribution.  Of course, for special cases of a complete graph and its complement, one obtains the semicircular distribution and the Bernoulli distribution respectively.
		
	\end{example}
	
	\begin{figure}
		\centering
		
		\begin{tikzpicture}[xscale=2.0, yscale=4.0]
			
			\draw[gray,<->] (0,-0.1) -- (0,0.5);
			\draw[gray,<->] (-2,0) -- (2,0);
			
			\draw[gray] (-1.5,-0.05) -- (-1.5,0.05);
			\node[text=gray] at (-1.5,-0.1) {$-1.5$};
			\draw[gray] (1.5,-0.05) -- (1.5,0.05);
			\node[text=gray] at (1.5,-0.1) {$1.5$};
			
			\draw[gray] (-0.05,0.4) -- (0.05,0.4);
			\node[text=gray] at (-0.2,0.4) {$0.4$};
			
			\draw[blue,smooth] plot coordinates {
				(-2.00000000000000,0.000265726183160442)
				(-1.92000000000000,0.000351439868333772)
				(-1.84000000000000,0.000504872146444988)
				(-1.76000000000000,0.000853007230988464)
				(-1.68000000000000,0.00238849187407116)
				(-1.64000000000000,0.0760553394118639)
				(-1.63600000000000,0.152469121557454)
				(-1.60000000000000,0.343774140750160)
				(-1.52000000000000,0.398582512975111)
				(-1.44000000000000,0.388891687138205)
				(-1.36000000000000,0.371225508557775)
				(-1.28000000000000,0.353942020307990)
				(-1.20000000000000,0.338611185306404)
				(-1.12000000000000,0.325361959427993)
				(-1.04000000000000,0.313992517670936)
				(-0.960000000000000,0.304249448401015)
				(-0.880000000000000,0.295900767303503)
				(-0.800000000000000,0.288751362731279)
				(-0.720000000000000,0.282641992241129)
				(-0.640000000000000,0.277444189639735)
				(-0.560000000000000,0.273054861679553)
				(-0.480000000000000,0.269391629639139)
				(-0.400000000000000,0.266389095127486)
				(-0.320000000000000,0.263995946405210)
				(-0.240000000000000,0.262172763899283)
				(-0.160000000000000,0.260890425871188)
				(-0.0800000000000000,0.260128926822380)
				(0.000000000000000,0.259878548477292)
				(0.0800000000000000,0.260128926822380)
				(0.160000000000000,0.260890425871188)
				(0.240000000000000,0.262172763899283)
				(0.320000000000000,0.263995946405210)
				(0.400000000000000,0.266389095127486)
				(0.480000000000000,0.269391629639139)
				(0.560000000000000,0.273054861679553)
				(0.640000000000000,0.277444189639735)
				(0.720000000000000,0.282641992241129)
				(0.800000000000000,0.288751362731279)
				(0.880000000000000,0.295900767303503)
				(0.960000000000000,0.304249448401015)
				(1.04000000000000,0.313992517670936)
				(1.12000000000000,0.325361959427993)
				(1.20000000000000,0.338611185306404)
				(1.28000000000000,0.353942020307990)
				(1.36000000000000,0.371225508557775)
				(1.44000000000000,0.388891687138205)
				(1.52000000000000,0.398582512975111)
				(1.60000000000000,0.343774140750160)
				(1.63600000000000,0.152469121557454)
				(1.64000000000000,0.0760553394118639)
				(1.68000000000000,0.00238849187407116)
				(1.76000000000000,0.000853007230988464)
				(1.84000000000000,0.000504872146444988)
				(1.92000000000000,0.000351439868333772)
				(2.00000000000000,0.000265726183160442)
			};
			
		\end{tikzpicture}
		
		\caption{Approximation of the central limit density for $2$-regular digraphs with $t_1 = 0.3$, $t_2 = 0.7$, $a_{1,1} = 0.2$, $a_{1,2} = 0.4$, $a_{2,1} = 0.2$, $a_{2,2} = 0.5$.  We approximated the density using the imaginary part of the Cauchy transform at $x+iy$ where $y = 0.001$, and we approximated the Cauchy transform using $10,000$ iterations of the fixed point equation.}
		\label{fig: CLT example 1}
	\end{figure}
	
	\begin{figure}
		\centering
		
		\begin{tikzpicture}[xscale=2.0, yscale=3.0]
			
			\draw[gray,<->] (0,-0.1) -- (0,1.2);
			\draw[gray,<->] (-2,0) -- (2,0);
			
			\draw[gray] (-1.5,-0.05) -- (-1.5,0.05);
			\node[text=gray] at (-1.5,-0.1) {$-1.5$};
			\draw[gray] (1.5,-0.05) -- (1.5,0.05);
			\node[text=gray] at (1.5,-0.1) {$1.5$};
			
			\draw[gray] (-0.05,1.0) -- (0.05,1.0);
			\node[text=gray] at (-0.2,1.0) {$1.0$};
			
			\draw[blue,smooth] plot coordinates {
				(-1.38000000000000,0.00361315722644640)
				(-1.32000000000000,1.10138839398701)
				(-1.26000000000000,0.728645335231011)
				(-1.20000000000000,0.583048925401042)
				(-1.14000000000000,0.499738653508108)
				(-1.08000000000000,0.443751706457028)
				(-1.02000000000000,0.402476883608009)
				(-0.960000000000000,0.370085253719915)
				(-0.900000000000000,0.343440549352775)
				(-0.840000000000000,0.320663473720847)
				(-0.780000000000000,0.300528526999409)
				(-0.720000000000000,0.282178863818168)
				(-0.660000000000000,0.264984934119877)
				(-0.600000000000000,0.248487047986846)
				(-0.540000000000000,0.232416590383439)
				(-0.480000000000000,0.216823392970320)
				(-0.420000000000000,0.202280867504630)
				(-0.360000000000000,0.189823085659090)
				(-0.300000000000000,0.180213996871458)
				(-0.240000000000000,0.173373254238837)
				(-0.180000000000000,0.168741136472075)
				(-0.120000000000000,0.165779819251086)
				(-0.0600000000000001,0.164130549269555)
				(0.000000000000000,0.163625875781091)
				(0.0600000000000001,0.164130549269555)
				(0.120000000000000,0.165779819251086)
				(0.180000000000000,0.168741136472075)
				(0.240000000000000,0.173373254238837)
				(0.300000000000000,0.180213996871458)
				(0.360000000000000,0.189823085659090)
				(0.420000000000000,0.202280867504630)
				(0.480000000000000,0.216823392970320)
				(0.540000000000000,0.232416590383439)
				(0.600000000000000,0.248487047986846)
				(0.660000000000000,0.264984934119877)
				(0.720000000000000,0.282178863818168)
				(0.780000000000000,0.300528526999409)
				(0.840000000000000,0.320663473720847)
				(0.900000000000000,0.343440549352775)
				(0.960000000000000,0.370085253719915)
				(1.02000000000000,0.402476883608009)
				(1.08000000000000,0.443751706457028)
				(1.14000000000000,0.499738653508108)
				(1.20000000000000,0.583048925401042)
				(1.26000000000000,0.728645335231011)
				(1.32000000000000,1.10138839398700)
				(1.38000000000000,0.00361315722644640)
			};
			
		\end{tikzpicture}
		
		\caption{Approximation of the central limit density for $2$-regular digraphs with $t_1 = 0.3$, $t_2 = 0.7$, $a_{1,1} = 0.2$, $a_{1,2} = 0.4$, $a_{2,1} = 0.2$, $a_{2,2} = 0.1$.  We use $y = 0.0001$ and $50,000$ iterations.}
		\label{fig: CLT example 2}
	\end{figure}

	\begin{figure}
		\centering
		
		\begin{tikzpicture}[xscale=2.0, yscale=4.0]
			
			\draw[gray,<->] (0,-0.1) -- (0,0.5);
			\draw[gray,<->] (-2.1,0) -- (2.1,0);
			
			\draw[gray] (-2,-0.05) -- (-2,0.05);
			\node[text=gray] at (-2,-0.1) {$-2$};
			\draw[gray] (2,-0.05) -- (2,0.05);
			\node[text=gray] at (2,-0.1) {$2$};
			
			\draw[gray] (-0.05,0.4) -- (0.05,0.4);
			\node[text=gray] at (-0.2,0.4) {$0.4$};
			
			\draw[blue,smooth] plot coordinates {
				(-2.00000000000000,0.0000455514099182060)
				(-1.80000000000000,0.141224932493458)
				(-1.60000000000000,0.218111289833082)
				(-1.40000000000000,0.253523401548082)
				(-1.20000000000000,0.273819052953625)
				(-1.00000000000000,0.284440904941143)
				(-0.800000000000000,0.290815875706709)
				(-0.600000000000000,0.299468336292987)
				(-0.400000000000000,0.300260311066609)
				(-0.200000000000000,0.302814800534135)
				(0.000000000000000,0.305083040603897)
				(0.200000000000000,0.302814800534135)
				(0.400000000000000,0.300260311066609)
				(0.600000000000000,0.299468336292987)
				(0.800000000000000,0.290815875706709)
				(1.00000000000000,0.284440904941143)
				(1.20000000000000,0.273819052953625)
				(1.40000000000000,0.253523401548082)
				(1.60000000000000,0.218111289833082)
				(1.80000000000000,0.141224932493458)
				(2.00000000000000,0.0000455514099182060)
			};
			
		\end{tikzpicture}
		
		\caption{Approximation of the central limit density for $2$-regular digraphs with $t_1 = 0.5$, $t_2 = 0.5$, $a_{1,1} = 0.4$, $a_{1,2} = 0.5$, $a_{2,1} = 0.4$, $a_{2,2} = 0.5$.  We used $y = 0.0001$ and $50,000$ iterations.}
		\label{fig: CLT example 3}
	\end{figure}

	\begin{example}[Poisson distribution for multiregular digraphs]
		Now consider the Poisson limit theorem or law of small numbers.  Recall from Example \ref{ex: BM LSN} that the Boolean analog of the Poisson distribution of intensity $\lambda$ is $\mu = \frac{1}{1+\lambda} \delta_0 + \frac{\lambda}{1 + \lambda} \delta_{1+\lambda}$, and its $K$-transform is $K_\mu(z) = \lambda z / (z - 1)$.  Thus, \eqref{eq: multiregular limit fixed point} becomes
		\[
		K_{\nu_i}(z) = \lambda \left(z - \sum_{j=1}^n a_{i,j} K_{\nu_j}(z) \right)\left(z - \sum_{j=1}^n a_{i,j} K_{\nu_j}(z) - 1\right)^{-1},
		\]
		which after algebraic manipulation can be written equivalently as
		\[
		K_{\nu_i}(z) = (K_{\nu_i}(z) - \lambda) \left(z - \sum_{j=1}^n a_{i,j} K_{\nu_j}(z) \right).
		\]
		Thus, similar to the central limit case, $K_{\nu_i}(z)$'s satisfy a quadratic system of $m$ equations in $m$ unknowns.
		
		Figures \ref{fig: Poisson example 1} and \ref{fig: Poisson example 2} show numerical approximations of Poisson limit distributions using \eqref{eq: multiregular limit fixed point}.
	\end{example}
	
	\begin{figure}
		\centering
		
		\begin{tikzpicture}[xscale=2,yscale=4]
			
			\draw[gray,<->] (0,-0.1) -- (0,0.7);
			\draw[gray,<->] (-0.2,0) -- (4.2,0);
			
			\draw[gray] (2,-0.05) -- (2,0.05);
			\node[text=gray] at (2,-0.1) {$2$};
			\draw[gray] (4,-0.05) -- (4,0.05);
			\node[text=gray] at (4,-0.1) {$4$};
			
			\draw[blue,smooth] plot coordinates {
				(0.0300000000000000,0.00125960268549968)
				(0.0350000000000000,0.296134506791876)
				(0.0400000000000000,0.465764916733141)
				(0.0450000000000000,0.538442964590653)
				(0.0500000000000000,0.575308216784733)
				(0.0800000000000000,0.596591549499774)
				(0.100000000000000,0.570353770078027)
				(0.160000000000000,0.494405383047892)
				(0.240000000000000,0.423023275593122)
				(0.320000000000000,0.376072778370167)
				(0.400000000000000,0.342702897233410)
				(0.480000000000000,0.317615850397270)
				(0.560000000000000,0.297985717328982)
				(0.640000000000000,0.282166150801879)
				(0.720000000000000,0.269128633813804)
				(0.800000000000000,0.258194392179862)
				(0.880000000000000,0.248895287806084)
				(0.960000000000000,0.240896412434663)
				(1.04000000000000,0.233952861329194)
				(1.12000000000000,0.227876642034585)
				(1.20000000000000,0.222524538062542)
				(1.28000000000000,0.217782787669322)
				(1.36000000000000,0.213559921222612)
				(1.44000000000000,0.209780842998276)
				(1.52000000000000,0.206382649414923)
				(1.60000000000000,0.203311478551432)
				(1.68000000000000,0.200520060818744)
				(1.76000000000000,0.197965733795604)
				(1.84000000000000,0.195608750122111)
				(1.92000000000000,0.193410746577751)
				(2.00000000000000,0.191333263002715)
				(2.08000000000000,0.189336205513990)
				(2.16000000000000,0.187376140205228)
				(2.24000000000000,0.185404278683698)
				(2.32000000000000,0.183363968701736)
				(2.40000000000000,0.181187418455048)
				(2.48000000000000,0.178791236614374)
				(2.56000000000000,0.176070113532119)
				(2.64000000000000,0.172887506160622)
				(2.72000000000000,0.169061317167831)
				(2.80000000000000,0.164340816726975)
				(2.88000000000000,0.158367296455823)
				(2.96000000000000,0.150601965806408)
				(3.04000000000000,0.140179998404377)
				(3.12000000000000,0.125567645125053)
				(3.20000000000000,0.103526248909625)
				(3.28000000000000,0.0636326849824644)
				(3.36000000000000,0.000590037919648674)
				(3.44000000000000,0.000278118008838087)
				(3.52000000000000,0.000189202665131829)
				(3.60000000000000,0.000144668125139997)
				(3.68000000000000,0.000117381752157389)
				(3.76000000000000,0.0000987634500950019)
				(3.84000000000000,0.0000851706516888289)
				(3.92000000000000,0.0000747732090002449)
				(4.00000000000000,0.0000665436875924323)
			};
			
		\end{tikzpicture}
		
		\caption{Approximation of the Poisson limit density with $\lambda = 1$ for $2$-regular digraphs.  We computed using the same parameters as in Figure \ref{fig: CLT example 1}, except on the interval $[0,0.1]$, we used a $y=0.00001$ and $100,000$ iterations.  The computation at $x = 0$ suggests that the measure to have an atom at $0$ of mass about $0.25$.}
		\label{fig: Poisson example 1}
	\end{figure}

	\begin{figure}
		\centering
		
		\begin{tikzpicture}[xscale=1.4,yscale=4]
			
			\draw[gray,<->] (0,-0.1) -- (0,0.7);
			\draw[gray,<->] (-0.2,0) -- (6.2,0);
			
			\draw[gray] (3,-0.05) -- (3,0.05);
			\node[text=gray] at (3,-0.1) {$3$};
			\draw[gray] (6,-0.05) -- (6,0.05);
			\node[text=gray] at (6,-0.1) {$6$};
			
			\draw[smooth,blue] plot coordinates {
				(0.000000000000000,0.000350256238652966)
				(0.100000000000000,0.000782158404867348)
				(0.200000000000000,0.158727558508121)
				(0.300000000000000,0.234333494580896)
				(0.400000000000000,0.237248428106128)
				(0.500000000000000,0.229768359770077)
				(0.600000000000000,0.220656327404564)
				(0.700000000000000,0.211916166179565)
				(0.800000000000000,0.204027944838049)
				(0.900000000000000,0.197033958743137)
				(1.00000000000000,0.190868262684686)
				(1.10000000000000,0.185465896366811)
				(1.20000000000000,0.180683783064938)
				(1.30000000000000,0.176467135973946)
				(1.40000000000000,0.172739904770195)
				(1.50000000000000,0.169441714399428)
				(1.60000000000000,0.166522520585243)
				(1.70000000000000,0.163940581519028)
				(1.80000000000000,0.161661021467772)
				(1.90000000000000,0.159654654104678)
				(2.00000000000000,0.157897019708814)
				(2.10000000000000,0.156367604987349)
				(2.20000000000000,0.155049214138394)
				(2.30000000000000,0.153927462792952)
				(2.40000000000000,0.152990371055298)
				(2.50000000000000,0.152228036397517)
				(2.60000000000000,0.151632371122183)
				(2.70000000000000,0.151196892379729)
				(2.80000000000000,0.150916555367118)
				(2.90000000000000,0.150787622444557)
				(3.00000000000000,0.150807562594208)
				(3.10000000000000,0.150974977004420)
				(3.20000000000000,0.151289547673815)
				(3.30000000000000,0.151752006854280)
				(3.40000000000000,0.152364125937501)
				(3.50000000000000,0.153128723068051)
				(3.60000000000000,0.154049689351066)
				(3.70000000000000,0.155132034005414)
				(3.80000000000000,0.156381949150450)
				(3.90000000000000,0.157806895006685)
				(4.00000000000000,0.159415705944999)
				(4.10000000000000,0.161218716678550)
				(4.20000000000000,0.163227905307654)
				(4.30000000000000,0.165457044718000)
				(4.40000000000000,0.167921843812008)
				(4.50000000000000,0.170640041106256)
				(4.60000000000000,0.173631377407214)
				(4.70000000000000,0.176917305788434)
				(4.80000000000000,0.180520163490652)
				(4.90000000000000,0.184461262423179)
				(5.00000000000000,0.188756798249979)
				(5.10000000000000,0.193409269722039)
				(5.20000000000000,0.198389332168241)
				(5.30000000000000,0.203596231253425)
				(5.40000000000000,0.208766901623418)
				(5.50000000000000,0.213250073377660)
				(5.60000000000000,0.215375678031128)
				(5.70000000000000,0.210341039084288)
				(5.80000000000000,0.180190052219278)
				(5.90000000000000,0.00188985819851512)
				(6.00000000000000,0.000437871451892657)
			};
			
		\end{tikzpicture}
		
		\caption{Approximation of the Poisson limit density with $\lambda = 3$ for $2$-regular digraphs.  We computed using the same parameters as in Figure \ref{fig: CLT example 1}.}
		\label{fig: Poisson example 2}
	\end{figure}
	
	\begin{example}[Cauchy distribution for multiregular digraphs]
		If we take $\mu$ to be the standard Cauchy distribution, then $K_\mu(z) = -i$ in the upper half plane.  Thus, \eqref{eq: multiregular limit fixed point} tells us that $K_{\nu_i}(z) = -i$.  Hence, $K_{\widehat{\mu}}(z) = \sum_{i=1}^m t_i K_{\nu_i}(z) = -i$ since $t_1 + \dots + t_m = 1$.  It follows that when $\mu$ is the standard Cauchy distribution, then also $\widehat{\mu}$ is the standard Cauchy distribution.
	\end{example}

	\subsection{Sparse graphs} \label{subsec: sparse graphs}
	
	The next proposition shows that if a sequence of digraphs is sufficiently sparse, then the normalized count of homomorphisms converges to $\beta_{G'} = 0$, and so the $G_n$-convolution is asymptotically Boolean convolution.  This is a generalization of the case of BM-independence for posets given by regular trees from \cite[\S 8]{JW2008}.
	
	\begin{proposition}[Limit theorem for sparse graphs] \label{prop: sparse}
		Let $G_n$ be a sequence of digraphs such that $|E_n| / |V_n|^2 \to 0$.  Then for every out-tree $G'$ with more than one vertex, we have
		\[
		\lim_{n \to \infty} \frac{|\Hom(G',G_n)|}{|V_n|^{|V'|}} = 0.
		\]
		Hence, in this case, if $\mu_n^{\uplus |V_n|} \to \mu$, then $\boxplus_{G_n}(\mu_n) \to \mu$ also.
	\end{proposition}
	
	\begin{proof}
		Suppose that $G'$ is an out-tree with more than one vertex.  Let $r$ be the root, and fix some $v$ with $r \rightsquigarrow v$.  Given a homomorphism $\phi: G' \to G_n$, the pair $(r,v)$ must be mapped to some edge.  There are $|E_n|$ choices for this edge.  Then the remaining $|V'| - 2$ vertices must be mapped to some vertex in $V_n$, and so the number of choices for the rest of the values of $\phi$ is at most $|V'|^{|V_n|-2}$.  Thus,
		\[
		|\Hom(G',G_n)| \leq |E_n| |V'|^{|V_n|-2},
		\]
		so
		\[
		\frac{|\Hom(G',G_n)|}{|V_n|^{|V'|}} \leq \frac{|E_n|}{|V_n|^2} \to 0.
		\]
		Hence, $\beta_{G'} = 0$ unless $G'$ has only one vertex.  More generally, if $G'$ is an out-forest, then $\beta_{G'} = 0$ unless $G'$ has no edges, in which case $\beta_{G'} = 1$.
		
		It follows that in Lemma \ref{lem:cpctsuppconvergence}, the moments $\boxplus_G(\mu^{\uplus 1/|V_n|})$ in \eqref{eq: limit moments} are given by the sum over interval partitions of $\kappa_{\bool,\pi}(\mu)$, since $\mathrm{F}(\pi)$ has no edges if and only if $\pi$ is interval partitions.  This means that when we take $\mu_n = \mu^{\uplus 1/|V_n|}$, then the moments of the limiting measure in \eqref{eq: limit moments} are the same as the moments of $\mu$.
		
		The general statement that if $\mu_n^{\uplus |V_n|} \to \mu$, then $\boxplus_{G_n}(\mu_n) \to \mu$ follows from the equicontinuity of the convolution operations as in \S \ref{subsec: limit general}.
	\end{proof}
	
	\begin{example}[Posets given by finite out-trees]
		The following example is from \cite[\S 8]{JW2008}. Fix $d$.  Let $T_n$ be the $d$-regular rooted out-tree (where each vertex has $d$ outgoing edges) truncated to depth $n$.  Let $G_n = (V_n,E_n)$ be the graph where $E_n$ is the transitive closure of the edge set for $T_n$ (as in Lemma \ref{lem: BM homomorphism volume}).  Observe that
		\[
		|V_n| = \sum_{j=0}^n d^j = \frac{d^{n+1} - 1}{d - 1}.
		\]
		Meanwhile, the number of edges can be counted as follows:  Each vertex $v$ at depth $j$ in the out-tree has $k(v) = \sum_{i=1}^{n-j} d^i$, which evaluates to $d(d^{n-j} - 1)/(d-1)$.  Now summing this over all the vertices, we obtain
		\[
		|E_n| = \sum_{j=0}^n d^j \cdot \frac{d(d^{n-j}-1)}{d-1} = \frac{d}{d-1} \sum_{j=0}^n (d^n - d^j) \leq \frac{n(d^{n+1}-1)}{d-1} = n|V_n|.
		\]
		Thus, 
		\[
		\frac{|E_n|}{|V_n|^2} \leq \frac{n}{|V_n|} = \frac{n(d-1)}{(d^{n+1} - 1)} \to 0.
		\]
		Hence, the limiting measures for the sequence $G_n$ reduce to those of the Boolean case.  This generalizes the observation of \cite{JW2008} that the central limit measure for this case is $(1/2)(\delta_{-1}+\delta_1)$.
	\end{example}
	
	\begin{remark}
		Our argument to bound the number of homomorphisms in the proof of Proposition \ref{prop: sparse} generalizes to yield the following statement:  If $G'$ is a out-tree, $G''$ out-subtree of it, and $G$ is any finite digraph, then
		\[
		\frac{|\Hom(G',G)|}{|V|^{|V'|}} \leq \frac{|\Hom(G'',G)|}{|V|^{|V''|}}.
		\]
		The reason for this is that every homomorphism $G' \to G$ restricts to a homomorphism $G'' \to G$.  Hence,
		\[
		\Hom(G',G) \subseteq \Hom(G'',G) \times V^{V' \setminus V''},
		\]
		where $V^{V' \setminus V''}$ denotes the set of all functions $V' \setminus V'' \to V$.  Thus, we get
		\[
		|\Hom(G',G)| \leq |\Hom(G'',G)| |V|^{|V'| - |V''|},
		\]
		which is the inequality asserted above.  In particular, in the situation where $\beta_{G'} = \lim_{n \to \infty} |\Hom(G',G_n)| / |V_n|^{|V'|}$ exists for all out-trees $G'$, then we have $\beta_{G'} \leq \beta_{G''}$ whenever $G''$ is an out-subtree of $G'$ (in fact, we do not even need the root of $G''$ to agree with the root of $G'$).
	\end{remark}
	
	\section{Fock space models} \label{sec: Fock space}
	
	Proposition \ref{prop: continuum limit} described a general situation when each limit $\beta_{G'}$ exists based on discrete approximations of measurable digraphs.  In this section, we define Fock spaces that serve as a continuum analog of the product space construction in \S \ref{sec: independence}.  These Fock spaces in particular furnish models for the limiting measures in Proposition \ref{prop: continuum limit} in the compactly supported setting (see Corollary \ref{cor: Fock space gives limit measure} for precise statement). As motivation, we point out that \cite{Attal2011} showed the free Fock space on $L^2[0,\infty)$ is a continuum limit of a free product Hilbert spaces.
	
	Our Fock space is a direct sum of Bochner spaces $L^2(\Omega^{\times k}, \rho_k;\mathcal{H}^{\otimes k})$ for some measures $\rho_k$ on a product space $\Omega^{\times k}$ and some Hilbert space $\mathcal{H}$.  The operators will have the form $\mathfrak{n}(\phi) + \ell(h) + \ell(h)^* + \mathfrak{m}(S)$, where $\ell(h)$ and $\ell(h)^*$ are creation and annihilation operators associated to some $h \in L^2(\Omega,\rho_1; \mathcal{H})$, and $\mathfrak{n}(\phi)$ is a multiplication operator associated to $\phi \in L^\infty(\Omega)$, and $\mathfrak{m}(S)$ is another type of multiplication operator associated to $S \in L^\infty(\Omega; B(\mathcal{H}))$.
	
	\subsection{Construction of a Fock space and operators thereon}
	
	Although in the last section, we considered $(\Omega,\rho)$ to be a probability measure space, here we will proceed more generally with a complete $\sigma$-finite measure space, in order to include such examples as the Fock spaces supporting Brownian motions on $[0,\infty)$.  Moreover, while in \S \ref{sec: examples}, we considered edges given by $\mathcal{E} \subseteq \Omega \times \Omega$, we now introduce a weighted version where a general nonnegative $w: \Omega \times \Omega \to [0,\infty)$ replaces the indicator function $\mathbbm{1}_{\mathcal{E}}$; the weighted version will be used in Example \ref{ex: multiregular Fock space}.
	
	\begin{definition}
		Let $(\Omega,\rho)$ be a complete $\sigma$-finite measure space and let $(\Omega^{\times k},\rho^{\times k})$ be the completed product measure space.  Let $\mathrm{w} \in L^\infty(\Omega \times \Omega)$ with $\mathrm{w} \geq 0$.   For $k \geq 1$, let $\rho_k$ be the measure on $\Omega^{\times k}$ given by
		\[
		d\rho_k^{\mathrm{w}}(\omega_1,\dots,\omega_k) = \mathrm{w}(\omega_k,\omega_{k-1}) \dots \mathrm{w}(\omega_2,\omega_1)\,d\rho^{\times k}(\omega_1,\dots,\omega_k);
		\]
		here $\rho_1^{\mathrm{w}} = \rho$.  Note that $(\Omega^{\times k},\rho_k^{\mathrm{w}})$ can be completed to a complete measure space. Let $\mathcal{H}$ be a Hilbert space.  Then we define the \emph{Fock space} as the Hilbert space
		\[
		\mathcal{F}(\Omega,\rho,\mathrm{w},\mathcal{H}) = \C \oplus \bigoplus_{k \in \N} L^2(\Omega^{\times k}, \rho_k^{\mathrm{w}}; \mathcal{H}^{\otimes k}).
		\]
		We denote the vector $1$ in the first summand by $\xi$.  Furthermore, we adopt the convention that $\Omega^{\times 0}$ is a single point, $\rho_0^{\mathrm{w}}$ is the unique probability measure on it, and $\mathcal{H}^{\otimes 0} = \C$; thus, $L^2(\Omega^{\times 0}, \rho_0, \mathcal{H}^{\times 0}) = \C$.
	\end{definition}
	
	\begin{definition}
		Consider the same setup as in the previous construction.  Let $h \in L^2(\Omega,\rho; \mathcal{H})$.  Then we define the \emph{left creation operator} $\ell(h): \mathcal{F}(\Omega,\rho,\mathrm{w},\mathcal{H}) \to \mathcal{F}(\Omega,\rho,\mathrm{w},\mathcal{H})$ as follows.  For $k \geq 0$ and $f \in L^2(\Omega^{\times k},\rho_k^{\mathrm{w}};\mathcal{H}^{\otimes k})$, let
		\[
		[\ell(h) f](\omega_1,\dots,\omega_{k+1}) = h(\omega_1) \otimes f(\omega_2,\dots,\omega_{k+1}).
		\]
		This formula immediately yields a well-defined element of $L^2(\Omega^{\times (k+1)},\rho_1 \times \rho_k^{\mathrm{w}};\mathcal{H}^{\otimes(k+1)})$, but in fact it even yields a well-defined element of $L^2(\Omega^{\times(k+1)},\rho_{k+1}^{\mathrm{w}};\mathcal{H}^{\otimes(k+1)})$ since
		\begin{multline*}
			\int_{\Omega^{\times k}} \norm{h(\omega_1) \otimes f(\omega_2,\dots,\omega_{k+1})}_{\mathcal{H}^{\otimes (k+1)}}^2 \mathrm{w}(\omega_k,\omega_{k-1}) \dots \mathrm{w}(\omega_2,\omega_1) \,d\rho^{\times (k+1)}(\omega_1,\dots,\omega_{k+1}) \\
			\leq \int_{\Omega^{\times k}} \norm{h(\omega_1)}_{\mathcal{H}}^2 \norm{f(\omega_2,\dots,\omega_{k+1})}_{\mathcal{H}^{\otimes (k+1)}}^2 \norm{\mathrm{w}}_{L^\infty(\Omega \times \Omega)} \,d\rho(\omega_1)\,d\rho_k^{\mathrm{w}}(\omega_2,\dots,\omega_k),
		\end{multline*}
		and moreover
		\[
		\norm{\ell(h) f}_{L^2(\Omega^{\times(k+1)},\rho_{k+1}^{\mathrm{w}},\mathcal{H})} \leq \norm{w}_{L^\infty(\Omega \times \Omega)} \norm{h}_{L^2(\Omega,\rho;\mathcal{H})} \norm{f}_{L^2(\Omega^{\times k},\rho_k^{\mathrm{w}})}.
		\]
		Note that in the case $k = 0$, we have $\ell(h) \xi = h \in L^2(\Omega,\rho;\mathcal{H})$. It follows that $\ell(h)$ defines a bounded operator on $\mathcal{F}(\Omega,\rho,w,\mathcal{H})$ with
		\[
		\norm{\ell(h)} \leq \norm{w}_{L^\infty(\Omega \times \Omega)} \norm{h}_{L^2(\Omega,\rho;\mathcal{H})}.
		\]
		Thus, the creation operator is well-defined.  Its adjoint $\ell(h)^*$ is called the \emph{left annihilation operator} associated to $h$.
	\end{definition}
	
	\begin{fact}
		The annihilation operator $\ell(h)$ satisfies
		\[
		\ell(h)^* \xi = 0
		\]
		Moreover, $\ell(h)^*$ maps $L^2(\Omega^{\times k},\rho_k^{\mathrm{w}}; \mathcal{H}^{\otimes k})$ into $L^2(\Omega^{\times (k-1)},\rho_{k-1}^{\mathrm{w}}; \mathcal{H}^{\otimes (k-1)})$ for each $k \geq 1$ and satisfies
		\[
		[\ell(h)^* f](\omega_1,\dots,\omega_{k-1}) = \int_\Omega (\ip{h(\omega),-}_{\mathcal{H}} \otimes \id_{\mathcal{H}^{\otimes(k-1)}})[f(\omega,\omega_1,\dots,\omega_{k-1})] \mathrm{w}(\omega_1, \omega) \,d\rho(\omega),
		\]
		where $\ip{h(\omega),-} \otimes \id_{\mathcal{H}^{\otimes(k-1)}}$ denotes the map $\mathcal{H}^{\otimes k} \to \mathcal{H}^{\otimes (k-1)}$ given by
		\[
		f_1 \otimes \dots \otimes f_k \mapsto \ip{h(\omega),f_1}_{\mathcal{H}} f_2 \otimes f_3 \otimes \dots \otimes f_k,
		\]
		and in the case $k = 1$, it is $f \mapsto \ip{h(\omega),f}_{\mathcal{H}} \in \C$.
	\end{fact}
	
	\begin{remark}
		It may be easier to understand the annihilation operator through its action on simple tensors.  If $f(\omega_1,\dots,\omega_k) = f_1(\omega_1) \otimes \dots \otimes f_k(\omega_k)$, where $f_j \in L^2(\Omega,\rho)$, then
		\[
		[\ell(h)^* f](\omega_1,\dots,\omega_{k-1}) = \int_\Omega \ip{h(\omega),f_1(\omega)} \mathrm{w}(\omega_1,\omega)\,d\rho(\omega) f_2(\omega_1) \otimes \dots \otimes f_k(\omega_{k-1}).
		\]
	\end{remark}
	
	\begin{definition}
		$L^\infty(\Omega, B(L^2(\Omega,\rho))$ denotes the space of essentially bounded $*$-SOT measurable maps from $\Omega$ into $B(L^2(\Omega,\rho))$.
	\end{definition}
	
	\begin{definition}
		Consider the Fock space  $\mathcal{F}(\Omega,\rho,w,\mathcal{H})$ defined above.  Let $S \in L^\infty(\Omega, B(\mathcal{H}))$.  Then we define the \emph{multiplication operator} $\mathfrak{m}(S): \mathcal{F}(\Omega,\rho,\mathrm{w},\mathcal{H}) \to \mathcal{F}(\Omega,\rho,\mathrm{w},\mathcal{H})$ by  $\mathfrak{m}(S)|_{\C} = 0$ and for $k \geq 1$ and $f \in L^2(\Omega^{\times k},\rho_k; \mathcal{H}^{\otimes k})$,
		\[
		[\mathfrak{m}(S)f](\omega_1,\dots,\omega_k) := (S(\omega_1) \otimes \id_{\mathcal{H}^{\otimes (n-1)}})(f(\omega_1,\dots,\omega_k)).
		\]
		Since $S(\omega_1) \otimes \id_{\mathcal{H}^{\otimes(k-1)}}$ defines a bounded operator on $\mathcal{H}^{\otimes k}$ with norm less than or equal to that of $S(\omega_1)$, we deduce that
		\[
		\norm{\mathfrak{m}(S)f}_{L^2(\Omega^{\times k},\rho_k)} \leq \norm{S}_{L^\infty(\Omega,B(L^2(\Omega,\rho))} \norm{f}_{L^2(\Omega^{\times k},\rho_k;\mathcal{H}^{\otimes k})}.
		\]
		Thus, $\mathfrak{m}(S)$ defines a bounded operator on $L^2(\Omega^{\times k},\rho_k)$ with norm less than or equal to that of $\norm{S}_{L^\infty(\Omega,B(\mathcal{H}))}$.  Since $\mathcal{F}(\Omega,\rho,w,\mathcal{H})$ is the direct sum of the subspaces $L^2(\Omega^{\times n},\rho_k;\mathcal{H}^{\otimes k})$, we conclude that $\mathfrak{m}(S)$ is a bounded operator on the Fock space.
	\end{definition}
	
	\begin{fact}
		$\mathfrak{m}: L^\infty(\Omega, B(\mathcal{H})) \to B(\mathcal{F}(\Omega,\rho,w,\mathcal{H}))$ is a $*$-homomorphism.
	\end{fact}
	
	\begin{definition}
		Let $\phi \in L^1(\Omega,\rho)$.  Then we define an operator $\mathfrak{n}(\phi)$ on $\mathcal{F}(\Omega,\rho,\mathrm{w},\mathcal{H})$ by
		\[
		\mathfrak{n}(\phi)|_{\C} = \int_{\Omega} \phi \,d\rho,
		\]
		and for $f \in L^2(\Omega^{\times k},\rho_k^{\mathrm{w}}; \mathcal{H}^{\otimes k})$,
		\[
		(\mathfrak{n}(\phi)f)(\omega_1,\dots,\omega_k) = \int \phi(\omega) w(\omega_1,\omega)\,d\rho(\omega) f(\omega_1,\dots,\omega_k).
		\]
		Here $\mathfrak{n}(\phi)$ maps $L^2(\Omega^{\times k},\rho_k^{\mathrm{w}}; \mathcal{H}^{\otimes k})$ into itself for each $k$.  Also, $\norm{\mathfrak{n}(\phi)} \leq \norm{\mathrm{w}}_{L^\infty(\Omega^{\times 2}, \rho^{\times 2})} \norm{\phi}_{L^1(\Omega,\rho)}$.
	\end{definition}
	
	The following observation may be helpful for understanding the motivation or intuition of the operator $\mathfrak{n}(\phi)$.
	
	\begin{fact} \label{obs:secondmultiplication}
		Let $h_1, h_2 \in \mathcal{H}$, let $\psi_1, \psi_2 \in L^2(\Omega,\rho)$, and let $\psi_j h_j \in L^2(\Omega,\rho;\mathcal{H})$ be the map $\omega \mapsto \psi_j(\omega) h_j$.  Then
		\[
		\ell(\psi_1 h_1)^* \ell(\psi_2 h_2) = \ip{h_1, h_2}_{\mathcal{H}} \mathfrak{n}(\overline{\psi}_1 \psi_2).
		\]
		This is proved by directly computing the effect of these operators on some $f \in L^2(\Omega^{\times k},\rho_k^{\mathrm{w}};\mathcal{H}^{\otimes k})$.
	\end{fact}
	
	\subsection{Combinatorial formula for operators on a Fock space}
	
	Our next goal is to derive a combinatorial formula for the ``joint moment''
	\[
	\ip{\xi, T_n T_{n-1} \dots T_1 \xi},
	\]
	where $T_1$, \dots, $T_n$ are creation, annihilation, or multiplication operators on the Fock space $\mathcal{F}(\Omega,\rho,\mathrm{w},\mathcal{H})$.  For the sake of induction, we will find a combinatorial expression for the vector $T_n T_{n-1} \dots T_1 \xi$ itself.
	
	{\bf Setup:}  Let $T_1$, \dots, $T_n$ be operators on $\mathcal{F}(\Omega,\rho,w,\mathcal{H})$ such that each $T_j$ is one of the following types:
	\begin{itemize}
		\item $T_j = \ell(h_j)$ for some $h_j \in L^2(\Omega,\rho,\mathcal{H})$,
		\item $T_j = \ell(h_j)^*$ for some $h_j \in L^2(\Omega,\rho,\mathcal{H})$.
		\item $T_j = \mathfrak{m}(S_j)$ for some $S_j \in L^\infty(\Omega,\rho;B(\mathcal{H}))$.
		\item $T_j = \mathfrak{n}(\phi_j)$ for some $\phi_j \in L^1(\Omega,\rho)$.
	\end{itemize}
	Let $k(j)$ be the number of creation operators among $\{T_1,\dots,T_j\}$ minus the number of annihilation operators among $\{T_1,\dots,T_j\}$.
	
	\begin{fact}
		Because a creation operator maps $L^2(\Omega^{\times k},\rho_k^{\mathrm{w}};\mathcal{H}^{\otimes k})$ into $L^2(\Omega^{\times (k+1)},\rho_{k+1}^{\mathrm{w}};\mathcal{H}^{\otimes (k+1)})$ while an annihilation operator does the opposite, one can verify by induction on $n$ that $T_n \dots T_1 \xi \in L^2(\Omega^{\times k(n)},\rho_{k(n)}^{\mathrm{w}};\mathcal{H}^{\otimes k(n)})$ if $k(j) \geq 0$ for all $j$.  Moreover, $T_n \dots T_1 \xi = 0$ if $k(j)$ is ever negative.
	\end{fact}
	
	Now assume that $k(j) \geq 0$ for all $j$.  For each $j = 1$, \dots, $n$, we define $m(j)$ as follows:
	\begin{itemize}
		\item If $T_j = \mathfrak{n}(\phi_j)$ or $T_j = \ell(h_j)$, then set $m(j) = j$.
		\item If $T_j = \ell(h_j)^*$ or $T_j = \mathfrak{m}(S_j)$, let $m(j)$ be the greatest index $m$ such that $k(m-1) < k(j-1)$.
	\end{itemize}
	In the second case, note that $k(j-1) \geq k(j)$.  By definition, $k(i) \geq k(j-1)$ for all $i$ between $m(j)$ and $j$.  Moreover, since $|k(i+1) - k(i)| \leq 1$, we deduce that $k(m-1) = k(m) - 1 = k(j-1) - 1$, and thus $T_{m(j)}$ is a creation operator.
	
	\begin{remark}
		The intuition behind the choice of $m(j)$ is the following:  Each creation operator ``creates'' a new particle that is tensored onto the left of the vector it acts on, while each annihilation operator ``annihilates'' a particle.  The multiplication operators neither create nor annihilate anything.  The number $k(j)$ represents the current tensor degree, or the number of particles that exist at time $j$ (after the application of $T_j$).  If $T_j$ is $\ell(h_j)^*$ or $\mathfrak{m}(S_j)$, then $T_{m(j)}$ is the creation operator that created the newest particle that still exists, the one that $T_j$ is acting on.  In the case where $T_j = \mathfrak{n}(\phi_j)$, in light of Fact \ref{obs:secondmultiplication}, we can imagine that $T_j$ creates and immediately annihilates some ephemeral particle.  With respect to this picture, the next construction will be to group together all the operators that act on ``the same particle''; the indices of these operators will form the block of a non-crossing partition: non-crossing because the operators can only act on the newest existing particle, and thus this particle must be annihilated before any operator can act on an older particle.
	\end{remark}
	
	Let $\pi$ be the partition of $[n]$ given by $i \sim_\pi j$ if and only if $m(i) = m(j)$.  Observe that
	\begin{itemize}
		\item Each creation operator satisfies $m(j) = j$ and hence is the first element of its block.
		\item If $T_j$ is an annihilation operator, then it is the last element of its block.  This is because $k(j) < k(j-1)$ and this prevents any later index $i > j$ from having $m(i) = m(j)$.
		\item If $T_j = \mathfrak{n}(\phi_j)$, then $\{j\}$ is a singleton block of $\pi$.
	\end{itemize}
	
	We call a block \emph{finished} if either has a single $\mathfrak{n}$ operator or has both a creation and annihilation operator.  Otherwise, we call a block \emph{unfinished}.  The unfinished blocks will have a creation operator but no annihilation operator.
	
	\begin{fact} \label{obs: CAM partition}
		For every sequence of creation, annihilation, and multiplication operators with $k(j) \geq 0$ for all $j$, the associated partition $\pi$ constructed above is non-crossing.
	\end{fact}
	
	\begin{proof}
		Suppose for contradiction that $i < j < i' < j'$ with $i \sim_\pi i'$ and $j \sim_\pi j'$ for $i \not \sim_\pi j$.  By the preceding discussion, since $i \sim_\pi i'$ and $i < j < i'$, we must have $k(j) \geq k(i) = k(i'-1)$.  Note that $T_{j'}$ must be an annihilation operator or $\mathfrak{m}$ operator, and $k(j' - 1) = k(j)$.  If $k(j) > k(i)$, then $k(i'-1) < k(j'-1)$, which would imply that $m(j') \geq i'$ by definition of $m$, but this contradicts the fact that $m(j') = m(j) \leq j < i'$.  On the other hand, suppose $k(j) = k(i) = k(i' - 1)$.  Since $j \sim_\pi j'$, $T_j$ cannot be an annihilation operator, so $k(j-1) \leq k(j)$, but also $k(j-1) \geq k(i) = k(j)$, since $j - 1$ is between $i$ and $i'$, hence $k(j-1) = k(j)$.  Recall that $m(i') = m(i) \leq i$ is the last index before $i'$ where $k(m-1) < k(i' - 1) = k(j)$.  Since $k(t) \geq k(m(i))$ for all $t$ between $m(i) = m(i')$ and $i'$, we deduce that $m(i')$ is also the last index before $j$ where $k(m - 1) < k(j-1) = k(j)$, which implies that $m(j) = m(i)$, which contradicts the assumption that $i$ and $j$ are in different blocks of $\pi$.
	\end{proof}
	
	With the notation above, for each unfinished block $B = \{i_1,\dots,i_{|B|}\}$, let $h_B \in L^2(\Omega,\rho)$ be given by
	\[
	h_B(\omega) = S_{i_{|B|}}(\omega) \dots S_{i_2}(\omega) h_{i_1}(\omega).
	\]
	For each finished block $B = \{i_1,\dots,i_{|B|}\}$, let
	\[
	\phi_B(\omega) = \begin{cases} \phi_{i_1}, & |B| = 1, \\ \ip{h_{i_{|B|}}(\omega), S_{i_{|B|-1}}(\omega) \dots S_{i_2}(\omega) h_{i_1}(\omega)}_{\mathcal{H}}. \end{cases}
	\]
	
	\begin{proposition} \label{prop: CAM moment}
		Let $T_n \dots T_1$ be a sequence of creation, annihilation, $\mathfrak{m}$, and $\mathfrak{n}$ operators as above, such that $k(j) \geq 0$ for all $j$.  Let $\pred(B)$ denote the predecessor of $B$ in $\pi$.  Let $B_1$, \dots, $B_s$ be the unfinished blocks of $\pi$ listed so that $\min B_1 < \dots < \min B_t$.  Then $t = k(n)$ and $T_n T_{n-1} \dots T_1 \xi \in L^2(\Omega^{\times t}, \rho_t; \mathcal{H}^{\otimes t})$ and
		\begin{multline} \label{eq:CAMformula}
			[T_n T_{n-1} \dots T_1 \xi](\omega_{B_t},\dots,\omega_{B_1}) \\
			= 
			\int_{\Omega^{\times \{B \text{ finished} \}}} \prod_{\substack{B \text{ finished} \\ \depth(B) > 1}} \mathrm{w}(\omega_{\pred(B)},\omega_B)
			\prod_{B \text{ finished}} [\phi_B(\omega_B) \,d\rho(\omega_B)] \\
			h_{B_t}(\omega_{B_t}) \otimes \dots \otimes h_{B_1}(\omega_{B_1})
		\end{multline}
		In the case $t = 0$, we interpret $h_{B_t}(\omega_{B_t}) \otimes \dots \otimes h_{B_1}(\omega_{B_1})$ as $\xi$.
	\end{proposition}
	
	\begin{proof}
		We proceed by induction on $n$.  The base case is $n = 0$, for which both sides reduce to $\xi$.
		
		For the induction step, suppose the claim holds for $T_n T_{n-1} \dots T_1$, and we will prove it for $T_{n+1} T_n \dots T_1$.  Let $\pi'$ be the partition associated to $T_{n+1}$, \dots, $T_1$.  By restricting $\pi'$ to $\{n,\dots,1\}$, we obtain the partition $\pi$ associated to $T_n$, \dots, $T_1$.
		\begin{itemize}
			\item If $T_{n+1} = \ell(h_{n+1})$, then $m(n+1) = n+1$ and $\{n+1\}$ is an unfinished block in $\pi'$, so $\pi' = \pi \cup \{\{n+1\}\}$.  If $B_1$, \dots, $B_t$ are the unfinished blocks in $\pi$, then the unfinished blocks of $\pi'$ will be $B_1$, \dots, $B_t$ and $B_{t+1} := \{n+1\}$.  Moreover, $h_{B_{t+1}} = h_{n+1}$.  Therefore, the right-hand side of \eqref{eq:CAMformula} for $\pi'$ will be the same as the right-hand side of \eqref{eq:CAMformula} for $\pi$ except with $h_{B_{t+1}}(\omega_{B_{t+1}})$ tensored onto the front of $h_{B_t}(\omega_{B_t}) \otimes \dots \otimes h_{B_1}(\omega_{B_1})$.  Meanwhile,
			\[
			T_{n+1} [T_n \dots T_1 \xi](\omega_{B_{t+1}},\dots, \omega_{B_1}) = h_{n+1}(\omega_{B_{t+1}}) \otimes [T_n \dots T_1](\omega_{B_t},\dots,\omega_{B_1}),
			\]
			and hence \eqref{eq:CAMformula} will be true for $T_{n+1}$, \dots, $T_1$.
			\item Suppose that $T_{n+1} = \ell(h_{n+1})^*$.  Then $n+1$ will be the last element of a finished block $B'$ in $\pi'$ and $B = B' \setminus \{n+1\}$ will be an unfinished block in $\pi$.  Write $B = \{i_1,\dots,i_s\}$.  Then
			\[
			\phi_{B'}(\omega_{B'}) = \ip{h_{n+1}(\omega_{B'}), S_{i_s}(\omega_{B'}) \dots S_{i_2}(\omega_{B'}) h_{i_1}(\omega_{B'})} = \ip{h_{n+1}(\omega_{B'}), h_B(\omega_{B'})}.
			\]
			where $h_B$ is the vector corresponding to $B$ as an unfinished block of $\pi$.  Thus, to obtain the right-hand side of \eqref{eq:CAMformula} for $\pi'$ from the right-hand side of \eqref{eq:CAMformula} for $\pi$, one removes the term $h_B(\omega_B)$ from the unfinished blocks and adds the term $\phi_{B'}(\omega_{B'}) \,d\rho(\omega_{B'})$ to the finished blocks along with $\mathrm{w}(\omega_{B'}, \omega_{\pred(B')})$ if $\depth(B') > 1$.  Meanwhile, looking at the left-hand side of \eqref{eq:CAMformula} the application of $\ell(h_{n+1})$ to $T_n \dots T_1 \xi$ will precisely pair $h_{n+1}(\omega_{B'})$ with $h_B(\omega_{B'})$ in $\mathcal{H}$, multiply by $\mathrm{w}(\omega_{\pred(B')},\omega_{B'})$ if $\depth(B') > 1$, and then integrate $d\rho(\omega_{B'})$.  Hence, the left- and right-hand sides of \eqref{eq:CAMformula} agree for $\pi'$.
			\item Suppose that $T_{n+1} = \mathfrak{n}(\phi_{n+1})$.  Then $B':= \{n+1\}$ is a new finished block in $\pi'$ and $\phi_{B'} = \phi_{n+1}$.  The right-hand side of \eqref{eq:CAMformula} for $\pi'$ differs from that for $\pi$ by adding a new term $\phi_{B'}(\omega_{B'}) d\rho(\omega_{B'})$ to the finished blocks, along with $\mathrm{w}(\omega_{\pred(B')},\omega_{B'})$ if $\depth(B') > 1$.  This agrees with what happens when we apply the operator $\ell(h_{n+1})^*$ to $T_n \dots T_1 \xi$.
			\item Finally, suppose that $T_{n+1} = \mathfrak{m}(S_{n+1})$.  Then $n+1$ is an element of some unfinished block $B'$ of $\pi'$ such that $B = B' \setminus \{n+1\}$ is also unfinished in $\pi$.  Write $B = \{i_1,\dots,i_s\}$.  The right-hand side of \eqref{eq:CAMformula} differs for $\pi'$ and $\pi$ by the replacement of $h_B(\omega_B)$ with
			\[
			h_{B'}(\omega_{B'}) = S_{n+1}(\omega_{B'}) S_{i_s}(\omega_{B'}) \dots S_{i_2}(\omega_{B'}) h_{i_1}(\omega_{B'}) = S_{n+1}(\omega_{B'}) h_{B}(\omega_{B'}).
			\]
			This agrees with the application of the operator $\mathfrak{m}(S_{n+1})$ to $T_n \dots T_1 \xi$.
		\end{itemize}
		This completes the induction step and hence the proof.
	\end{proof}
	
	\subsection{Fock space operators as limits of independent sums}
	
	Now we adapt Proposition \ref{prop: CAM moment} to the case of a sum of a creation, annihilation, and multiplication operators that will model limit distributions arising in applications of Theorem \ref{thm:limit}.
	
	\begin{proposition} \label{prop: CAM moment 2}
		Let a Hilbert space $\mathcal{H}$ with unit vector $\xi$ be given, and let $X \in B(\mathcal{H})$ be self-adjoint.  Let $\mathcal{H}^\circ = \mathcal{H} \ominus \C \xi$, and write $X$ in block form based on the decomposition $\mathcal{H} = \C \xi \oplus \mathcal{H}^\circ$ as
		\begin{equation} \label{eq: X decomposition}
			X = \begin{bmatrix}
				\alpha & h^* \\
				h & S 
			\end{bmatrix}, \text{ where } \alpha \in \C, h \in \mathcal{H}^\circ, S \in B(\mathcal{H}^\circ).
		\end{equation}
		Fix a measure space $(\Omega,\rho)$ and nonnegative $\mathrm{w} \in L^\infty(\Omega \times \Omega)$, and let $\mathcal{F} = \mathcal{F}(\Omega,\rho,\mathrm{w},\mathcal{H}^\circ)$ be the associated Fock space.  Let $A \subseteq \Omega$ with finite measure. 
		Then define $\widehat{X} \in B(\mathcal{F})$ by
		\[
		\widehat{X} = \alpha \mathfrak{n}(\mathbbm{1}_A) + \ell(\mathbbm{1}_A h) + \ell(\mathbbm{1}_A h)^* + \mathfrak{m}(\mathbbm{1}_A S),
		\]
		where we view $\mathbbm{1}_A h \in L^2(\Omega,\rho;\mathcal{H})$ and $\mathbbm{1}_A S \in L^\infty(\Omega,B(\mathcal{H}^\circ)$.  Let $\mu$ and $\widehat{\mu}$ be the spectral distributions of $X$ and $\widehat{X}$ respectively with respect to the appropriate state vectors.  Then
		\[
		m_k(\widehat{\mu}) = \sum_{\pi \in \mathcal{NC}_k} \kappa_{\bool,\pi}(\mu) \int_{A^{\times k}} \prod_{\substack{B \in \pi \\ \depth(B) > 1}} \mathrm{w}(\omega_B,\omega_{\pred(B)}) \prod_{B \in \pi} d\rho(\omega_B).
		\]
	\end{proposition}
	
	\begin{proof}
		To compute $m_k(\widehat{\mu}) = \ip{\xi, \widehat{X}^k \xi}_{\mathcal{F}}$, we expand $\widehat{X}^k = (\alpha \mathfrak{n}(\mathbbm{1}_A) + \ell(\mathbbm{1}_A h) + \ell(\mathbbm{1}_A h)^* + \mathfrak{m}(\mathbbm{1}_A S))^k$ by multilinearity into the sum of $\ip{\xi,T_k \dots T_1 \xi}_{\mathcal{F}}$, where
		\[
		T_j \in \{ \alpha \mathfrak{n}(\mathbbm{1}_A),\ell(\mathbbm{1}_A h), \ell(\mathbbm{1}_A h)^*,\mathfrak{m}(\mathbbm{1}_A S)\}.
		\]
		Then we apply Proposition \ref{prop: CAM moment} to each term.  Each sequence of creation, annihilation, and multiplication operators such that $k(j) \geq 0$ has an associated non-crossing partition as in Fact \ref{obs: CAM partition}.  If the partition has unfinished blocks, then $T_k \dots T_1 \xi$ is orthogonal to $\xi$ in $\mathcal{F}$ and hence $\ip{\xi, T_k \dots T_1 \xi}$ vanishes.  We are thus left with the terms where the partition does not have any unfinsished blocks.  In this case, similar to Lemma \ref{lem: bijection}, the partition $\pi$ uniquely determines the sequences of creation, annihilation, and multiplication operators by the rule that for singleton blocks $T_j = \alpha \mathfrak{n}(\mathbbm{1}_A)$, and for all other blocks, the leftmost element is the annihilation operator, the rightmost element is the creation operator, and the remaining terms are $\mathfrak{m}(\mathbbm{1}_A S)$.  Therefore, we obtain
		\begin{equation} \label{eq: CAM partition decomposition}
			\ip{\xi, \widehat{X}^k \xi}_{\mathcal{F}} = \sum_{\pi \in \mathcal{NC}_k}
			\int_{\Omega^{\times \pi}} \prod_{\substack{B \in \pi \\ \depth(B) > 1}} \mathrm{w}(\omega_B,\omega_{\pred(B)})
			\prod_{B \in \pi} [\phi_B(\omega_B) \,d\rho(\omega_B)],
		\end{equation}
		where
		\[
		\phi_B(\omega) = \begin{cases}
			\alpha \mathbbm{1}_A(\omega), & |B| = 1 \\
			\mathbbm{1}_A(\omega) \ip{h,S^{|B|-2}h}_{\mathcal{H}^\circ}, & \text{else.}
		\end{cases}
		\]
		By Lemma \ref{lem: Boolean cumulants} and equation \eqref{eq: X decomposition}, we have $\phi_B(\omega) = \mathbbm{1}_A(\omega) \kappa_{\bool,|B|}(\mu)$.  Thus, \eqref{eq: CAM partition decomposition} becomes
		\[
		\ip{\xi, \widehat{X}^k \xi}_{\mathcal{F}} = \sum_{\pi \in \mathcal{NC}_k}
		\prod_{B \in \pi} \kappa_{\bool,|B|}(\mu) \int_{\Omega^{\times \pi}} \prod_{\substack{B \in \pi \\ \depth(B) > 1}} w(\omega_B,\omega_{\pred(B)})
		\prod_{B \in \pi} d\rho(\omega_B),
		\]
		which is the desired formula.
	\end{proof}
	
	We now relate this back to the ideas of \S \ref{subsec: continuum limit method} by specializing to the case where $(\Omega,\rho)$ is a probability space and $w = \mathbbm{1}_{\mathcal{E}}$ for some measurable $\mathcal{E} \subseteq \Omega \times \Omega$.
	
	\begin{corollary} \label{cor: Fock space gives limit measure}
		Consider the same setup as Proposition \ref{prop: CAM moment 2}, and assume that $(\Omega,\rho)$ is a probability measure space, $\mathcal{E} \subseteq \Omega \times \Omega$ is measurable, and $w = \mathbbm{1}_\mathcal{E}$.  Let $\widehat{\mu}$ be the distribution of the operator $\widehat{X}$ in that proposition.  Then we have
		\begin{equation} \label{eq: continuum limit Fock moments}
			m_k(\widehat{\mu}) = \sum_{\pi \in NC(k)} \rho^{\times \pi}(\Hom(\mathrm{F}(\pi),(\Omega,\mathcal{E}))) \kappa_{\bool,\pi}(\mu).
		\end{equation}
		Now as in Proposition \ref{prop: continuum limit}, let $G_n = (V_n,E_n)$ be a finite digraph for each $n \in \N$; let $(A_{n,v})_{v \in V_n}$ be a measurable partition of $\Omega$ into sets of measure $1/|V_n|$, and let $\tilde{\mathcal{E}}_n = \bigcup_{(v,w) \in E_n} A_{n,v} \times A_{n,w}$; suppose that $\rho^{\times 2}(\tilde{\mathcal{E}}_n \, \Delta \, \mathcal{E}) \to 0$.  If $\mu_n \in \mathcal{P}(\R)$ such that $\mu_n^{\uplus |V_n|} \to \mu$, then $\boxplus_{G_n}(\mu_n) \to \widehat{\mu}$.
	\end{corollary}
	
	\begin{proof}
		In Proposition \ref{prop: CAM moment 2}, we take $A = \Omega$, and note that
		\[
		\prod_{\substack{B \in \pi \\ \depth(B) > 1}} w(\omega_{\pred(B)}, \omega_B) = \prod_{\substack{B \in \pi \\ \depth(B) > 1}} \mathbbm{1}_{\mathcal{E}}(\omega_{\pred(B)}, \omega_B) = \mathbbm{1}_{\Hom(\mathrm{F}(\pi),(\Omega,E))}(\omega),
		\]
		and hence
		\[
		\int_{\Omega^{\times k}} \prod_{\substack{B \in \pi \\ \depth(B) > 1}} \mathrm{w}(\omega_{\pred(B)}, \omega_B) \prod_{B \in \pi} d\rho(\omega_B) = \rho^{\times \pi}(\Hom(\mathrm{F}(\pi),(\Omega,E))),
		\]
		so that
		\[
		m_k(\widehat{\mu}) = \sum_{\pi \in NC(k)} \rho^{\times \pi}(\Hom(\mathrm{F}(\pi),(\Omega,\mathcal{E}))) \kappa_{\bool,\pi}(\mu).
		\]
		Now in the setting of Proposition \ref{prop: continuum limit}, we have for finite out-forests $G'$ that
		\[
		\beta_{G'} = \lim_{n \to \infty} \frac{|\Hom(G',G_n)|}{|V_n|^{|V'|}} = \rho^{\times V'}(\Hom(G',(\Omega,\mathcal{E}))).
		\]
		Now from Lemma \ref{lem:cpctsuppconvergence}, it follows that
		\[
		\boxplus_{G_n}(\mu^{\uplus 1/|V_n|}) \to \widehat{\mu}
		\]
		since the $k$th moments converge.  From \S \ref{subsec: limit general}, it is clear that if $\mu_n^{\uplus |V_n|} \to \mu$, then
		\[
		\lim_{n \to \infty} \boxplus_{G_n}(\mu_n) = \lim_{n \to \infty} \boxplus_{G_n}(\mu^{\uplus 1/|V_n|}) = \widehat{\mu}.  \qedhere
		\]
	\end{proof}
	
	\begin{example}[Fock space for iterated compositions] \label{ex: Fock for iterated composition}
		Recall that in \S \ref{subsec: iterated composition}, we applied Proposition \ref{prop: continuum limit} to iterated compositions $G^{\circ k}$ of a digraph $G$ on vertex set $n$ as $k \to \infty$.  In particular, we constructed a digraph on $(\Omega,\mathcal{E})$ with $\Omega = [n]^{\times \N}$ and $\rho$ the infinite product of the uniform probability measure.  Hence, by Corollary \ref{cor: Fock space gives limit measure}, we see that for a compactly supported $\mu$, the measure $\widehat{\mu}$ is modeled by the Fock space associated to $\Omega$, $\mathcal{E}$, $\rho$, and an appropriate Hilbert space $\mathcal{H}^\circ$.  These are a special case of the Fock spaces in \cite[\S 9]{JekelLiu2020} (see \cite[Example 9.16]{JekelLiu2020}).  This also includes the case of free, Boolean, and monotone Fock spaces (see \cite[\S 9.6]{JekelLiu2020}).
	\end{example}
	
	\begin{example}[BM Fock space and Brownian motion] \label{ex: BM Fock space}
		Fock spaces for BM independence associated to symmetric cones have been studied in \cite{AKJW2010}.  As in \S \ref{subsec: BM limit theorem}, let $\Pi \subseteq \R^d$ be a closed convex salient cone.  Consider the measure space $\Omega = \Pi$ with $\rho$ equal to the Lebesgue measure on the linear span of $\Pi$.  For $\eta \in \Pi$, let
		\[
		\widehat{X}_\eta = \ell(\mathbbm{1}_{[0,\eta)}) + \ell(\mathbbm{1}_{[0,\eta)})^*.
		\]
		Then $\widehat{X}_\eta$ is the BM Brownian motion.  Note that if $\eta_1 \preceq \eta_2$, then $\ell(\mathbbm{1}_{[0,\eta_1)}) + \ell(\mathbbm{1}_{[0,\eta_1)})^*$ and $\ell(\mathbbm{1}_{[\eta_1,\eta_2)}) + \ell(\mathbbm{1}_{[\eta_1,\eta_2)})^*$ are monotone independent.\footnote{In this setting, it does not matter whether we include the endpoints of the interval or not, since the boundary of interval has measure zero.}  Similarly, if two intervals are elementwise incomparable, then the associated variables are Boolean independent.
	\end{example}
	
	\begin{example}[Multiregular digraphs] \label{ex: multiregular Fock space}
		As in \S \ref{subsec: multiregular}, we consider multiregular digraphs $G_n = (V_n,E_n)$ with $V_n = \bigsqcup_{j=1}^m V_{n,j}$ so that $|\{w \in V_{n,j}: v \rightsquigarrow w\}| = A_{n,i,j}$ for $v \in V_{n,i}$.  Assume again that
		\[
		\lim_{n \to \infty} \frac{|V_{n,i}|}{|V_n|} = t_i, \qquad \lim_{n \to \infty} \frac{A_{n,i,j}}{|V_n|} = a_{i,j},
		\]
		and recall by Proposition \ref{prop: multiregular 1} that for an out-tree $G'$,
		\[
		\lim_{n \to \infty} \frac{|\Hom(G',G_n)|}{|V_n|^{|V'|}} = \sum_{\ell: V' \to [m]} t_{\ell(r)} \prod_{v \in V' \setminus \{r\}} a_{\ell(v_-),\ell(v)} =: \beta_{G'}.
		\]
		These coefficients can be realized with a tuple $(\Omega,\rho,w)$ as follows.  Let
		\[
		\Omega = [m], \qquad \rho = \sum_{i=1}^m t_i \delta_i, \qquad w(i,j) = \frac{a_{i,j}}{t_j}.
		\]
		Then a direct computation shows that for an out-tree $G' = (V',E')$,
		\[
		\int_{\Omega^{V'}} \prod_{v \in V' \setminus \{r\}} \mathrm{w}(\omega_{v_-},\omega_v) \,d\rho^{\times V'}(\omega) = \sum_{\ell: V' \to [m]} t_{\ell(r)} \prod_{v \in V' \setminus \{r\}} a_{\ell(v_-),\ell(v)}.
		\]
		Therefore, the construction in Proposition \ref{prop: CAM moment 2} with this choice of $(\Omega,\rho,\mathrm{w})$ will realize the moments of limit distributions from Proposition \ref{prop: multiregular 1} in the case of compact support.
	\end{example}

	
	
	
	
	
	\providecommand{\bysame}{\leavevmode\hbox to3em{\hrulefill}\thinspace}
	\providecommand{\MR}{\relax\ifhmode\unskip\space\fi MR }
	\providecommand{\MRhref}[2]{%
		\href{http://www.ams.org/mathscinet-getitem?mr=#1}{#2}
	}
	\providecommand{\href}[2]{#2}

	
	\textbf{Acknowledgements.} We thank the organizers and hosts of the \textit{International Workshop on Operator Theory and its Applications} (Lisbon, Portugal  2019) and of the \textit{19th Workshop: Noncommutative probability, noncommutative harmonic analysis and related topics, with applications}, (Bedlewo, Poland 2022) which facilitated collaboration among the authors.  We also thank the Mathematische Forschungsinstitut Oberwolfach and the University of Wroc{\l}aw for funding DJ's visit to Wroc{\l}aw in May 2024.  DJ thanks Weihua Liu, Ethan Davis, and Zhichao Wang for past collaboration on related projects.  We thank the referees for careful reading and comments which improved the manuscript.
	
	DJ acknowledges funding from the Discovery grant ``Logic and $\mathrm{C}^*$-algebras'' from the Natural Sciences and Engineering Research Council of Canada; the Danish Independent Research Fund, grant 1026-00371B; and a Horizon Europe Marie Sk{\l}odowska Curie Action, FREEINFOGEOM, grant 101209517.

	
\end{document}